\documentclass[11pt]{amsart}
\usepackage{amsmath}
\usepackage{amsfonts}
\usepackage{amsthm}

\theoremstyle{plain} \numberwithin{equation}{section}
\newtheorem{Theorem}{Theorem}
\newtheorem{Lemma}[Theorem]{Lemma}
\newtheorem{Proposition}[Theorem]{Proposition}
\newtheorem{Corollary}[Theorem]{Corollary}

\newtheorem{Definition}[Theorem]{Definition}

\theoremstyle{remark}

\date{}

\title[Equiconvergence]
{Equiconvergence of spectral decompositions of 1D Dirac
operators with regular boundary conditions}

\author[P. Djakov]{Plamen Djakov}\thanks{P. Djakov acknowledges
the hospitality  of Department of Mathematics 
and the support of Mathematical Research Institute of 
The Ohio State University, July - August 2011.}

\author[B. Mityagin]{Boris Mityagin}
\thanks{B. Mityagin acknowledges the support of the Scientific and
Technological Research Council of Turkey and the hospitality of
Sabanci University, April--June, 2011.}

\begin{document}

\address{Sabanci University, Orhanli,
34956 Tuzla, Istanbul, Turkey}
 \email{djakov@sabanciuniv.edu}
\address{Department of Mathematics,
The Ohio State University,
 231 West 18th Ave,
Columbus, OH 43210, USA}
\email{mityagin.1@osu.edu}

\begin{abstract}
One dimensional  Dirac operators $$ L_{bc}(v) \, y = i
\begin{pmatrix} 1 & 0 \\ 0 & -1
\end{pmatrix}
\frac{dy}{dx}  + v(x) y, \quad y = \begin{pmatrix} y_1\\y_2
\end{pmatrix},  \quad
x\in[0,\pi],$$ considered with  $L^2$-potentials $ v(x) =
\begin{pmatrix} 0 & P(x) \\ Q(x) & 0 \end{pmatrix} $ and subject to
regular boundary conditions ($bc$), have discrete spectrum. For
strictly regular $bc,$  the spectrum of the free
operator $ L_{bc}(0) $ is simple while the spectrum of $ L_{bc}(v) $
is eventually simple, and the corresponding normalized root function
systems are Riesz bases.
For expansions of functions of bounded variation about these Riesz bases,
we prove the uniform equiconvergence property
and point-wise convergence on the closed interval $[0,\pi].$
Analogous results are obtained for regular but not strictly regular $bc.$
\vspace{1mm} \vspace{1mm} 

\end{abstract}
\maketitle

\section*{Content}
\begin{enumerate}

\item[Section 1.]  Introduction \vspace{2mm}

\item[Section 2.]  Preliminaries \vspace{2mm}

\item[Section 3.] Estimates for the resolvent
of $L_{bc} $ and localization of spectra  \vspace{2mm}

\item[Section 4.]  Equiconvergence \vspace{2mm}

\item[Section 5.]  Point-wise convergence of spectral decompositions \vspace{2mm}

\item[Section 6.]  Generalizations and comments \vspace{2mm}

\item[Section 7.]  Self-adjoint separated $bc$ \vspace{2mm}

\item[Section 8.]   Appendix: Discrete Hilbert transform and multipliers
\vspace{2mm}

\item[]       References

\end{enumerate}
\bigskip

\section{Introduction}

  Spectral theory of non-self-adjoint boundary value problems ($BVP$) for
ordinary  differential equations on a finite interval $I$ goes back
to the classical works of Birkhoff \cite{Bir1,Bir2} and Tamarkin
\cite{Tam1,Tam2,Tam3}. They introduced a concept of regular ($R$)
 boundary conditions ($bc$) and investigated asymptotic behavior
 of eigenvalues and eigenfunctions of such problems. Moreover, they
proved that the system of eigenfunctions and associated functions
($SEAF$) of a regular $BVP$ is complete.
Detailed presentation of this topic could be found in \cite{Na69}.

   More subtle is the question whether $SEAF$ is a basis or an
unconditional basis in the Hilbert space $H^0 = L^2(I)$. N. Dunford
\cite{Du58} (see also \cite{DS71}),
V. P. Mikhailov \cite{Mi62}, G. M.  Keselman
\cite{Ke64} independently proved that the $SEAF$ is an
unconditional, or Riesz, basis if $bc$ are strictly regular ($SR$).
This property is lost if $bc$ are $R \setminus SR$, i.e., regular
but not strictly regular; unfortunately, this is just the case of
periodic ($Per^+$) and anti-periodic ($Per^-$) $bc.$  But A.  A.
Shkalikov \cite{Sh79, Sh82, Sh83} proved that in $R \setminus SR$
cases a proper chosen finite-dimensional projections form a Riesz
basis of projections.

Dirac operators
\begin{equation}
\label{i1}
Ly = i \begin{pmatrix}    1  & 0 \\  0 & -1
\end{pmatrix}
\frac{dY}{dx} + v(x) Y, \quad Y =
 \begin{pmatrix}    y_1 \\  y_2
\end{pmatrix}, \quad   v(x)=\begin{pmatrix}  0& P(x) \\ Q(x) & 0
\end{pmatrix}
\end{equation}
with $P, Q \in L^2 (I),$ and more general operators
\begin{equation}
\label{i2}
My = i B \frac{dY}{dx} + v(x) Y, \quad Y =(y_j (x))_1^d,
\end{equation}
where $B$ is a $d \times d $-matrix and $v(x) $ is
 a $d \times d$ matrix-valued $L^2 (I)$ function,
 bring new difficulties. One of them comes from the fact that the values
of the resolvent  $(\lambda - L_{bc})^{-1}$ are not trace class operators.

For general system (\ref{i2})  M. M. Malamud and L. L.
Oridoroga \cite{MO,MO10} proved
completeness  of SEAF for a wide class of BVP which includes regular
(in the sense of \cite{BL23}) BVP's.

The Riesz basis property for $2 \times 2 $  Dirac
operators (\ref{i1})  was proved by I. Trooshin and M. Yamamoto
\cite{TY01, TY02} in the case of separated $bc$  and $v \in L^2$.
S. Hassi and
L. L. Oridoroga \cite{HO09}
proved the Riesz basis property for (\ref{i2}) when $B =
\begin{pmatrix} a  & 0 \\  0  & -b   \end{pmatrix}, $
with $a,b >0,$  for separated  $bc$  and $v \in  C^1 (I).$

  B. Mityagin \cite{Mit03}, \cite[Theorem 8.8]{Mit04}  proved that
  {\em periodic (or
anti-periodic)} $bc$ give a rise of a Riesz system of 2D projections
(or 2D invariant subspaces) under the smoothness restriction $P, Q
\in H^\alpha, \; \alpha > 1/2,$ on the potentials $v$ in (\ref{i1}).
The authors removed that restriction in \cite{DM20}, where the same
result is obtained for {\em any} $L^2$ potential $v.$ This became possible
in the framework of the
 general approach
 to analysis of invariant (Riesz)
subspaces and their closeness to 2D subspaces of the free
operator developed and used by the authors in
\cite{DM3, DM5, DM6, DM7, DM15}.

Moreover, in \cite{DM23}
these results are extended to Dirac operators with {\em any
regular} $bc.$
Careful analysis of
  {\em regular} and {\em strictly regular}
 $bc$ and construction of Riesz bases or Riesz system of projections
 which is done in \cite{DM23} give us the background for treating
 questions on equiconvergence  and point-wise convergence  of spectral
 decompositions
(or "the development in characteristic functions of the system"
as G. Birkgoff and R. Langer \cite{BL23} would say).

These
questions for o.d.o. were raised by G. D. Birkhoff  \cite{Bir1,Bir2}
and J. Tamarkin \cite{Tam1,Tam2,Tam3} as well, or even earlier for
second order operators by V. A. Steklov, E. W. Hobson and A.A.Haar.
 A nice survey of further development of equiconvergence theory over
 the last 100 years
(we do not provide the names of authors -- any list would be incomplete
and unfair) is given by A. Minkin \cite{Min99}.
\bigskip

In this paper we analyze in detail one-dimensional Dirac operators
(\ref{i1});  we address the following questions:

(i) for  given $bc,$  does uniform convergence of the series
$$
f(x) = S_N f + \sum_{|k|>N}^\infty P_k f, \quad x \in [0,\pi],
$$
for an individual $f\in L^2 ([0,\pi], \mathbb{C}^2) $ depend on
the potential  $v?$

(ii) for good enough functions $f,$   say $f$ is of bounded
variation, do point-wise limits
$$
\lim_{m \to \infty} \left (S_N f (x) + \sum_{|k|>N}^m P_k  f (x)
\right ) =F(x)
$$
exist?  If {\em yes},  how to describe the limit function $F(x) $ in
terms of $f$ and $bc?$

The less rigid questions ask about uniform convergence on compact
subsets of $(0,\pi).$
In this case for any complete system $\{u_n (x) \} $
of eigenfunctions of the operator (\ref{i1})  with its biorthogonal
system $\{\psi_n \},$  let us define
$$
\sigma_m (x,f)= \sum_{n \leq m }  \langle f,\psi_n \rangle u_n
$$
and compare these partial sums with
$$
S_m (x,f) = \frac{1}{\pi} \int_0^\pi \frac{\sin (x-y)}{x-y} f(y) dy,
$$
 the proxy of partial sums of the standard Fourier series.

Many authors (V. A. Ilin \cite{I83, I91-1, I91-2}, Horvath
\cite{Ho95}) compare $\sigma_m$
 and $S_m$ on compacts in $(0, \pi).$
 For example,  in \cite{Ho95} it is shown , under the assumption
 that the system $\{u_j\}$ is a Riesz basis and
  $v \in L^p, \, p >2,$  that for any compact $K \subset (0,\pi) $
we have
$$
\lim_m \left ( \sup_{x \in K}  |\sigma_m (x,f) - S_m (x,f) |     \right )
=0  \quad  \forall  f \in L^2 ([0,\pi], \mathbb{C}^2). $$

We focus on questions on equiconvergence and point-wise convergence
on the {\em entire closed interval} $[0, \pi].$ The structure of
this paper is the following.

Section 2 reminds elementary facts on Riesz bases and Riesz systems
of projections in a Hilbert space, and gives (after \cite{DM23})
explicitly such
bases and systems of projections in the case of free Dirac
operators {\em subject to arbitrary regular} $bc,$
  with special attention on their dependence
on parameters of boundary conditions.

Any analysis of spectral decompositions requires accurate information
on localization of spectra $Sp (L_{bc}) $
and good estimates of the resolvent $(z-L_{bc})^{-1}$
outside the $Sp (L_{bc}). $  Such analysis is done in
\cite{DM23}, but in Section 3 we carry it in a different way in order to
obtain at the same time some basic preliminary inequalities
that play an essential role later.

Section 4 is the core of this paper. As usually, the deviation
$$ S_N -S_N^0    = \frac{1}{2 \pi i} \int_{\partial R_{NT}}
(R_\lambda -R^0_\lambda) d\lambda = A_N  + B_N,
$$
where
$$A_N =
\frac{1}{2 \pi i} \int_{\partial R_{NT}}
R^0_\lambda VR^0_\lambda d\lambda,
\quad
B_N  = \frac{1}{2 \pi i} \int_{\partial R_{NT}}
\sum_{m=2}^\infty
R^0_\lambda (VR^0_\lambda)^m
d\lambda.
$$
It happens that the estimates of the "nonlinear" component
$B_N $  (see Proposition \ref{propBN}) are a little bit simpler;
they reduce the problem of equiconvergence to questions on behavior
of the "linear" component $A_N (F)$ when $n \to \infty$
and its dependence on the smoothness of potentials $v$
or a vector-function $F.$  Proposition \ref{propAN}
and Lemma \ref{lemsob}  specify these smoothness conditions
and lead to our main result (Theorem \ref{EC}):

{\em For regular $bc, $  Dirac potentials $v = \begin{pmatrix}
0 & P\\ Q &0 \end{pmatrix} $ with $P, Q \in L^2 ([0,\pi])$ and $F=
\begin{pmatrix} F_1\\ F_2 \end{pmatrix} $ with $F_1, F_2  \in L^2([0,\pi],$
\begin{equation}
\label{i10} \left \| \left ( S_N - S^0_N \right ) F \right
\|_{\infty} \to 0 \quad \text{as} \quad N \to \infty
\end{equation}
whenever one of the following conditions is satisfied:

(a) $ \exists \beta > 1 $ such that $$ \sum_{k\in 2\mathbb{Z}}
(|F_{1,k}|^2 +|F_{2,k}|^2) (\log(e+|k|))^\beta< \infty, $$ where $
(F_{1,k})_{k\in 2\mathbb{Z}} $ and $(F_{2,k})_{k\in 2\mathbb{Z}}$
are, respectively, the Fourier coefficients of $F_1 $ and $F_2$ about
the system $\{e^{ikx},\, k\in 2\mathbb{Z}\};$

(b) $\exists \beta >1 $ such that $$ \sum_{k\in 2\mathbb{Z}}
(|p(k)|^2 +|q(k)|^2) (\log(e+|k|))^\beta < \infty, $$ where $ (p(k))_{k\in
2\mathbb{Z}} $ and $(q(k))_{k\in 2\mathbb{Z}}$ are, respectively, the
Fourier coefficients of $P $ and $Q$ about the system $\{e^{ikx},\,
k\in 2\mathbb{Z}\}.$

In particular, if  $F_1,\, F_2 $ are functions of bounded variation
or  $P, Q $ are functions of bounded variation, then (\ref{i10})
holds.}

This equiconvergence claim reduces (Section 5, Theorem \ref{GCT})
point-wise convergence problem to the case of free operator
where we can use explicit information on Riesz bases of root
functions and answer question (ii)  -- see Formulas (\ref{c03})
and (\ref{c3}).

In Section 6 we consider Dirac operators with more general
potential matrices $T=\begin{pmatrix}   T_{11}   &  T_{12} \\ T_{21}
& T_{22} \end{pmatrix}$
and weighted eigenvalue problems.
The results and formulas of Section 5 are properly adjusted
to this case.

Finally, in Section 7 we consider Examples (motivated by the  paper
of R. Szmytkowski \cite{SZ01})  with self-adjoint separated boundary
conditions -- see Theorems \ref{thms1} and \ref{SBC}.

Appendix (Section 8) gives a detailed proof of a technical lemma 
(Lemma 18) on on $C^1$-multipliers in the weighted sequence 
spaces.  Discrete Hilbert transform is an essential component of 
this proof. 

\section{Preliminaries}

1. Riesz bases

Let $H$ be a separable Hilbert space, and let $(e_\gamma, \, \gamma \in
\Gamma)$ be an o.n.b. in $H.$ If $A: H\to H$ is an automorphism,
then the system
\begin{equation}
\label{p2} f_\gamma = A e_\gamma, \quad  \gamma \in \Gamma,
\end{equation}
is an unconditional basis in $H.$ Indeed, for each $x\in H$ we have
$$ x= A(A^{-1} x )= A \left (\sum_\gamma \langle A^{-1} x,e_\gamma
\rangle e_\gamma \right)= \sum_\gamma \langle x,(A^{-1})^*e_\gamma
\rangle f_\gamma =\sum_\gamma \langle x,\tilde{f}_\gamma \rangle
f_\gamma,$$ i.e., $(f_\gamma)$ is a basis, and its biorthogonal
system is
\begin{equation}
\label{p2a} \tilde{f}_\gamma =(A^{-1})^* e_\gamma, \quad \gamma \in
\Gamma.
\end{equation}
Moreover, it follows that
\begin{equation} \label{p3} 0< c \leq
\|f_\gamma\| \leq  C, \quad m^2\|x\|^2 \leq \sum_\gamma |\langle
x,\tilde{f}_\gamma \rangle|^2 \|f_\gamma\|^2 \leq M^2 \|x\|^2,
\end{equation}
with $ c= 1/\|A^{-1}\|, \; C=\|A\|,  \; M= \|A\|\cdot \|A^{-1}\|$
and $m=1/M.$

A basis of the form (\ref{p2}) is called {\em Riesz basis.} One can
easily see that the property (\ref{p3}) characterizes Riesz bases,
i.e., a basis $(f_\gamma)$ is a Riesz bases if and only if
(\ref{p3}) holds with some constants $C\geq c>0$ and $M \geq m >0.$
Another characterization of Riesz bases gives the following
assertion (see \cite[Chapter 6, Section 5.3, Theorem 5.2]{GK}):  {\em If
$(f_\gamma)$ is a normalized  basis (i.e., $\|f_\gamma\|=1 \;
\forall \gamma $), then it is a Riesz basis if and only if it is
unconditional.}
\bigskip

2. We consider the Dirac operators $L=L(v)$ and $L^0 = L(0)$ given by
(\ref{i1}) on the interval $I=[0,\pi]. $
In the following, the space $L^2 (I, \mathbb{C}^2) $ is
regarded with the scalar product
\begin{equation}
\label{0} \left \langle \begin{pmatrix} f_1 \\ f_2
\end{pmatrix},\begin{pmatrix} g_1 \\ g_2  \end{pmatrix}  \right
\rangle =\frac{1}{\pi} \int_0^\pi \left (f_1 (x) \overline{g_1 (x)}
+ f_2 (x) \overline{g_2 (x)} \right ) dx.
\end{equation}

A general boundary condition ($bc$) for the operator  $L(v)$ is
given by a system of two linear equations
\begin{eqnarray}
\label{1} a_1 y_1 (0) +b_1 y_1 (\pi) + a_2 y_2 (0) + b_2 y_2
(\pi)=0  \\ \nonumber c_1 y_1 (0) +d_1 y_1 (\pi) + c_2 y_2 (0) +
d_2 y_2 (\pi)=0
\end{eqnarray}
Consider the corresponding operator $L_{bc}(v)$ in the domain $Dom
\, L_{bc}(v) $ which consists of all absolutely continuous $y $ such
that (\ref{1}) holds and $y^\prime_1, y^\prime_1 \in L^2 (I,
\mathbb{C}^2).$ It is easy to see that $L_{bc}(v)$ is a closed
densely defined operator.

Let  $A_{ij}$  denote the  $2\times 2$ matrix
 formed by the $i$-th and $j$-th columns of the
matrix
$ \left [
\begin{array}{cccc}
a_1 & b_1 & a_2 & b_2\\ c_1 & d_1& c_2 & d_2
\end{array}
\right ],
$
and let $|A_{ij}|$ denote the determinant of the matrix $A_{ij}.$
Each solution of the equation
$L^0 y =\lambda y$
has the form
$ y =\begin{pmatrix} \xi e^{-i\lambda x}\\ \eta
e^{i\lambda x} \end{pmatrix}.$
It satisfies the boundary condition (\ref{1}) if and only if $(\xi,
\eta)$ is a solution of the system of two linear equations
\begin{eqnarray}
\label{4} \xi (a_1 + b_1 z^{-1} ) + \eta (a_2 + b_2 z) =0  \\
\nonumber \xi (c_1 + d_1 z^{-1} ) + \eta (c_2 + d_2 z) =0
\end{eqnarray}
with $z= \exp (i\pi\lambda).$  Therefore, there is a non-zero solution $y$
 if and only if the determinant of (\ref{4}) is
zero, i.e.,
\begin{equation}
\label{5} |A_{14}| z^2 + (|A_{13}| + |A_{24}|) z + |A_{23}|=0.
\end{equation}
\begin{Definition}
The boundary condition (\ref{1}) is called: {\bf regular} if
\begin{equation}
\label{6} |A_{14}| \neq 0, \quad |A_{23}| \neq 0,
\end{equation}
and {\bf strictly regular} if additionally
\begin{equation}
\label{7} (|A_{13}|+|A_{24}|)^2 \neq 4|A_{14}| |A_{23}|.
\end{equation}
\end{Definition}

Further only regular boundary conditions
are considered. A multiplication
from the left of the system (\ref{1})
by the matrix $A^{-1}_{14} $  gives  an equivalent to (\ref{1}) system
\begin{eqnarray}
\label{8a}  y_1 (0) +b y_1 (\pi) + a y_2 (0) =0,
\\ \nonumber d y_1 (\pi) + c y_2 (0) +  y_2 (\pi)=0.
\end{eqnarray}
So, without loss of generality one may consider only $bc$ of the
form (\ref{8a}). The boundary conditions (\ref{8a}) are uniquely
determined by the matrix of coefficients $\begin{pmatrix} 1 & b& a &0\\
0 & d & c & 1 \end{pmatrix}.$ Then $bc$ is {\bf regular} if
\begin{equation}
\label{10}   bc-ad \neq 0,
\end{equation}
and {\bf strictly regular} if additionally
\begin{equation}
\label{11}   (b-c)^2 +4ad \neq 0.
\end{equation}
The characteristic equation (\ref{5}) becomes
\begin{equation}
\label{13} z^2 + (b+c)z + bc-ad =0.
\end{equation}

In the case of strictly regular boundary $bc$  (\ref{10}) and
(\ref{11}) guarantee that  (\ref{13}) has two
distinct nonzero roots $z_1$ and $z_2$
(i.e., the matrix $A_{23}=\begin{pmatrix} b & a \\ d  & c
\end{pmatrix}$ has
two {\em distinct} eigenvalues $-z_1, -z_2 $). Let us fix a pair
of corresponding eigenvectors $\begin{pmatrix} \alpha_1 \\
\alpha_2\end{pmatrix}$ and $
\begin{pmatrix} \beta_1 \\ \beta_2 \end{pmatrix}. $
Then the matrix $\begin{pmatrix}
\alpha_1 & \beta_1\\ \alpha_2 & \beta_2
\end{pmatrix}$ is invertible,  and we set
\begin{equation}
\label{20a}
\begin{pmatrix} \alpha^\prime_1 & \alpha^\prime_2\\
\beta^\prime_1 & \beta^\prime_2 \end{pmatrix} :=
\begin{pmatrix} \alpha_1 & \beta_1\\
\alpha_2 & \beta_2 \end{pmatrix}^{-1}.
\end{equation}

Let $\tau_1$ and $ \tau_2 $ be chosen so that
\begin{equation}
\label{14} z_1 = e^{i\pi \tau_1}, \quad z_2 = e^{i\pi \tau_2},
\quad  |Re\, \tau_1 - Re\, \tau_2| \leq 1.
\end{equation}
Then the eigenvalues  of $L^0_{bc}$ are $\lambda_{k,\nu}^0 = k +
\tau_\nu, \; \nu \in \{1,2\}, $ $ k \in 2\mathbb{Z},$ and a
 corresponding system of eigenvectors is $\Phi =\{\varphi^1_k,
 \varphi^1_k, \, k \in 2\mathbb{Z}\},$ where
\begin{equation}
\label{21} \varphi^1_k :=
\begin{pmatrix}  \alpha_1 e^{i \tau_1 (\pi-x)} e^{-ikx}
\\  \alpha_2 e^{i \tau_1 x} e^{ikx}\end{pmatrix},
\qquad \varphi^2_k :=\begin{pmatrix}  \beta_1 e^{i \tau_2 (\pi-x)}
e^{-ikx}
\\  \beta_2 e^{i \tau_2 x} e^{ikx}\end{pmatrix}.
\end{equation}

If $bc$ is regular but not strictly regular, then (\ref{10}) holds
but (\ref{11}) fails, i.e.,
\begin{equation}
\label{68} (b+c)^2 - 4(bc-ad) = (b-c)^2 + 4ad =0.
\end{equation}
In this case the  equation (\ref{13}) has a double root $z_*=-(b+c)/2
\neq 0 $ (because $bc-ad \neq 0$). Choose $\tau_* $  so that $ z_* =
\exp (i \pi \tau_*), \quad |\tau_*| \leq 1.$ Then each eigenvalue of
$L^0_{bc}$ is of algebraic multiplicity 2 and has the form $ \tau_*
+k,\;\; k \in 2\mathbb{Z}. $

We call the boundary conditions given by the system (\ref{8a}) {\em
periodic--type} if
\begin{equation}
\label{72} b=c, \quad a=0, \quad d=0,
\end{equation}
holds. The condition (\ref{72}) takes place if and only if $A_{23}+
z_* I$ is the zero matrix, so then any two linearly independent
vectors $\begin{pmatrix} \alpha_1
\\ \alpha_2
\end{pmatrix}$ and $\begin{pmatrix} \beta_1 \\ \beta_2
\end{pmatrix}$
are eigenvectors of $A_{23}. $ With any choice of such vectors, the
system $\Phi$  given by (\ref{21})  but with $\tau_2=\tau_1= \tau_*,
$ consists of corresponding eigenfunctions of $L^0_{bc}.$

Next we consider the case when (\ref{68}) holds but (\ref{72})
fails, i.e.,
\begin{equation}
\label{74} |b-c|+ |a| + |d| >0.
\end{equation}
In this case each eigenvalue of $L^0_{bc} $ is
of algebraic multiplicity 2 but of geometric multiplicity 1, i.e.,
associated eigenvectors appear. Here we have the following subcases:

(i)  $a= 0,$ then  (\ref{68}) implies $b=c,$ and by (\ref{74}) we
have $d\neq 0;$

(ii)  $d= 0,$ then  (\ref{68}) implies $b=c,$ and by (\ref{74}) we
have $a\neq 0;$

(iii)  $a,d\neq 0,$ then  (\ref{68}) implies $b\neq c.$

Now we set
\begin{equation}
\label{80}
\begin{pmatrix}
\alpha_1  &  \beta_1  \\ \alpha_2  &  \beta_2
\end{pmatrix}
= \begin{cases}
\begin{pmatrix}
0  &  \pi b  \\ d  &  0
\end{pmatrix}
&  \text{for}  \quad  (i), \vspace{1mm}\\
\begin{pmatrix}
a  &  0  \\ \frac{c-b}{2}  &  \pi b
\end{pmatrix}
&  \text{for}  \quad  (ii), (iii).
\end{cases}
\end{equation}
A corresponding system of eigenvectors is given by
\begin{equation}
\label{83} \Phi^1 =\{\varphi^1_k,\;k\in 2\mathbb{Z}\},     \quad
\varphi^1_k=
\begin{pmatrix}  \alpha_1 e^{i \tau_* (\pi-x)} e^{-ikx}
\\ \alpha_2 e^{i \tau_* x} e^{ikx}\end{pmatrix},
\end{equation}
and
\begin{equation}
\label{85}\Phi^2 =\{\varphi^2_k,\;k\in 2\mathbb{Z}\},     \quad
\varphi^2_k=
\begin{pmatrix}    (\beta_1 - \alpha_1 x) e^{i \tau_* (\pi-x)} e^{-ikx}
\\  (\beta_2 + \alpha_2 x) e^{i \tau_* x} e^{ikx}\end{pmatrix}
\end{equation}
 is a system of corresponding associated vectors.

\begin{Theorem}
\label{thm01} (a) For strictly regular or periodic type $bc,$ the
system $\Phi$ given by (\ref{21}) is a Riesz basis in the space $L^2
(I, \mathbb{C}^2), \; I= [0,\pi].$ Its biorthogonal system is
$\tilde{\Phi} =\{ \tilde{\varphi}^1_k, \tilde{\varphi}^2_k, \: k \in
2\mathbb{Z} \},$ where
\begin{equation}
\label{21*} \tilde{\varphi}^1_k :=
\begin{pmatrix}  \overline{\alpha^\prime_1} e^{i \overline{\tau_1}
(\pi-x)} e^{-ikx}
\\  \overline{\alpha^\prime_2} e^{i \overline{\tau_1} x} e^{ikx}\end{pmatrix},
\qquad \tilde{\varphi}^2_k :=
\begin{pmatrix}  \overline{\beta^\prime_1} e^{i \overline{\tau_2}
(\pi-x)} e^{-ikx}
\\ \overline{\beta^\prime_2} e^{i \overline{\tau_2} x} e^{ikx}\end{pmatrix},
\end{equation}
with $\alpha_1^\prime, \alpha_2^\prime, \beta_1^\prime,
\beta_2^\prime $
  coming, respectively, from (\ref{20a}) for strictly regular $bc$
or periodic type $bc.$

(b) For regular but not strictly regular $bc,$ the system $\Phi
=\Phi^1 \cup \Phi^2$  given in (\ref{83}) and (\ref{85})  is a Riesz
basis in the space $L^2 (I,\mathbb{C}^2).$ Its biorthogonal system is
$\tilde{\Phi} =\{ \tilde{\varphi}^1_k, \tilde{\varphi}^2_k, \: k \in
2\mathbb{Z} \},$ where
\begin{equation}
\label{91} \tilde{\varphi}^1_k= \begin{pmatrix} \bar{\Delta}^{-1}
\overline{\alpha_2}
e^{i\overline{\tau_*} (\pi - x)} e^{-ikx} \\
\bar{\Delta}^{-1} \overline{\alpha_1} e^{i\overline{\tau_*} x}
e^{ikx}
\end{pmatrix}, \quad
\tilde{\varphi}^2_k= \begin{pmatrix} \bar{\Delta}^{-1}
[\overline{\beta_2} +\overline{\alpha_2} (\pi -x)]
e^{i\overline{\tau_*} (\pi - x)} e^{-ikx}
\\ \bar{\Delta}^{-1} [\overline{\beta_1} - \overline{\alpha_1}(\pi-x)]
e^{i\overline{\tau_*} x} e^{ikx}
\end{pmatrix}
\end{equation}
with $\Delta =\alpha_1 \beta_2 - \alpha_2 \beta_1 + \pi \alpha_1
\alpha_2. $
\end{Theorem}

\begin{proof}
(a) First we consider the case of strictly regular or periodic type
$bc.$ The system
\begin{equation}
\label{23} E = \left \{e^1_k, e^2_k, \; k \in 2\mathbb{Z} \},
\quad e^1_k :=\begin{pmatrix} e^{i kx}\\
0
\end{pmatrix}, \quad  e^2_k :=\begin{pmatrix} 0
\\   e^{i kx} \end{pmatrix} \right \} ,
\end{equation}
is an orthonormal basis in $L^2 (I, \mathbb{C}^2), \; I= [0,\pi].$ We
have $\Phi = A(E),$ where the operator $A: L^2 (I, \mathbb{C}^2) \to
L^2 (I, \mathbb{C}^2)$ is defined by
\begin{equation}
\label{24} A
\begin{pmatrix}  f\\ g \end{pmatrix} =
\begin{pmatrix}  \alpha_1 e^{i \tau_1 (\pi-x)} f(\pi -x)
\\  \alpha_2 e^{i \tau_1 x} f(x) \end{pmatrix}
+ \begin{pmatrix}  \beta_1 e^{i \tau_2 (\pi-x)} g(\pi -x)
\\  \beta_2 e^{i \tau_2 x} g(x) \end{pmatrix}.
\end{equation}
Since the functions $\displaystyle e^{i \tau_\nu x} $ and
$\displaystyle e^{i \tau_\nu (\pi -x)}, \; \nu =1,2, $ are bounded,
it follows that $A$ is bounded operator. In view of (\ref{20a}) its
inverse operator $A^{-1} $ is given by
\begin{equation}
\label{26} A^{-1}
\begin{pmatrix}  F\\ G \end{pmatrix} =
\begin{pmatrix} e^{-i \tau_1 x}[ \alpha^\prime_1 F (\pi-x) + \alpha^\prime_2
G(x)]
\\  e^{-i \tau_2 x} [\beta^\prime_1 F (\pi-x) + \beta^\prime_2 G(x)] \end{pmatrix},
\end{equation}
so one can easily see that $A^{-1} $ is bounded as well. Thus, $A$ is
an isomorphism, which proves that the system $\Phi$ is a Riesz basis.
By (\ref{p2a}), its biorthogonal system is $\tilde{\Phi}
=(A^{-1})^*(E),$ so we obtain (\ref{21*}) by a direct computation.

(b) In the case of regular but not strictly regular or periodic type
$bc$ we consider the operator $A: L^2 (I, \mathbb{C}^2) \to L^2 (I,
\mathbb{C}^2)$ defined by
\begin{equation}
\label{94} A
\begin{pmatrix}  f\\ g \end{pmatrix} =
\begin{pmatrix}  \alpha_1 e^{i \tau_* (\pi-x)} f(\pi -x)
\\  \alpha_2 e^{i \tau_* x} f(x) \end{pmatrix}
+ \begin{pmatrix}  (\beta_1 - \alpha_1 x) e^{i \tau_* (\pi-x)} g(\pi
-x)
\\  (\beta_2 + \alpha_2 x) e^{i \tau_* x} g(x) \end{pmatrix}.
\end{equation}
Then $\Phi = A(E),$ where $E$ is the orthonormal basis (\ref{23}).
One can easily see that the operator $A$ is bounded and its inverse operator
\begin{equation}
\label{96} A^{-1}
\begin{pmatrix}  F\\ G \end{pmatrix} =\frac{1}{\Delta}
\begin{pmatrix} \left [(\beta_2+ \alpha_2 x) F (\pi-x) -
(\beta_1 - \pi \alpha_1 + \alpha_1 x)   G(x) \right ] e^{-i \tau_*
x}\\ \left [ -\alpha_2 F (\pi-x) + \alpha_1 G(x)\right ]e^{-i \tau_*
x}
\end{pmatrix}
\end{equation}
is also bounded. Thus the system $\Phi $ given by (\ref{83}) and
(\ref{85}) is a Riesz basis. A direct computation of $\tilde{\Phi}
=(A^{-1})^*(E)$ shows that (\ref{91}) holds.

\end{proof}

{\em Corollary.}
By Theorem~\ref{thm01},
\begin{equation}
\label{ec8}
\sup_{k\in 2\mathbb{Z}, \nu\in \{1,2\}}  \sup_{[0,\pi]}
\left | \varphi^\nu_k (x) \right | = C =C(bc)< \infty.
\end{equation}

{\em Remark.} Here and thereafter, we denote by $C$ any
constant that is absolute up to dependence on $bc.$

\begin{Lemma}
\label{lembc}
Let $bc$ be given by (\ref{8a}), and let $A$ be defined,
respectively, by (\ref{24}) if $bc$ is strictly regular and by (\ref{94}) if
$bc$ is regular but not strictly regular.
If $f, g, F, G $ are continuous functions on $[0, \pi]$
such that $ \begin{pmatrix} F  \\ G \end{pmatrix}
= A \begin{pmatrix} f  \\ g \end{pmatrix},$
then $\begin{pmatrix} F \\ G \end{pmatrix}$
satisfies $bc$ if and only if
$\begin{pmatrix} f  \\ g \end{pmatrix}$ satisfies the periodic boundary
conditions $f(0) = f(\pi), \; g(0) = g(\pi).$
\end{Lemma}

\begin{proof}
In view of (\ref{24}) and (\ref{94}), the operator $A$ acts
"point-wise" on vector-functions $\begin{pmatrix} f  \\ g
\end{pmatrix}$ multiplying components $f$ and $g$ by smooth
functions. Therefore, $A$ generates a linear operator $\tilde{A}: \,
\mathbb{C}^4 \to \mathbb{C}^4 $ such that
$$
\tilde{A} (f(0), f(\pi), g(0), g(\pi)) = (F(0), F(\pi), G(0),
G(\pi)).
$$
Since $A$ is invertible, $\tilde{A}$ is also invertible.

The periodic boundary conditions $f(0) = f(\pi), \; g(0) = g(\pi)$
define a two-dimensional subspace $E_{Per} \subset \mathbb{C}^4,$
and the boundary conditions (\ref{8a}) define a two-dimensional
subspace $E_{bc} \subset \mathbb{C}^4.$ In fact, the lemma claims
that $\tilde{A}E_{Per}=E_{bc}.$  Since $\tilde{A}$ is an
isomorphism, is is enough to show that $\tilde{A}E_{Per}\subset
E_{bc}.$  In other words, we need only to prove that if
$\begin{pmatrix} f  \\ g \end{pmatrix}$ satisfies the periodic
boundary conditions then $ \begin{pmatrix} F  \\ G \end{pmatrix} = A
\begin{pmatrix} f  \\ g \end{pmatrix}$  satisfies (\ref{8a}).

Since
$$ \begin{pmatrix} F  \\ G \end{pmatrix}
= A\begin{pmatrix} f  \\ g \end{pmatrix}
= A\begin{pmatrix} f  \\ 0 \end{pmatrix} + A\begin{pmatrix} 0  \\ g \end{pmatrix},
$$
it is enough to show that $A\begin{pmatrix} f  \\ 0 \end{pmatrix}$
and $A\begin{pmatrix} 0  \\ g \end{pmatrix}$
satisfy $bc.$  Set
$\begin{pmatrix} y_1  \\ y_2 \end{pmatrix}:=
A\begin{pmatrix} f  \\ 0 \end{pmatrix}.$

If $bc$ is strictly regular or periodic type
boundary condition, then by (\ref{14}) and
(\ref{24}) it follows that
$$
y_1 (0) =\alpha_1 z_1 f(\pi), \quad
y_1 (\pi) =\alpha_1  f(0), \quad  y_2 (0) =\alpha_2  f(0), \quad
y_2 (\pi) =\alpha_2 z_1 f(\pi).
$$
Therefore, taking into account that $f(\pi) = f(0),$ we obtain
that $\begin{pmatrix} y_1  \\ y_2 \end{pmatrix}$ satisfies
the boundary conditions (\ref{8a}):
$$
\begin{pmatrix} 1 & b & a & 0\\  0  &  d  &  c  &  1 \end{pmatrix}
\begin{pmatrix} \alpha_1  z_1 f(\pi)\\ \alpha_1 f(0) \\
\alpha_2 f(0)
\\ \alpha_2  z_1 f(\pi)\end{pmatrix} =  f(0)
\begin{pmatrix} b+z_1  &  a \\  d &  c + z_1 \end{pmatrix}
\begin{pmatrix} \alpha_1  \\ \alpha_2 \end{pmatrix}= 0
$$
(due to the definition of
$\begin{pmatrix} \alpha_1  \\ \alpha_2 \end{pmatrix},$
see the lines prior to (\ref{20a})).

If $bc $ is regular but not strictly regular,
then the characteristic equation (\ref{13})
has a double root $z_* = - (b+c)/2 = \exp (i \pi \tau_*),$
and by (\ref{94}) we have
$$
y_1 (0) =\alpha_1 z_* f(\pi), \quad
y_1 (\pi) =\alpha_1  f(0), \quad  y_2 (0) =\alpha_2  f(0), \quad
y_2 (\pi) =\alpha_2 z_* f(\pi).
$$
Using (\ref{80}), one can easily verify that
$\begin{pmatrix} y_1  \\ y_2 \end{pmatrix}$ satisfies
the boundary conditions (\ref{8a}).

The proof that
$A\begin{pmatrix} 0  \\ g \end{pmatrix}$
satisfies the boundary conditions (\ref{8a}) is similar;
we omit the details.

\end{proof}

\section{Estimates for the resolvent of $L_{bc} $ and localization of spectra}

The operator $L_{bc} (v) $
maybe considered as a perturbation of the free
operator   $ L^0_{bc}.$  We study   $L_{bc} (v)$
by considering its Fourier matrix representation
with respect to the Riesz basis $\Phi = \Phi (bc) $
consisting of root functions of  $ L^0_{bc},$ which was
constructed in Theorem~\ref{thm01}.

For strictly regular $bc,$
the spectrum of $L_{bc}^0$ consists of two disjoint
sequences of simple eigenvalues
$$
Sp (L_{bc}^0) = \{\tau_1 +k, \; k \in 2\mathbb{Z} \} \cup \{\tau_2
+k, \; k \in 2\mathbb{Z} \}, $$
where $\tau_1, \tau_2  $ depend on $bc.$
For regular but not strictly regular $bc$ the
spectrum of $L_{bc}^0$ consists of eigenvalues of algebraic multiplicity 2
and has the form
$$
Sp \, (L^0_{bc}) = \{ \tau_* +k,\;\; k \in 2\mathbb{Z} \},
$$
where $\tau_* $ depends on $bc.$
In both cases the resolvent operator $R_{bc}^0 (\lambda) = (\lambda
-L_{bc}^0)^{-1}$ is well defined for $\lambda \not\in Sp (L_{bc}^0),
$ and we have
\begin{equation}
\label{1.10} R_{bc}^0 (\lambda) \varphi^\mu_k = \frac{1}{\lambda -
\tau_\nu -k}\varphi^\nu_k, \quad   k \in 2\mathbb{Z}, \;\;
\nu \in \{1,2\},
\end{equation}
where we put $\tau_1 = \tau_2 = \tau_*$
in the case of regular but not strictly regular $bc.$
Moreover, $R_{bc}^0 (\lambda)$ acts continuously from $L^2 ([0,\pi], \mathbb{C}^2)$
into $L^\infty ([0,\pi], \mathbb{C}^2).$
Indeed, if $F = \sum_{k,\nu} F^\nu_k  \varphi^\nu_k,$
then
$$R_{bc}^0 (\lambda) F = \sum_{k,\nu} \frac{F^\nu_k}{\lambda - k -\tau_\nu}  \varphi^\nu_k. $$
By (\ref{ec8}), $
\|R_{bc}^0 (\lambda) F \|_\infty \leq
\sum_{k,\nu} \frac{C|F^\nu_k |}{|\lambda - k -\tau_\nu |}, $
so by the Cauchy inequality,
 $$ \|R_{bc}^0 (\lambda) F \|_\infty
\leq C\left ( \sum_{k,\nu} \frac{1}{|\lambda - k -\tau_\nu|^2} \right )^{1/2}
\left ( \sum_{k,\nu} |F^\nu_k |^2 \right )^{1/2}.
$$
Thus, in view of (\ref{p3}),
\begin{equation}
\label{1.15} \|R_{bc}^0 (\lambda)  \|_{L^2 \to L^\infty} \leq C_1
[a(\lambda -\tau_1) + a(\lambda -\tau_2)]^{1/2},
\end{equation}
where $C_1= C_1 (bc) $ and $a(\lambda ) $ is defined and estimated in
the following lemma.

\begin{Lemma}
\label{lema1}
Consider the function
\begin{equation}
\label{a01}
a(\lambda) = \sum_{k \in 2\mathbb{Z}} \frac{1}{|\lambda -k|^2}, \quad
\lambda \not\in 2\mathbb{Z}.
\end{equation}
If $\lambda = m+ \xi + it $ with $m \in 2\mathbb{Z} $ and  $ |\xi|<1, $  then
\begin{equation}
\label{a02}
a(\lambda) \leq \frac{1}{\xi^2 + t^2} + \frac{8}{1+2|t|}.
\end{equation}

\end{Lemma}

\begin{proof}
  Since $a(\lambda + 2) = a(\lambda), $
it is enough to consider the case where $m=0. $
Then we have
$$ a(\lambda ) =   \sum_{k \in 2\mathbb{Z}}  \frac{1}{(\xi -k)^2 + t^2}
\leq \frac{1}{\xi^2 + t^2} + 2\sum_{k \in 2\mathbb{N}}
\frac{1}{(k-1)^2 + t^2}.$$ Since $ (k-1)^2 + t^2 \geq \frac{1}{2}
(k-1+|t|)^2, $ we obtain
$$
\sum_{k \in 2\mathbb{N}}  \frac{1}{(k-1)^2 + t^2} \leq
\int_{3/2}^\infty \frac{2}{(u-1+|t|)^2}du = \frac{4}{1+2|t|},$$ which
completes the proof.

\end{proof}

Let $V:  L^2 (I, \mathbb{C}^2) \to L^2 (I, \mathbb{C}^2)$ be the operator of
multiplication by the matrix $v(x) = \left [
\begin{array}{cc}0&P(x)\\Q(x)&0
\end{array} \right ],$  i.e.,
$ V \begin{pmatrix}y_1 \\ y_2  \end{pmatrix} =
 \begin{pmatrix} Py_2 \\ Qy_1 \end{pmatrix}.
$ The operator $V$ could be unbounded in $L^2 (I, \mathbb{C}^2)$ but
it acts continuously from $ L^\infty (I, \mathbb{C}^2)$ into $L^2 (I,
\mathbb{C}^2).$  Indeed, if $y_1, y_2 \in L^\infty  $ then
$$
\left \| \begin{pmatrix} Py_2 \\ Qy_1 \end{pmatrix}  \right \|^2
= \frac{1}{\pi} \int_0^\pi ( |Qy_1|^2 + |Py_2|^2) dx \leq
\|v\|^2 \cdot  \left \| \begin{pmatrix} y_1 \\ y_2 \end{pmatrix}  \right \|_\infty^2,
$$
where we set for convenience
\begin{equation}
\label{1.16} \|v\|^2 = \|P\|^2 + \|Q\|^2, \quad \left \|
\begin{pmatrix} y_1 \\ y_2 \end{pmatrix}  \right \|_\infty= max
(\|y_1\|_\infty, \|y_2\|_\infty).
\end{equation}
Therefore,
\begin{equation}
\label{1.17}
\|V\|_{L^\infty \to L^2}  \leq \|v\|.
\end{equation}

The following lemma follows immediately from the explicit form of the
bases $\Phi $ and their biorthogonal systems given in
Theorem~\ref{thm01}.

\begin{Lemma}
\label{lem10} The matrix representation of $V$ with respect to the
basis $\Phi$ has the form
\begin{equation}\label{m1}
V \sim \left [ \begin{array}{cc} V^{11} & V^{12}\\V^{21}&V^{22}
\end{array} \right ], \quad  V^{\eta\nu}= \left ( V^{\eta\nu}_{jk} \right
)_{j,k\in 2\mathbb{Z}}, \quad \eta, \nu \in \{1,2\},
\end{equation}
\begin{equation}\label{m2}
 V^{\eta\nu}_{jk} =\langle V\varphi^\nu_k, \tilde{\varphi}^\eta_j \rangle =
 w^{\eta\nu} (j+k),
\end{equation}
where
\begin{equation}\label{m3}
w^{\eta\nu}=\left
 (w^{\eta\nu}(m)\right ) \in \ell^2 (2\mathbb{Z}), \quad
 \|w^{\eta\nu}\|_{\ell^2} \leq C \|v\|,
 \end{equation}
with $C=C(bc).$
\end{Lemma}
\begin{proof}
Indeed, in view of the explicit formulas for the basis
$\Phi=\{\varphi_k^\nu \} $ and its biorthogonal system $\tilde{\Phi}
= \{\tilde{\varphi}^{\nu} \}$ it follows that (\ref{m2}) holds with
\begin{equation}\label{m4}
 w^{\eta\nu} (m)=  p^{\eta\nu} (-m) + q^{\eta\nu} (m),
\end{equation}
where $p^{\eta\nu} (k), \; q^{\eta\nu} (k), \; k \in 2\mathbb{Z}$
are, respectively,  the Fourier coefficients of functions of the form
\begin{equation}\label{m5}
g^{\eta\nu} (x) P(x), \;\;  h^{\eta\nu} (x) Q(x), \quad g^{\eta\nu},
h^{\eta\nu} \in C^\infty ([0,\pi]), \;\; \eta, \nu = 1,2.
\end{equation}

\end{proof}

For convenience, we set
\begin{equation}
\label{1.18} r(m) = \max \{|w^{\mu\nu} (m)|, \; \mu, \nu =1,2 \},
\quad m \in 2\mathbb{Z};
\end{equation}
then
\begin{equation}
\label{1.18a} r= (r(m)) \in \ell^2 (2\mathbb{Z}), \quad \|r\|\leq C
\|v\|.
\end{equation}

The standard perturbation formula for the resolvent
\begin{equation}
\label{1.21}
R_{bc} (\lambda) =  R_{bc}^0 (\lambda) +R_{bc}^0 (\lambda)VR_{bc}^0
(\lambda)+\sum_{m=2}^\infty R_{bc}^0 (\lambda)(VR_{bc}^0 (\lambda))^m
\end{equation}
is valid if the series on the right converges. From (\ref{1.15}) and
(\ref{1.17}) it follows that $ VR_{bc}^0 (\lambda) $ is a continuous
operator in $L^2 ([0,\pi]), \mathbb{C}^2) $ which norm does not
exceed
\begin{equation}
\label{2.22} \|VR_{bc}^0 (\lambda)\| \leq \|V\|_{L^\infty \to L^2}
\|R_{bc}^0 (\lambda)\|_{L^2 \to L^\infty} \leq \|v\|\cdot [a(\lambda
-\tau_1) + a(\lambda -\tau_2)]^{1/2}.
\end{equation}
But this estimate does not guarantee the convergence in (\ref{1.21})
for $\lambda$'s which are close to the real line. Therefore,  next we
estimate the norms $\|R_{bc}^0 (\lambda)VR_{bc}^0 (\lambda)\|_{L^2
\to L^\infty}.$

But first we introduce some notations. For each $\ell^2$--sequence
$x=(x(j))_{j \in 2\mathbb{Z}} $ and $m>0$ we set
\begin{equation}
\label{2.23} \mathcal{E}_m (x) = \left (   \sum_{|j|\geq m} |x(j)|^2
\right )^{1/2}.
\end{equation}
In a case of strictly regular $bc,$ we subdivide the complex plane
$\mathbb{C}$ into strips
\begin{equation}
\label{2.24} H_m = \left \{z\in \mathbb{C}: \; -1 \leq Re \left
(z-m- \frac{\tau_1 +\tau_2}{2}  \right ) \leq 1  \right \}, \quad m
\in 2\mathbb{Z},
\end{equation}
and set
\begin{equation}
\label{2.25} H^N = \bigcup_{|m|\leq N}    H_m,
\end{equation}
\begin{equation}
\label{2.29}  \rho: = \frac{1}{2}\min ( 1-|Re (\tau_1 -\tau_2)|/2,\,
|\tau_1 -\tau_2|/2),
\end{equation}
\begin{equation}
\label{2.30} D_m^\mu =\{z\in \mathbb{C}: \;\; |z-\tau_\mu -m| < \rho
\}, \quad m \in 2\mathbb{Z}.
\end{equation}
and
\begin{equation}
\label{2.26}  R_{NT} = \left \{z= x+it: \;  \;
 \left |  x-Re \, \frac{\tau_1 +\tau_2}{2}   \right | < N+1,\;
  |t| < T \right \},
\end{equation}
where  $N \in 2\mathbb{N}$  and $T>0.$

In case of regular but not strictly regular boundary conditions  we
subdivide the complex plane $\mathbb{C}$ into strips
\begin{equation}
\label{2.24a} H_m = \left \{z\in \mathbb{C}: \; -1 \leq Re \left
(z-m- \tau_*  \right ) \leq 1  \right \}, \quad m \in 2\mathbb{Z},
\end{equation}
and set
\begin{equation}
\label{2.25a} H^N = \bigcup_{|m|\leq N}  H_m,
\end{equation}
\begin{equation}
\label{2.30a} D_m =\{z\in \mathbb{C}: \;\; |z-\tau_* -m| < 1/4 \},
\quad m \in 2\mathbb{Z}.
\end{equation}
and
\begin{equation}
 \label{2.26a}  R_{NT} = \left \{z= x+it: \;  \;
 \left |  x-Re \, \tau_*   \right | < N+1,\;     |t| < T \right \},
\end{equation}
where $N \in 2\mathbb{N}$ and $T>0.$

\begin{Lemma}
\label{loc1}
(a) For $\lambda \in H_m
\setminus ( D^1_m \cup D^2_m
 ) $ in a case of strictly regular $bc,$
or for $\lambda \in H_m  \setminus  D_m $ in a case of regular but not
strictly regular $bc,$ we have
\begin{equation}
\label{2.31} \|R_{bc}^0 (\lambda)VR_{bc}^0 (\lambda)\|_{L^2 \to
L^\infty} \leq  C \left ( \frac{\|v\|}{\sqrt{|m|}} +
(\mathcal{E}_{|m|} (r))  \right ) \quad \text{for} \;\; m\neq 0,
\end{equation}
where $C=C(bc).$

(b) If $\; T\geq 1 + 2|\tau_1| +2|\tau_2|, $   then
\begin{equation}
\label{2.32} \|R_{bc}^0 (\lambda)VR_{bc}^0 (\lambda)\|_{L^2 \to
L^\infty} \leq   C \frac{\|v \|}{T} \quad \text{for} \;\;  |Im \,
\lambda| \geq T,
\end{equation}
where $C=C(bc).$
\end{Lemma}

\begin{proof}
In view of the matrix representations of the operators $V$ and
$R_{bc}^0 (\lambda)$ given in (\ref{1.10}) and Lemma~\ref{lem10},  if
$ F = \sum_{k,\nu} F_k^\nu \varphi_k^\nu $ then
$$
R_{bc}^0 (\lambda)VR_{bc}^0 (\lambda)F =
\sum_{k,\nu} \sum_{j, \eta} \frac{w^{\eta \nu} (j+k) F_k^\nu}
{(\lambda - j - \tau_\eta)(\lambda - k -\tau_\nu)} \varphi_k^\nu,
$$
where $\tau_1 = \tau_2 = \tau_*$ in case of regular
but not strictly regular $bc.$

For convenience, we set
\begin{equation}
\label{2.34}
g_k = \max (|F_k^1|, |F_k^2|), \quad k \in 2\mathbb{Z}.
\end{equation}
Then $ g = (g_k) \in \ell^2 $ and $\|g\| \leq C \|F\|,$ where $C=
C(bc).$

By (\ref{ec8}) and (\ref{1.18}),
\begin{equation}
\label{2.35}
\|R_{bc}^0 (\lambda)VR_{bc}^0 (\lambda)F\|_\infty \leq  C
\sum_{k,\nu} \sum_{j, \eta} \frac{r (j+k) |g_k|}
{|\lambda - j - \tau_\eta||\lambda - k -\tau_\nu |}.
\end{equation}

Let $\lambda  \in H_m. $ Then
$\lambda = m+\xi +
Re \frac{\tau_1 + \tau_2}{2} + i t $ with $|\xi| \leq 1 $
(in case of regular but not strictly regular $bc$
$ \tau_1 = \tau_2= \tau_*$).
Therefore, in view of (\ref{14}), we have for $k\neq m $
\begin{equation}
\label{2.36}
|\lambda -k -\tau_\nu | \geq |m-k|- 1- \frac{1}{2}|Re (\tau_1 -  \tau_2)|
\geq |m-k| - \frac{3}{2} \geq \frac{1}{4} |m-k|.
\end{equation}

For strictly regular $bc $ and $ \lambda \in H_m \setminus (D_m^1
\cup D_m^2), $ (\ref{2.36}) implies
\begin{equation}
\label{2.36a} \|R_{bc}^0 (\lambda)VR_{bc}^0  (\lambda)F\|_\infty \leq
4C \left ( \frac{r(2m)g_m}{\rho^2} + \frac{4}{\rho} \sigma_1 +
\frac{4}{\rho} \sigma_2 + 16 \sigma_3 \right )
\end{equation}
where $\displaystyle \sigma_1 =
 \sum_{k\neq m}\frac{r(m+k)g_k}{|m-k|}, $
$$
\sigma_2 = \sum_{j\neq m}\frac{r(j+m)g_m}{|j-m|}, \quad \sigma_3 =
 \sum_{k\neq m} \sum_{j\neq m} \frac{r(j+k)g_k}{|m-j||m-k|}.
$$
We have
$$
 \sum_{k\neq m} \frac{r(m+k)}{|m-k|} =
 \sum_{|k-m|>|m|}\frac{r(m+k)}{|m-k|}+
 \sum_{|k-m| \leq |m|}\frac{r(m+k)}{|m-k|}.
$$
By the Cauchy inequality,
$$ \sum_{|k-m|>|m|}\frac{r(m+k)}{|m-k|}\leq
\|r\| \cdot \left ( \sum_{|k-m|>|m|} \frac{1}{(m-k)^2} \right )^{1/2}
\leq \frac{\|r\|}{\sqrt{|m|}}.
$$
On the other hand, if $|k-m| \leq |m| $ then $|k+m| \geq 2|m| -
|m-k|\geq |m|.  $ Therefore
$$\sum_{|k-m|\leq |m|}\frac{r(m+k)}{|m-k|}
\leq \left ( \sum_{|k-m|\leq |m|} r(m+k)|^2 \right )^{1/2} \left (
\sum_{k \neq m} \frac{1}{(m-k)^2} \right )^{1/2}
$$
$$\leq
\left ( \sum_{j \geq |m|} r(j)|^2 \right )^{1/2}
\cdot \frac{\pi}{\sqrt{3}}  \leq 2 \mathcal{E}_{|m|} (r).
$$
Since $|g_k| \leq \|g\|, $ the above inequalities imply that
\begin{equation}
\label{2.37}
\sigma_\alpha \leq  \|g\|\left (\frac{\|r\|}{\sqrt{|m|}} +
2\mathcal{E}_{|m|} (r) \right ), \quad \alpha=1,2.
\end{equation}
Next we estimate $\sigma_3 \leq \sigma_3^1 + \sigma_3^2  +\sigma_3^3,
$ where
$$
\sigma_3^1 =
\sum_{|m-k| > \frac{|m|}{2}} \sum_{j\neq m} \frac{r(j+k)g_k}{|m-j||m-k|},
\quad
\sigma_3^2 =
 \sum_{k\neq m} \sum_{|m-j| > \frac{|m|}{2}} \frac{r(j+k)g_k}{|m-j||m-k|},
$$
$$
\sigma_3^3  =
\sum_{|m-k| \leq \frac{|m|}{2}} \; \sum_{|m-j| \leq \frac{|m|}{2}} \frac{r(j+k)g_k}{|m-j||m-k|} \;.
$$

By the Cauchy inequality,
$$
(\sigma_3^1)^2 \leq  \sum_k |g_k|^2 \cdot  \sum_{j} |r(j+k)|^2
\cdot \sum_{j\neq m} \frac{1}{(j-m)^2} \; \cdot
\sum_{|m-k| > \frac{|m|}{2}} \frac{1}{(k-m)^2}
$$
$$
\leq \|r\|^2 \cdot \|g\|^2 \cdot \frac{\pi^2}{3} \cdot \frac{4}{|m|}
\leq  \frac{16}{|m|} \|r\|^2 \cdot \|g\|^2.
$$
The same argument gives the same estimate for $\sigma_3^2. $

On the other hand, if $|j-m| \leq |m|/2$ and
$|k-m| \leq |m|/2,$ then $|j+k| \geq 2|m| - |m-j| - |m-k| \geq |m|. $
Therefore, by the Cauchy inequality we obtain
$$
(\sigma^3_3)^2 \leq \
\sum_{|k-m| \leq \frac{|m|}{2}} |g_k|^2   \sum_{|j-m|
\leq \frac{|m|}{2}} |r(j+k)|^2  \sum_{j\neq m} \frac{1}{(j-m)^2}
\sum_{k\neq m} \frac{1}{(k-m)^2}
$$
$$
\leq \|g\|^2 (\mathcal{E}_{|m|} (r))^2 \frac{\pi^2}{3} \frac{\pi^2}{3} \leq
16 \|g\|^2 (\mathcal{E}_{|m|} (r))^2.
$$
From the above estimates it follows
\begin{equation}
\label{2.38}
\sigma_3 \leq
4\|g\|\left (\frac{2 \|r\|}{\sqrt{|m|}} +
\mathcal{E}_{|m|} (r) \right ).
\end{equation}

Now, in view of (\ref{2.36a}), the estimates (\ref{2.37}) and
(\ref{2.38}) imply (\ref{2.31}) in the case of strictly regular $bc.
$ In the case of regular but not strictly regular $bc $ the proof is
the same because for $ \lambda \in H_m \setminus D_m $ (\ref{2.35})
implies (\ref{2.36a}) with $\rho = 1/4.$

Finally, we prove (b). By  (\ref{1.15}) and (\ref{1.17}), it follows
that
$$
\| R_{bc}^0 (\lambda) V R_{bc}^0 (\lambda) \|_{L^2 \to L^\infty}
\leq C \|v\|\cdot [a(\lambda - \tau_1)  + a(\lambda - \tau_2)].
$$
If $T > 1+ 2|\tau_1 | + 2|\tau_2 |$  and $|Im \, \lambda | \geq T,$
then
$$
|Im \, (\lambda - \tau_\nu ) | \geq T - |\tau_\nu| \geq T/2, \quad
\nu = 1,2.
$$
Therefore, by Lemma \ref{lema1} we obtain, for $ |Im \, \lambda |\geq
T, $
$$ a(\lambda - \tau_1) + a(\lambda - \tau_2) \leq
\frac{2}{(T/2)^2} + \frac{16}{T} \leq \frac{24}{T}, $$ which implies
(\ref{2.32}).
\end{proof}

By (\ref{1.17}) we have
$$ \|VR_{bc}^0 (\lambda)VR_{bc}^0 (\lambda)\|_{L^2 \to L^2} \leq
\|v\| \cdot \|R_{bc}^0 (\lambda)VR_{bc}^0 (\lambda)\|_{L^2 \to
L^\infty}.
$$
Therefore, in view of (\ref{2.31}) and (\ref{2.32}), the following
holds.

\begin{Corollary}
\label{cor1} There are $N_0 \in 2\mathbb{N} $ and $ T_0 >0 $ such
that if $|m| \geq N $ and $\lambda \in H_m \setminus ( D^1_m \cup
D^2_m  ) $ in case of strictly regular $bc $ or $\lambda \in H_m
\setminus  D_m  $ in case of regular but not strictly regular $bc, $
or if $| Im \, \lambda | \geq T_0,$
\begin{equation}
\label{2.50}
\|VR_{bc}^0 (\lambda)VR_{bc}^0 (\lambda)\|_{L^2 \to L^2} \leq 1/2.
\end{equation}
\end{Corollary}

By (\ref{1.21}), the validity of (\ref{2.50}) for some $\lambda $
means that the resolvent $R_{bc} (\lambda) $ exists for that
$\lambda, $ which leads to the following localization of spectra
assertion.
\begin{Lemma}
\label{srl} (a) For strictly regular $bc$
 there is an $N_0 = N_0(v, bc) \in 2\mathbb{N}$ and $T_0=T_0 (v,bc)>0$ such that
if $N\geq N_0, \; N\in 2\mathbb{N}, $ and $T\geq T_0 $ then
\begin{equation}
\label{2.41} Sp\, (L_{bc} (v_{\zeta}) \subset R_{NT} \cup
\bigcup_{|m|>N} \left ( D^1_m \cup D^2_m \right ) \quad  \text{for}
\;\; v_{\zeta}= \zeta v, \; |\zeta| \leq 1.
\end{equation}
(b) For regular but not strictly regular $bc$
 there is an $N_0 = N_0(v, bc) \in 2\mathbb{N}$
  and $T_0=T_0(v,bc)>0$ such that
if $N\geq N_0, \; N\in 2\mathbb{N}, $ and $T\geq T_0 $ then
\begin{equation}
\label{2.41a} Sp\, (L_{bc} (v_\zeta) \subset R_{NT} \cup \bigcup_{|m|>N}
D_m \quad \text{for} \;\; v_\zeta = \zeta v, \;\; |\zeta| \leq 1.
\end{equation}
\end{Lemma}

For strictly regular $bc,$ consider the Riesz projections associated
with $L=L_{bc}$
\begin{equation}
\label{2.61} S_N = \frac{1}{2\pi i} \int_{\partial R_{NT}} (\lambda -
L)^{-1} d\lambda, \quad P_{n,\alpha} = \frac{1}{2\pi i}
\int_{\partial D_n^\alpha} (\lambda - L)^{-1} d\lambda, \quad
\alpha=1,2,
\end{equation}
and let $S^0_N$ and $P_{n,\alpha}^0$ be the Riesz projections
associated with the free operator $L^0_{bc}.$ A standard argument
(continuity about the parameter $\zeta $ in (\ref{2.41}) shows that
\begin{equation}
\label{2.62}  \dim P_{n,\alpha}= \dim P_{n,\alpha}^0= 1, \quad \dim
S_N = \dim S_N^0 =2N+2.
\end{equation}

 If $bc$ are regular but not strictly
regular,   consider the Riesz projections associated with $L=L_{bc}$
\begin{equation}
\label{2.64} S_N = \frac{1}{2\pi i} \int_{\partial R_{NT}} (\lambda -
L)^{-1} d\lambda, \quad P_n = \frac{1}{2\pi i} \int_{\partial D_n}
(\lambda - L)^{-1} d\lambda,
\end{equation}
and let $S^0_N$ and $P_n^0$ be the Riesz projections associated with
the free operator $L^0_{bc}.$ The same argument, as in the case of
strictly regular $bc,$ shows that
\begin{equation}
\label{2.65}  \dim P_n= \dim P_n^0= 2, \quad \dim S_N = \dim
S_N^0=2N+2.
\end{equation}
Further analysis of the Riesz projections leads to the following
result -- see \cite[Theorems 15, 20]{DM23}.

\begin{Theorem}
\label{thm1} Suppose $v$ is an $L^2 $-Dirac potential.

(a) If $bc$ is strictly regular then there are $N_0=N_0(v,bc) \in
2\mathbb{N}$ and $T_0= T_0(v,bc)>0 $ such that if $N\geq N_0, \; N\in
2\mathbb{N}, $ and $T\geq T_0 $ then the Riesz projections $S_N,
\;S_N^0$ and $ P_{n,\alpha}, \, P_{n,\alpha}^0,$ $ \; n\in
2\mathbb{Z}, \;|n|>N, $ are well defined by (\ref{2.61}), and we have
\begin{equation}
\label{2.71} \sum_{|n|>N} \|P_{n,\alpha} - P_{n,\alpha}^0\|^2 <
\infty, \quad \alpha =1,2.
\end{equation}
Moreover,  the system $\{S_N; \;P_{n,\alpha}, \;n\in 2\mathbb{Z}, \;
|n|>N, \, \alpha = 1,2\}$ is a Riesz basis of projections in $L^2 (
[0,\pi] , \mathbb{C}^2),$  i.e.,
\begin{equation}
\label{2.72}  {\bf f} = S_N ({\bf f}) +  \sum_{\alpha=1}^2
\sum_{|n|>N} P_{n,\alpha} ({\bf f})
 \quad    \forall  {\bf f} \in L^2 ( [0,\pi] , \mathbb{C}^2),
\end{equation}
where the series converge unconditionally.

(b) If $bc $ is regular but not strictly regular, there are
$N_0=N_0(v,bc) \in 2\mathbb{N}$ and $T_0= T_0(v,bc)>0 $ such that if
$N\geq N_0, \; N\in 2\mathbb{N}, $ and $T\geq T_0 $ then the Riesz
projections $S_N, S_N^0 $ and $ \,P_n, \, P_n^0, \; n\in 2\mathbb{Z},
\; |n|>N, $  are well defined by (\ref{2.64}), and we have
\begin{equation}
\label{2.74}   \sum_{|n|>N} \|P_n - P_n^0\|^2 < \infty.
\end{equation}
Moreover,  the system $\{S_N; \;P_n, \; n\in 2\mathbb{Z}, \; |n|>N
\, \}$ is a Riesz basis of projections in $L^2 ( [0,\pi] ,
\mathbb{C}^2),$  i.e.,
\begin{equation}
\label{2.75}  {\bf f} = S_N ({\bf f}) + \sum_{|n|>N} P_n ({\bf f})
 \quad    \forall  {\bf f} \in L^2 ( [0,\pi] , \mathbb{C}^2),
\end{equation}
where the series converge unconditionally.
\end{Theorem}

Since the projections $P_{n,\alpha} $ in (\ref{2.72}) are
one-dimensional, we obtain the following.
\begin{Corollary}
\label{cor2} If $bc$ is strictly regular, then there exists a Riesz
basis in $L^2([0,\pi], \mathbb{C}^2)$ consisting of eigenfunctions
and at most finitely many associated functions of the Dirac operator
$L_{bc}(v).$
\end{Corollary}

\section{Equiconvergence}

By Theorem~\ref{thm01}, for every regular $bc$ there is a Riesz basis
$\Phi= \{\varphi_k^\nu\}$ in $L^2([0,\pi], \mathbb{C}^2)$ consisting
of root functions of the free Dirac operator $L^0_{bc}.$ In the next
section, we study the point-wise convergence of the $L^2$-expansions
with respect to the basis $\Phi = \Phi (bc) $
\begin{equation}
\label{ec1}  \sum_{k \in 2 \mathbb{Z}}\sum_{\nu =1}^2 \langle F,
\tilde{\varphi}_k^\nu \rangle \varphi_k^\nu.
\end{equation}

By Corollary \ref{cor2}, for strictly regular $bc$  there is a Riesz
basis $\Psi= \{\psi_k^\nu\}$ in $L^2([0,\pi], \mathbb{C}^2)$
consisting of root functions of the Dirac operator $L_{bc}(v).$ The
point-wise convergence of the corresponding  $L^2$-expansions
\begin{equation}
\label{ec2} F=  \sum_{k\in \mathbb{Z}} \sum_{\nu =1}^2 \langle F,
\tilde{\psi}_k^\nu \rangle  \psi_k^\nu.
\end{equation}
is closely related to the point-wise convergence in (\ref{ec1})
because (under some assumptions on $F$ or $v$)
\begin{equation}
\label{ec3} \left \|  \sum_{|k|\leq N} \sum_{\nu =1}^2 \langle F,
\tilde{\psi}_k^\nu \rangle  \psi_k^\nu - \sum_{|k|\leq N} \sum_{\nu
=1}^2 \langle F, \tilde{\varphi}_k^\nu \rangle \varphi_k^\nu \right
\|_\infty \to 0 \quad \text{as} \; N\to \infty.
\end{equation}
Further we refer to (\ref{ec3}) as {\em equiconvergence} of
spectral decompositions (\ref{ec1}) and (\ref{ec2}).

Let $R_{NT} $ be the rectangle defined in (\ref{2.26}) and let $S_N $
and $S_N^0 $ be the corresponding projections defined by
(\ref{2.61}). Then  (\ref{ec3}) can be written as
\begin{equation}
\label{ec4}
 \left \| \left ( S_N - S^0_N \right ) F \right
\|_{\infty} \to 0 \quad \text{as} \quad N \to \infty.
\end{equation}
This form of the equiconvergence statement is suitable for regular
but not strictly regular $bc $ as well.  Of course, $S_N = S_N (v,
bc)$ but for the sake of simplicity the dependence on $v$ and  $bc$
is suppressed in notations. Further, we also write $R_\lambda $
instead of $R_{bc} (\lambda).$

Since
$$
R_\lambda -R^0_\lambda = R^0_\lambda VR^0_\lambda + \sum_{m=2}^\infty
R^0_\lambda (VR^0_\lambda)^m,
$$
we have
\begin{equation}
\label{ec20}
 S_N -S_N^0    = \frac{1}{2 \pi i} \int_{\partial R_{NT}}
(R_\lambda -R^0_\lambda) d\lambda = A_N  + B_N,
\end{equation}
where
\begin{equation}
\label{ec21}
A_N =
\frac{1}{2 \pi i} \int_{\partial R_{NT}}
R^0_\lambda VR^0_\lambda d\lambda
\end{equation}
\begin{equation}
\label{ec22}
B_N  = \frac{1}{2 \pi i} \int_{\partial R_{NT}}
\sum_{m=2}^\infty
R^0_\lambda (VR^0_\lambda)^m
d\lambda.
\end{equation}
Next we show that all restrictions on the class of functions for
which (\ref{ec4}) holds come from analysis of the operators $A_N. $

\begin{Proposition}
\label{propBN}  For every regular $bc,$  $\; L^2$-potential $v$ and
 $  F \in L^2 ([0,\pi], \mathbb{C}^2),$
\begin{equation}
\label{ec24}
\|B_N (F)\|_\infty \to 0 \quad \text{as} \;\;  N \to \infty.
\end{equation}

\end{Proposition}

\begin{proof}
To prove (\ref{ec24}) it is enough to show that

(i)   there is a constant $K>0 $ such that
\begin{equation}
\label{ec25}
   \|B_N \|_{L^2 \to L^\infty}  \leq K \quad \forall \, N\geq N_0;
\end{equation}

(ii)  $\left \|B_N (G)  \right \|_{\infty} \to 0 $ as $N\to \infty$
for functions $G$ in a dense subset of $L^2 ([0,\pi], \mathbb{C}).$\\
Then a standard argument shows that (i) and  (ii) imply (\ref{ec24}).

First we prove (i). The integral in (\ref{ec22}) does not depend on
the choice of $T>T_0 $ because the integrand depends analytically on
$\lambda $ if $|Im \, \lambda | > T_0. $ Therefore, for every $T>T_0,
$
\begin{equation}
\label{ec26}
\|B_N \|_{L^2 \to L^\infty}
\leq \frac{1}{2\pi} \int_{\partial R_{NT}}
\left \|\sum_{m=2}^\infty R^0_\lambda (VR^0_\lambda)^m
\right \|_{L^2 \to L^\infty} d|\lambda|.
\end{equation}

In view of (\ref{1.15}), to estimate the latter integral we need to
find estimates from above of $a(\lambda -\tau_1) + a(\lambda
-\tau_2)$ for $\lambda \in
\partial R_{NT}, $  where $a(\lambda) $ is the function defined in
Lemma~\ref{lema1}.

The boundary $\partial R_{NT} $ consist of four segments $\Delta^-_1,
\Delta_1^+, \Delta^-_2, \Delta_2^+,$
where \\
$ \Delta^\pm_1 = \{\lambda= x+it:  -N-1+ Re \frac{\tau_1 + \tau_2}{2}
\leq x \leq N+1 +Re \frac{\tau_1 + \tau_2}{2}, \; t= \pm T \},
$  \\
$
\Delta^\pm_2 = \{\lambda = x+it: \; x =\pm (N+1) + Re \frac{\tau_1 + \tau_2}{2},
\; -T \leq t \leq T  \}.
$\\
If $\lambda \in \Delta^\pm_1, $ then $ |Im \, \lambda| =T, $ so for
$T>T_0 \geq 1+2|\tau_1| +2|\tau_2|$
$$
|Im \, (\lambda - \tau_\nu ) | \geq T - |\tau_\nu| \geq T/2, \quad
\nu = 1,2.
$$
Then, by Lemma \ref{lema1},
\begin{equation}
\label{ec31} a(\lambda -\tau_1) + a(\lambda -\tau_2) \leq
\frac{2}{(T/2)^2} + \frac{16}{T} \leq \frac{24}{T}, \quad \lambda \in
\Delta^\pm_1.
\end{equation}

If $\lambda \in \Delta^\pm_2, $ then $ \lambda - \tau_\alpha = \pm (N
+1) + (-1)^\alpha Re \, \frac{\tau_1 - \tau_2}{2} + i (t- Im \,
\tau_\alpha),$ $ \alpha =1,2. $ By (\ref{14}), it follows that
$$
\lambda - \tau_\alpha = m_\alpha +\xi_\alpha +
 i (t- Im \, \tau_\alpha),   \quad m_\alpha \in 2\mathbb{Z},
\;\;  1/2 \leq  |\xi_\alpha | \leq 1, \quad \; \alpha =1,2.
$$
Therefore, by Lemma~\ref{lema1} we obtain
\begin{equation}
\label{ec32} a(\lambda -\tau_1) + a(\lambda -\tau_2) \leq h (t),
\quad t = Im \, \lambda, \quad \lambda \in \Delta^\pm_2,
\end{equation}
where
\begin{equation}
\label{ec33}  h (t) := \sum_{\alpha =1}^2 \left ( \frac{1}{1/2+|t- Im
\, \tau_\alpha |^2} + \frac{8}{1+2|t- Im \, \tau_\alpha |} \right ).
\end{equation}

The integrand of the integral in (\ref{ec26}) does not exceed
$$
\left \|\sum_{m=2}^\infty R^0_\lambda (VR^0_\lambda)^m \right \|_{L^2
\to L^\infty} \leq \|( R^0_\lambda VR^0_\lambda VR^0_\lambda\|_{L^2
\to L^\infty} \sum_{m=0}^\infty  \|(VR^0_\lambda)^m \|_{L^2 \to L^2}.
$$
By (\ref{2.22}), $\|VR^0_\lambda\|\leq \|v\| \cdot [a(\lambda
-\tau_1) + a(\lambda -\tau_2)]^{1/2},$  so (\ref{ec31}) and
(\ref{ec32}) yield
$$
\|VR^0_\lambda\|\leq \begin{cases} 5\|v\|/\sqrt{T}   & \text{if} \;\;
\lambda \in \Delta^\pm_1, \\ \\\|v\| \sqrt {h(t)} & \text{if} \;\;
\lambda \in \Delta^\pm_2, \; t = Im \, \lambda.
\end{cases}
$$
By (\ref{ec33}), $h(t) \to 0 $ as $|t|\to \infty. $ Therefore, there
is a constant $C_1= C_1 (bc) > 0 $ such that
\begin{equation}
\label{ec34} \sup \{ \|VR^0_\lambda\|_{L^2 \to L^2} : \; \lambda \in
\partial R_{NT} \} \leq C_1 \|v\|.
\end{equation}
In view of Corollary \ref{cor1} -- see (\ref{2.50}) -- if $T>T_0, $
$N>N_0$ and $\lambda \in \partial R_{NT} $ then $\|VR^0_\lambda
VR^0_\lambda\|_{L^2 \to L^2} \leq 1/2,$ so
\begin{equation}
\label{ec35} \sum_{m=0}^\infty  \|VR^0_\lambda)^m \|_{L^2 \to
L^2}\leq \sum_{s=1}^\infty (1+C_1 \|v||)  \|(VR^0_\lambda
VR^0_\lambda)^s \|_{L^2 \to L^2} \leq 1+C_1 \|v\|.
\end{equation}

On the other hand, (\ref{1.15}) and (\ref{1.17}) imply that
\begin{equation}
\label{ec38} \|( R^0_\lambda VR^0_\lambda VR^0_\lambda\|_{L^2 \to
L^\infty} \leq C \|v\|^2 [a(\lambda -\tau_1) + a(\lambda
-\tau_2)]^{3/2}.
\end{equation}
Therefore, in view of (\ref{ec31}), (\ref{ec32}) and (\ref{ec35}), it
follows that
$$
\left \|\sum_{m=2}^\infty R^0_\lambda (VR^0_\lambda)^m
\right \|_{L^2 \to L^\infty}
\leq \begin{cases}
C_2 (24/T)^{3/2}  & \text{if} \;\; \lambda \in \Delta^\pm_1, \\
C_2 [h (t)]^{3/2} & \text{if} \;\; \lambda \in \Delta^\pm_2, \; t=Im
\, \lambda
\end{cases}
$$
with $C_2 =C(1+C_1 \|v\|) \|v\|^2.$ Now, choosing $T=N$ and taking
into account that $\int_{-\infty}^\infty [h (t)]^{3/2} dt < \infty$
(see (\ref{ec33})), we obtain that the integrals in (\ref{ec26}) are
uniformly bounded. This completes the proof of (i). \vspace{2mm}

Next we prove (ii). Fix a regular $bc,$  and let $\Phi =
\{\varphi_k^\nu \} $ be the corresponding Riesz basis of
eigenfunctions of the free operator $L_{bc}^0 $ given in
Theorem~\ref{thm01}. Since the linear combinations of $\varphi_k^\nu
$ are dense in $L^2 ([0,\pi],\mathbb{C}^2),$  it is enough to prove
(ii) for $G=\varphi_k^\nu. $

Fix $k \in 2\mathbb{Z}$ and $\nu \in \{1,2\}.$ By (\ref{ec22}), we
have for every $T>T_0 $
\begin{equation}
\label{ec40} \| B_N \varphi_k^\nu \|_\infty \leq \frac{1}{2 \pi }
\int_{\partial R_{NT}} \left \|\sum_{m=2}^\infty R^0_\lambda
(VR^0_\lambda)^m \varphi_k^\nu \right \|_\infty d|\lambda|.
\end{equation}
The integrand does not exceed
$$
\left \|\sum_{m=2}^\infty
R^0_\lambda (VR^0_\lambda)^m \varphi_k^\nu \right \|_\infty
\leq \|R^0_\lambda VR^0_\lambda \|_{L^2 \to L^\infty}
 \sum_{m=0}^\infty \|
(VR^0_\lambda)^m  \|_{L^2 \to L^2} \|V R^0_\lambda \varphi_k^\nu  \|.
$$

By (\ref{1.10}), (\ref{1.17})  and  (\ref{ec8}),
$$\|V R^0_\lambda \varphi_k^\nu\| \leq  \|v\|
\cdot \| R^0_\lambda \varphi_k^\nu \|_\infty \leq
\frac{C\|v\|}{|\lambda - k- \tau_\nu |}.
$$
If $ \lambda \in \Delta^\pm_1, $ then $|Im \, \lambda|= T $ so for
large enough $T$  it follows $|\lambda - k- \tau_\nu | \geq T/2.$ If
$ \lambda \in \Delta^\pm_2, $ then $ \lambda =\pm (N+1) + Re
\frac{\tau_1 + \tau_2}{2} + it, $ so for large enough $N$ we obtain
$$  |\lambda - k- \tau_\nu | \geq \frac{1}{\sqrt{2}} \left (
N+1 - |k| - \left | Re \frac{\tau_1 - \tau_2}{2} \right | \right ) +
\frac{1}{\sqrt{2}} | t- Im \, \tau_\nu | \geq \frac{N+|t|}{2}. $$
Therefore, for large enough $N,$
$$\|V R^0_\lambda \varphi_k^\nu \| \leq
\begin{cases}
\frac{2C \|v\|}{T}        \quad  \text{for} \;\;  \lambda \in \Delta^\pm_1,\\
\frac{2C\|v\|}{N+|t|}     \quad  \text{for} \;\;  \lambda \in
\Delta^\pm_2.
\end{cases}
$$
On the other hand,
 (\ref{1.15}) and (\ref{1.17}) imply that
$$
\|( R^0_\lambda VR^0_\lambda \|_{L^2 \to L^\infty} \leq
C \|v\|
[a(\lambda -\tau_1) + a(\lambda -\tau_2)].
$$
Therefore, by (\ref{ec31}), (\ref{ec32}) and (\ref{ec35}) it follows
that
$$
\left \|\sum_{m=2}^\infty R^0_\lambda (VR^0_\lambda)^m \varphi_k^\nu
\right \|_\infty
\leq \begin{cases}
C_3 T^{-2}  & \text{if} \;\; \lambda  \in \Delta^\pm_1, \\
 C_3 \frac{h (t)}{N+|t|}
& \text{if} \;\; \lambda \in \Delta^\pm_2, \; Im \, \lambda =t,
\end{cases}
$$
where $C_3 = C_3 (\|v\|, bc). $  The integral in (\ref{ec40}) is sum
of integrals on $\Delta^\pm_1 $ and $\Delta^\pm_2.$ In view of the
latter estimates, if we choose $T=N, $ then the integrals over
$\Delta^\pm_1 $ go to zero as $N \to \infty.$ On the other hand, by
(\ref{ec33}) $ h (t) \asymp \frac{1}{1+|t|} ,$ so the integrals over
$\Delta^\pm_2 $ do not exceed a multiple of
$$
\int_{-\infty}^\infty \frac{1}{(1+|t|)(N+|t|) } dt = 2\frac{\log
N}{N-1} \to 0 \quad \text{as} \;\; N \to \infty.
$$
This completes the proof.
\end{proof}

\begin{Corollary}
\label{corec}
In the above notations, for every
regular $bc$, $\; L^2$-potential $v$ and $F\in L^2([0,\pi], \mathbb{C})^2)$
\begin{equation}
\label{ec100}  \lim_N \left \| \left ( S_N - S^0_N \right ) F \right
\|_{\infty} = 0 \quad \quad \Longleftrightarrow  \quad  \lim_N  \|
A_N F \|_{\infty} = 0.
\end{equation}

\end{Corollary}

Next we give conditions on $v$ or $F$ that guarantee the existence of
the right-hand limit in (\ref{ec100}). Let  us fix a regular $bc,$
and let $\Phi = \{\varphi_k^\nu \} $ be the corresponding Riesz basis
of eigenfunctions of the free operator $L_{bc}^0 $ given in
Theorem~\ref{thm01}.

\begin{Lemma}
\label{lema3}
 Under the above assumptions, for every $\varphi_k^\nu
\in \Phi $
\begin{equation}
\label{ec101} \|A_N \varphi_k^\nu\|_{\infty} \to 0 \quad \text{as}
\;\; N \to \infty.
\end{equation}
\end{Lemma}

\begin{proof}
Fix $ \varphi_k^\nu $ and consider $N>|k|.$ By (\ref{1.10}) and
Lemma~\ref{lem10},
$$
(A_N \varphi_k^\nu)(x) = \frac{1}{2\pi i} \int_{\partial R_{NT}}
 \sum_{\eta=1}^2 \sum_{j\in \mathbb{Z}} \frac{w^{\eta \nu}(j+k)}
{(\lambda - j -\tau_\eta) (\lambda - k -\tau_\nu)} \varphi^\eta_j (x)
d\lambda,
$$
so by the Residue Theorem
$$
(A_N \varphi_k^\nu)(x)  = \sum_{\eta=1}^2 \sum_{|j|>N} \frac{w^{\eta
\nu}(j+k) }{(k+\tau_\nu ) - (j+ \tau_\eta )} \varphi^\eta_j (x). $$

If $j, k \in 2\mathbb{Z}, \; j\neq k, $ then
\begin{equation}
\label{ec11} |k-j+\tau_\nu - \tau_\eta| \geq |k-j| - |Re (\tau_\nu -
\tau_\eta)| \geq |k-j| -1 \geq \frac{|k-j|}{2}.
\end{equation}
Therefore, in view of (\ref{ec8}) and (\ref{1.18}), it follows that
$$
\|A_N \varphi_k^\nu \|_\infty \leq  4C\sum_{|j|>N} \frac{r(j+k)}{|k-j
|}.
$$
Thus,  the Cauchy inequality implies
$$
\|A_N \varphi_k^\nu\|_\infty \leq 4C \|r\| \left ( \sum_{|j|>N}
\frac{1}{|k-j|^2} \right )^{1/2} \leq \frac{4C \|r\|}{(N-|k|)^{1/2}}
\to 0
$$
as $N\to \infty, $ which completes the proof.
\end{proof}

Fix an $L^2$-function $F: [0,\pi] \to \mathbb{C}^2 $ and consider
\begin{equation}
\label{ec5}
(A_N F)(x) =  \frac{1}{2 \pi i} \int_{\partial R_{NT}}
R^0_\lambda VR^0_\lambda F  d\lambda,
\end{equation}
Let
$
 F = \sum_{\nu =1}^2\sum_{k\in 2\mathbb{Z}} F^\nu_k \varphi^\nu_k
$ be the expansion of $F$ about the basis $\{\varphi^\nu_k \}.$ By
the matrix representation of the operators  $V$ and $R^0_{bc} $ it
follows that
\begin{equation}
\label{ec80} (A_N F)(x) = \frac{1}{2\pi i} \int_{\partial R_{NT}}
\sum_{\nu, \eta=1}^2 \sum_{k\in 2\mathbb{Z}} \sum_{j\in 2\mathbb{Z}}
\frac{w^{\eta \nu}(j+k) F^\nu_k} {(\lambda - j -\tau_\eta) (\lambda -
k -\tau_\nu)} \varphi^\eta_j (x) d\lambda.
\end{equation}
The Residue Theorem implies
$$
(A_N F)(x) = \sum_{\nu, \eta=1}^2 \sum_{ |k|\leq N} \sum_{|j|>N}
\frac{w^{\eta \nu}(j+k)  F^\nu_k}{(k+\tau_\nu ) - (j+ \tau_\eta )}
\varphi^\eta_j (x) $$ $$ + \sum_{\nu, \eta=1}^2 \sum_{|k|> N}
\sum_{|j|\leq N} \frac{w^{\eta \nu}(j+k)  F^\nu_k}{(j+ \tau_\eta
)-(k+\tau_\nu )} \varphi^\eta_j (x).
$$
We set
\begin{equation}
\label{ec9}
g_m = \max \{\left |F^1_m \right |, \left |F^2_m \right |\};
\end{equation}
then $g=(g_m) \in \ell^2(2\mathbb{Z}) $ and $\|g\| \leq  const \cdot
\|F\|. $   Therefore, in view of  (\ref{ec8}), (\ref{1.18})  and
(\ref{ec9}), it follows that
\begin{equation}
\label{ec9a}
  \|(A_N F)\|_\infty \leq 8C (\sigma_1 (N) + \sigma_2 (N)),
\end{equation}
  where
\begin{equation}
\label{ec12} \sigma_1 (N) = \sum_{|k|\leq N} \sum_{|j|>N} \left
|\frac{r(j+k)g_k}{j-k} \right |, \qquad \sigma_2 (N)  = \sum_{|k|> N}
\sum_{|j|\leq N} \left | \frac{r(j+k)g_k}{k-j} \right |,
\end{equation}
and $j, k , N $ are even numbers.

\begin{Lemma}
\label{lema2}
 (a)  If $ g= (g_k) \in \ell^p
(2\mathbb{Z}), \; p\in (1,2) $ and $ r = (r(k)) \in \ell^2
(2\mathbb{Z}),$ then
\begin{equation}
\label{ec12a}  \sigma_\mu (N) \leq C(p) \|r\| \|g\|_p, \quad \mu
=1,2.
\end{equation}

(b) If $g= (g_k) \in \ell^2 (2\mathbb{Z})$ and  $ r = (r(k)) \in
\ell^p (2\mathbb{Z}), \; p\in (1,2),$  then
\begin{equation}
\label{ec12b}  \sigma_\mu (N) \leq C(p) \|r\|_p \|g\|, \quad \mu
=1,2.
\end{equation}

 (c) If $  \exists \delta > 1 : \; |g|^2_\delta :
 =\sum_k |g_k|^2 [\log (|k|+e)]^\delta  <
\infty$ and $r \in \ell^2 (2\mathbb{Z}),$ then
\begin{equation}
\label{ec12c}  \sigma_\mu (N) \leq  C(\delta) \|r\| |g|_\delta, \quad
\mu =1,2.
\end{equation}

(d) If $  \exists \delta >1 : \; |r|^2_\delta :
 =\sum_k |r(k)|^2 [\log (|k|+e)]^\delta  <
\infty$ and $g \in \ell^2 (2\mathbb{Z}),$ then
\begin{equation}
\label{ec12d}  \sigma_\mu (N) \leq  C(\delta) |r|_\delta \|g\|, \quad
\mu =1,2.
\end{equation}

\end{Lemma}

\begin{proof}
Throughout the proof $j,k\in 2\mathbb{Z} $ and $N \in 2\mathbb{N}.$
We will use the following inequalities: if $|k|\leq N $ then for $s
\geq 2$
\begin{equation}
\label{ec18} \sum_{|j|>N} \frac{1}{|j-k|^s} \leq
\int_{N+1-|k|}^\infty  \frac{1}{x^s} dx \leq
\frac{1}{(N+1-|k|)^{s-1}},
\end{equation}
and if $|k|>N $ then for $s \geq 2$
\begin{equation}
\label{ec18a} \sum_{|j|\leq N} \frac{1}{|k-j|^s} \leq
\int_{|k|-N-1}^\infty  \frac{1}{x^s} dx \leq
\frac{1}{(|k|-N-1)^{s-1}}.
\end{equation}

Suppose $|k|\leq N. $ In case (a), the Cauchy inequality and
(\ref{ec18}) with $s=2$ imply
$$
 \sum_{|j|>N} \left | \frac{r(j+k)}{j-k} \right |\leq
  \|r\|  \left ( \sum_{|j|>N}
\frac{1}{(j-k)^2} \right )^{1/2} \leq \|r\| \,
\frac{1}{(N+1-|k|)^{1/2}}.
$$
Therefore, by the H\"older inequality (with $q=p/(p-1)>2),$ it
follows
$$
\sigma_1 (N) \leq    \|r\| \sum_{|k|\leq N}
\frac{|g_k|}{(N+1-|k|)^{1/2}} \leq $$ $$\|r\| \left ( \sum_{|k|\leq
N}|g_k|^p \right )^{1/p}  \left ( \sum_{|k|\leq N} (N+1-|k|)^{-q/2}
\right )^{1/q} \leq  C(p) \|r\| \|g\|_p,
$$
where $C(p) =\left ( \sum_{m=1}^\infty  2m^{-q/2}\right )^{1/q}, \;
q= (p-1)/p. $

In case (b),  the H\"older inequality (with $q >2$) and (\ref{ec18})
with $s=q$ imply
$$
\sigma_1 (N) \leq \sum_{|k|\leq N}
  |g_k|\, \|r\|_p
 \left ( \sum_{|j|>N}
\frac{1}{(j-k)^q} \right )^{1/q} \leq
 \sum_{|k|\leq N}
  \frac{|g_k|\, \|r\|_p}{(N+1-|k|)^{1-1/q}}.$$
By the Cauchy inequality, we obtain
$$
\sigma_1 (N) \leq \|r\|_p \|g\| \left (  \sum_{m=1}^\infty
\frac{2}{m^{2-2/q}}  \right )^{1/2} \leq C(p)\|r\|_p \|g\|
$$
with $C(p) =\left (  \sum_{m=1}^\infty \frac{2}{m^{2/p}}  \right
)^{1/2}. $

In case (c), the Cauchy inequality and (\ref{ec18}) with $s=2 $ imply
$$
\sigma_1 (N) \leq \|r\|\sum_{|k|\leq N}
\frac{|g_k|}{(N+1-|k|)^{1/2}}.
$$
If $|k|\leq N/2, $ then $ N+1 -|k| \geq N+1-N/2 =(N+2)/2, $ so
applying again the Cauchy inequality we obtain
$$
\sum_{|k|\leq N/2} \frac{|g_k|}{(N+1-|k|)^{1/2}} \leq  \|g\| \left (
\sum_{|k|\leq N/2} \frac{1}{N+1-|k|}   \right )^{1/2} \leq \|g\|.
$$
On the other hand, if $|k|>N/2 $ then $|k|\geq N+2-|k|,$ so
$$
\sum_{N/2<|k|\leq N} \frac{|g_k|}{(N+1-|k|)^{1/2}} \leq
\sum_{N/2<|k|\leq N} \frac{|g_k| [\log (e+|k|)]^{\delta/2}
}{(N+1-|k|)^{1/2} [\log (N+2 -|k|)]^{\delta/2} }
$$
$$
\leq |g|_\delta  \left ( \sum_{N/2<|k|\leq N} \frac{1}{(N+1-|k|)[\log
(N+2 -|k|)]^{\delta}} \right )^{1/2} \leq C_1 (\delta) |g|_\delta
$$
with $C_1 (\delta) =\left (\sum_{m=1}^\infty \frac{2}{m (\log
(m+1))^\delta}  \right )^{1/2}.$ Since $\|g\| \leq |g|_\delta, $  it
follows that
$$ \sigma_1 (N) \leq C (\delta) \|r\| |g|_\delta \quad \text{with} \;\;
   C (\delta)= 1+C_1 (\delta).  $$

In case (d), the Cauchy inequality implies
\begin{equation}
\label{ec17}
 \sum_{|j|>N} \left | \frac{r(j+k)}{j-k} \right |\leq
  \mathcal{E}_{N+2-|k|}(r)  \left ( \sum_{|j|>N}
\frac{1}{(j-k)^2} \right )^{1/2},
\end{equation}
where $\mathcal{E}_m (r)= \left ( \sum_{|i|\geq m} |r(i)|^2 \right
)^{1/2}. $ Since
$$
\left ( \mathcal{E}_{m} (r)\right )^2 \leq \frac{1}{(\log m)^\delta }
\sum_{|k|\geq m} |r(k)|^2 [\log (e+|k|)]^\delta \leq
\frac{|r|_\delta^2}{(\log m)^\delta },
$$
in view of (\ref{ec17}) and (\ref{ec18}) with $s=2 $ it follows
$$
\sigma_1 (N) \leq  |r|_\delta \sum_{|k|\leq N}
\frac{|g_k|}{(N+1-|k|)^{1/2} (\log (N+2-|k|))^{\delta/2} }
$$
$$
\leq  |r|_\delta \|g\| \cdot  \left (\sum_{m=1}^\infty \frac{2}{m
(\log (m+1))^\delta}  \right )^{1/2}  \leq C(\delta) |r|_\delta \|g\|
$$
with $C(\delta) =\left (\sum_{m=1}^\infty \frac{2}{m (\log
(m+1))^\delta}  \right )^{1/2}.$

This completes the proof of (\ref{ec12a})--(\ref{ec12d})  for
$\sigma_1 (N). $ The proof in the case of sums $\sigma_2 (N) $  is
essentially the same -- but then $|k| \geq N+2 $ and $|j| \leq N, $
so one has to use (\ref{ec18a}) instead of (\ref{ec18}) and replace
in all formulas  $N+1-|k| $ by $|k|-N-1. $ For example, in case (c),
the Cauchy inequality and (\ref{ec18a}) with $s=2 $ imply
$$
\sigma_2 (N) \leq \|r\|\sum_{|k|> N} \frac{|g_k|}{(|k|-N-1)^{1/2}}.
$$
Therefore, again by the Cauchy inequality, it follows

$$
\sigma_2 (N) \leq \|r\|\sum_{|k|> N} \frac{|g_k| [\log
(e+|k|)]^{\delta/2} }{(|k|-N-1)^{1/2} [\log (|k|-N)]^{\delta/2} }
$$
$$
\leq \|r\| \,|g|_\delta  \left ( \sum_{|k|> N}
\frac{1}{(|k|-N-1)^{1/2} [\log (|k|-N)]^{\delta/2} } \right )^{1/2}
\leq C_1 (\delta) \|r\| \, |g|_\delta
$$
with $C_1 (\delta) =\left (\sum_{m=1}^\infty \frac{2}{m (\log
(m+1))^\delta}  \right )^{1/2}.$ We omit the details about cases (a),
(b) and (d).

\end{proof}

\begin{Proposition}
\label{propAN} Given a regular $bc,$ $L^2$-potential $v$ and $F \in
L^2([0,\pi], \mathbb{C}^2), $ let $g=(g_k)_{k\in 2\mathbb{Z}}$ be
defined by (\ref{ec9}) and let $r=(r(k))_{k\in 2\mathbb{Z}}$ be
defined by Lemma~\ref{lem10} and (\ref{1.18}). Then
\begin{equation} \label{ec54}
\|A_N (F)\|_\infty \to 0 \quad \text{as} \;\;  N \to \infty
\end{equation}
whenever one of the following conditions holds:

(a) $\exists p\in (1,2) $ such that $ (g_k) \in \ell^p
(2\mathbb{Z}),$ where $(g_k) $ is defined by (\ref{ec9});

(b) $\exists p\in (1,2)$ such that $ (r(k)) \in \ell^p
(2\mathbb{Z}),$ where $(r(k)) $ is defined by (\ref{1.18});

(c) $\exists \delta >1 $ such that $ \sum_{k} |g_k|^2 [\log (
|k|+e)]^\delta < \infty;$

(d) $\exists \delta >1 $ such that $ \sum_{k} |r(k)|^2 [\log (
|k|+e)]^\delta < \infty.$

Moreover, each of the conditions (a), (b), (c) and (d) guarantees
that
$$
\|(S_N - S_N^0) F\|_\infty \to 0 \quad \text{as} \;\;  N \to \infty.
$$
\end{Proposition}

\begin{proof}
Suppose $F \in L^2([0,\pi], \mathbb{C}^2) $  and $F= \sum_{k\in
2\mathbb{Z}} \sum_{\nu =1}^2 F^\nu_k \varphi^\nu_k $ is the expansion
of $F$ about the basis $\Phi. $ Let (a) holds (with $g=(g_k)$ defined
by $g_k = \max(|F_k^1|, |F_k^2|),$ i.e., by (\ref{ec9})).

 Fix $\varepsilon >0 $ and choose  an $N_1 \in 2\mathbb{Z}$ such
that $\sum_{|k|>N_1 } |g_k|^p < \varepsilon^p. $ Set
$$\tilde{F} = \sum_{|k|>N_1} \sum_{\nu =1}^2 F^\nu_k \varphi^\nu_k,
\quad \tilde{g}=(\tilde{g}_k), \; \tilde{g}_k =0 \;\text{if} \;
|k|\leq N_1, \;\; \tilde{g}_k = g_k \; \text{if} \; |k|>N_1.
$$
By (\ref{ec9a}) and Lemma \ref{lema2}(a), it follows that
$$
\|A_N \tilde{F} \|_\infty \leq 16 C C(p)\|r\| \|\tilde{g}\|_p \leq 16
C C(p)\|r\| \cdot\varepsilon.
$$
On the other hand, since $F-\tilde{F} $ is a finite linear
combination of basis functions $\varphi_k^\nu, $  Lemma~\ref{lema3}
implies that
$$
\|A_N (F-\tilde{F} \|_\infty \to 0 \quad \text{as} \;\; n \to \infty.
$$
Therefore,  $ \limsup \|A_N F \|_\infty \leq 16 C C(p)\|r\|
\cdot\varepsilon $  for every $\varepsilon >0, $ so (\ref{ec54})
holds.

The proof is similar in the cases (b), (c) and (d) -- we use
Lemma~\ref{lema3} and, respectively, parts (b), (c) and (d) of
Lemma~\ref{lema2}.

Of course, in view of Proposition~\ref{propBN}, each of the
conditions (a), (b), (c)  and (d) implies that $ \lim_N \|(S_N -
S_N^0) F\|_\infty=0. $
\end{proof}

For a given regular $bc,$ Proposition~\ref{propAN} gives sufficient
conditions for equiconvergence in terms of matrix representation of
the potential $v=\begin{pmatrix} 0&P\\Q &0  \end{pmatrix}$ and
coefficients of the expansion of $F=\begin{pmatrix} F_1\\F_2
\end{pmatrix}$ about the basis $\Phi $ (which consists of root
functions of $L_{bc}^0 $).

In particular, for periodic ($Per^+ $) or antiperiodic ($Per^- $)
boundary conditions,
$$
Per^\pm: \;\; y_1 (\pi) =\pm y_1 (0), \quad  y_2 (\pi) =\pm y_2 (0),
$$
we may consider the following bases of eigenfunctions:
$$
\Phi_{Per^+} = \left \{ \varphi^1_k= \begin{pmatrix}  e^{-ikx}  \\ 0
\end{pmatrix}, \; \varphi^2_k=\begin{pmatrix} 0 \\ e^{ikx}
\end{pmatrix}, \; \; k \in 2\mathbb{Z} \right \},
$$$$
\Phi_{Per^-} = \left \{ \varphi^1_k=\begin{pmatrix} e^{-ikx}  \\ 0
\end{pmatrix}, \; \varphi^2_k=\begin{pmatrix} 0 \\ e^{ikx}
\end{pmatrix}, \; \; k \in 1+2\mathbb{Z} \right \}.
$$
Now the matrix representation of the operator of multiplication by
$v$ is
$$
V \sim \left [ \begin{array}{cc} V^{11} & V^{12}\\V^{21}&V^{22}
\end{array} \right ], \quad  V_{jk}^{11}=V_{jk}^{22}=0, \;\;
V_{jk}^{12}=p(-j-k), \; V_{jk}^{21}=q(j+k),
$$
where $p(m)$ and $q(m), \; m\in 2\mathbb{Z}$ are, respectively,  the
Fourier coefficients of the functions $P$ and $Q$ about the system
$\{e^{imx},\, m\in 2\mathbb{Z}\},$  and the corresponding sequence
$(r(m)) $ (compare with Lemma~\ref{lem10} and (\ref{1.18})) is
$$ r(m) = \max (|p(-m)|, |q(m)|), \quad m\in 2\mathbb{Z}. $$
Therefore, the following holds (compare with \cite[Prop. 7.3]{Mit04}).
\begin{Corollary}
Let $v=\begin{pmatrix} 0&P\\Q &0  \end{pmatrix}.$ Suppose there is
$\delta
>1 $ such that $$ \sum_{m \in 2\mathbb{Z}} (|p(m)|^2+ |q(m)|^2 (\log
|k|)^\delta < \infty,$$ where $p(m)$ and $q(m)$
 are, respectively,  the Fourier coefficients of
the functions $P$ and $Q$ about the system $\{e^{imx},\, m\in
2\mathbb{Z}\}.$ Then we have, for periodic $Per^+ $ or antiperiodic
$Per^- $ boundary conditions,
$$
\|(S_N -S_N^0)F\|_\infty \to 0 \quad \text{as} \;\; N \to \infty
\quad \forall F \in L^2 ([0,\pi], \mathbb{C}^2).
$$
\end{Corollary}
However, one needs to impose on $F$  conditions depending on $bc $ to
guarantee equiconvergence for every $L^2$-potential $v.$ For example,
if $bc= Per^\pm,$ Proposition~\ref{propAN} implies the following.
\begin{Corollary}
Suppose $v$ is an $L^2$-potential matrix. Then
$$
\|(S_N -S_N^0)F\|_\infty \to 0 \quad \text{as} \;\; N \to \infty
$$
holds

 (i) for $bc=Per^+, $ if there is $ \delta >1 $ such that $
\sum_{m \in 2\mathbb{Z}} (F_{1,m}|^2+ |F_{1,m}|^2 [\log(e+
|m|)]^\delta < \infty,$ where  $F_{1,m}$ and $F_{2,m}$
 are, respectively,  the Fourier coefficients of
the functions $F_1$ and $F_2$ about the system $\{e^{imx},\, m\in
2\mathbb{Z}\}; $

(ii) for $bc=Per^-, $ if there is $\delta >1 $ such that $ \sum_{m
\in 2\mathbb{Z}} (F_{1,m}|^2+ |F_{1,m}|^2 [\log(e+ |m|)]^\delta <
\infty,$ where  $F_{1,m}$ and $F_{2,m}$
 are, respectively,  the Fourier coefficients of
the functions $F_1$ and $F_2$ about the system $\{e^{imx},\, m\in 1+
2\mathbb{Z}\}.$

\end{Corollary}

Next we discuss what conditions guarantee equiconvergence property
simultaneously for all regular $bc.$

Recall that if $\Omega= (\Omega (k))_{k \in 2\mathbb{Z}}  $ is a
sequence of positive numbers (weight sequence), one may consider the
weighted sequence space
$$
\ell^2 (\Omega,2\mathbb{Z})= \left \{x=(x_k): \; \sum_{k \in
2\mathbb{Z}} (|x_k|\Omega (k))^2 < \infty \right \}
$$
and the corresponding Sobolev space
\begin{equation}
\label{sob}
H(\Omega) =\left  \{ f = \sum_{k \in 2\mathbb{Z}} f_k e^{ikx} : \;\;
(f_k) \in \ell^2 (\Omega) \right \}.
\end{equation}
In particular, consider the Sobolev weights
\begin{equation}
\label{sob1}
 \Omega_\alpha (k) = (1+k^2)^{\alpha/2}, \quad k
\in 2\mathbb{Z},
\end{equation}
and the logarithmic weights
\begin{equation}
\label{sob2}
 \omega_\beta (k) = (\log(e+|k|))^\beta, \quad k
\in 2\mathbb{Z}, \quad \beta \in \mathbb{R}.
\end{equation}

Let $H^\alpha $ and $h^\beta $ denote the corresponding
Sobolev spaces (\ref{sob}). Of course, $H^\alpha \subset h^\beta $
if $\alpha >0 $  and    $h^\beta \subset H^\alpha $
if $\alpha <0   $   for any $\beta.$

\begin{Lemma}
\label{lemsob}
Let $g \in C^1 ([0,\pi]).$

(a)  If $f\in H^\alpha,  \;  - 1/2 < \alpha < 1/2,   $
then $f \cdot g \in H^\alpha. $

(b)  If $f\in h^\beta,  \;  - \infty < \beta < \infty,   $
then $f \cdot g \in h^\beta . $

\end{Lemma}

Proof is given in  Appendix.

\begin{Theorem}
\label{EC} For regular $bc, $  Dirac potentials $v = \begin{pmatrix}
0 & P\\ Q &0 \end{pmatrix} $ with $P, Q \in L^2 ([0,\pi])$ and $F=
\begin{pmatrix} F_1\\ F_2 \end{pmatrix} $ with $F_1, F_2  \in L^2([0,\pi],$
\begin{equation}
\label{ec10} \left \| \left ( S_N - S^0_N \right ) F \right
\|_{\infty} \to 0 \quad \text{as} \quad N \to \infty
\end{equation}
whenever one of the following conditions is satisfied:

(a) $\exists \beta >1 $ such that $$ \sum_{k\in 2\mathbb{Z}}
(|F_{1,k}|^2 +|F_{2,k}|^2) (\log(e+|k|))^\beta  < \infty, $$ where $
(F_{1,k})_{k\in 2\mathbb{Z}} $ and $(F_{2,k})_{k\in 2\mathbb{Z}}$
are, respectively, the Fourier coefficients of $F_1 $ and $F_2$ about
the system $\{e^{ikx},\, k\in 2\mathbb{Z}\};$

(b) $\exists \beta >1 $ such that $$ \sum_{k\in 2\mathbb{Z}}
(|p(k)|^2 +|q(k)|^2) (\log(e+|k|))^\beta < \infty, $$ where $ (p(k))_{k\in
2\mathbb{Z}} $ and $(q(k))_{k\in 2\mathbb{Z}}$ are, respectively, the
Fourier coefficients of $P $ and $Q$ about the system $\{e^{ikx},\,
k\in 2\mathbb{Z}\}.$

In particular, if  $F_1,\, F_2 $ are functions of bounded variation
or  $P, Q $ are functions of bounded variation, then (\ref{ec10})
holds.
\end{Theorem}

\begin{proof}
Suppose a regular $bc $ is fixed. Let $\Phi = (\varphi_k^\nu) $ and
$\tilde{\Phi} = (\tilde{\varphi}_k^\nu) $ be the corresponding Riesz
basis of eigenvectors of $L_{bc}^0$  and its biorthogonal system
constructed in Theorem~\ref{thm1}.

Suppose (a) holds for a function $F.$ In view of the explicit
formulas (\ref{21*}) and (\ref{91}) for the biorthogonal system
$\tilde{\Phi}$ it follows that the expansion coefficients  $ F_k^\nu
= \langle F, \tilde{\varphi}_k^\nu  \rangle $ can be represented as
$$
F^\nu_k = f^\nu_{1, -k} + f^\nu_{2,k}, \quad k \in 2\mathbb{Z}, \quad
\nu =1,2,
$$
where $ f^\nu_{1,k} $ and $ f^\nu_{2,k} $ are the Fourier
coefficients of functions of the form
$$
f^\nu_1 (x)  = h^\nu_1 (x) F_1 (x) + h^\nu_2 (x) F_2 (x), \quad \nu
=1,2,
$$
with $h^\nu_1 (x), h^\nu_2 (x) \in C^\infty([0,\pi]). $ Now
Lemma~\ref{lemsob}(b) implies that
$$ \sum_{k\in 2\mathbb{Z}}
(|F^1_k|^2 +|F^2_k|^2) (\log(e+|k|))^\beta  < \infty, $$ but then Condition (c)
of Proposition~\ref{propAN} holds, hence (\ref{ec10}) follows.

Suppose (b) holds with some $\alpha \in (0,1/2).$ In view of
Lemma~\ref{lem10} and its proof, the sequences $(w^{\eta \nu} (m))_{
 m \in 2\mathbb{Z}}, $ which generate the matrix representation of
the operator $V$ (see (\ref{m2}) and (\ref{m3})) are given by
(\ref{m4}) in terms of the Fourier coefficients of some products of
$P$ and $ Q$ by $C^\infty $-functions (see (\ref{m5})). Therefore, by
Lemma~\ref{lemsob} we have $(w^{\eta \nu} (m)) \in \ell^2
(\Omega_\alpha).$ In view of (\ref{1.18}), this implies that
Condition (d) in Proposition~\ref{propAN} holds, hence (\ref{ec10})
follows.

It is well-known (see \cite[Ch. 2, Sect. 4, Theorem 4.12]{Z}), that
if $f: [0,\pi] $ is a function of
bounded variation then its Fourier coefficients $f_k = \frac{1}{\pi}
\int_0^\pi f(x) e^{-ikx} dx $  satisfy $|f_k| \leq C/|k|, \; k\neq 0,
$ where $C=C(f).$ Therefore, if $F_1, \, F_2 $ are functions of
bounded variation then (a) holds, and if $P, \, Q $ are functions of
bounded variation then (b) holds, so in both cases (\ref{ec10})
follows, which completes the proof.

Of course, one can handle the case of functions of bounded variation
directly, without using Lemma~\ref{lemsob}. Indeed, the matrix
representation coefficients of $V$ and the expansion coefficients of
$F$ about the basis $\Phi = \{\varphi_k^\nu\} $ are coming from the
Fourier coefficients of products of  $ P, Q $ and $F_1, F_2 $ with
some smooth functions. Since a product of a function of bounded
variation with a smooth function is also a function of bounded
variation, the corresponding sequences $(r(m)) $ and $(g_k)$ are
dominated by $const/|k|,$ so they are in the space
$\ell(\Omega_\alpha) $ for $ \alpha \in (0,1/2).$  Thus,
respectively, (c) or (d) in Proposition~\ref{propAN} holds, so the
claim follows.

\end{proof}

\section{Point-wise convergence of spectral decompositions }

 It is well-known that if $f$  is a function of bounded
variation on $[0,\pi]$ then its Fourier series with respect to the
system $\{e^{imx}, \; m\in 2\mathbb{Z}\}$ converges  point-wise to
$\frac{1}{2} [f(x-0) + f(x+0)] $  for $x \in (0,\pi),$  and to
$\frac{1}{2} [f(\pi-0) + f(0+0)] $  for $x=0, \pi.$  More precisely,
the following holds.

\begin{Lemma}
\label{CT}(see \cite[Vol 1, Theorem 8.14]{Z}) If $f: [0,\pi] \to
\mathbb{C} $   is a function of bounded variation, then
\begin{equation}
\label{c1}
\lim_{M \to \infty}  \sum_{m=-M}^M   \langle  f(x),  e^{imx} \rangle e^{imx}
  = \begin{cases} \frac{1}{2} [f(x-0) + f(x+0)]    &
 \text{if}  \; \; x \in (0,\pi),  \\ \frac{1}{2} [f(\pi-0) + f(0+0)]  &
 \text{if}  \;  \; x=0, \pi.  \end{cases}
\end{equation}
Moreover, the convergence is uniform on every closed subinterval of $
(0,\pi)$ on which $f$ is continuous, and the convergence is uniform on the closed interval $[0, \pi]$ if and only if
$f$ is continuous and $f(0) = f(\pi).$
\end{Lemma}

For systems of o.d.e., an interesting point-wise
convergence statement has been proven in \cite{BL23}, pp. 127--128,
but under rather  restrictive assumptions.
For example, the matrices $W_a $
and $W_b$ (see \cite{BL23}, lines 5--6 on p. 64 and
Formula (7) there, or Formula (46) on p. 87) are assumed to be invertible
but this never happens in the case of separated $bc$ like
(\ref{s3}) in Section 7 below.
Another strong  assumption is that the spectrum  $Sp (L_{bc})$ is
eventually simple (see \cite{BL23}, lines 9--11 on p. 98). In
general, this assumption is not satisfied for regular but not
strictly regular $bc.$

We don't impose any assumption on the boundary conditions but
regularity. The main result of this section is the following.

\begin{Theorem}
 \label{GCT}
Let $bc $  be a regular  boundary  condition given by (\ref{8a}), and let  $\Phi= \{\varphi^\nu_k,\, k\in 2\mathbb{Z}, \,
\nu =1,2\} $ and $\tilde{\Phi}= \{\tilde{\varphi}^\nu_k, \, k\in
2\mathbb{Z}, \, \nu =1,2\} $ be the corresponding Riesz  basis of
root functions of $L^0_{bc} $ and its biorthogonal system which are
constructed in Theorem~\ref{thm01}.  If $f,g: [0,\pi] \to \mathbb{C}
$ are functions of bounded variation which are continuous at 0 and
$\pi,$ then
\begin{equation}
\label{c2} \lim_{M \to \infty}  \sum_{m=-M}^M  \sum_{\nu=1}^2 \left
\langle
\begin{pmatrix}  f\\g  \end{pmatrix}, \tilde{\varphi}_m^\nu
\right \rangle \varphi_m^\nu (x)  =
\begin{pmatrix}  \tilde{f}(x) \\ \tilde{g}(x)     \end{pmatrix},
\end{equation}
where
 \begin{equation}
\label{c03}
 \begin{pmatrix}   \tilde{f} (x) \\ \tilde{g} (x) \end{pmatrix}
=\frac{1}{2} \begin{pmatrix} f(x-0) + f(x+0) \\ g(x-0) + g(x+0)
\end{pmatrix}   \quad   \text{for}  \; \; x \in (0,\pi)
\end{equation}
and
\begin{equation}
\label{c3}
\begin{pmatrix}  \tilde{f}(x) \\ \tilde{g}(x)     \end{pmatrix} =
\begin{cases}
{\displaystyle \frac{1}{2}} \begin{pmatrix} f(0) -b f(\pi) -a g(0) \\
   \frac{d}{bc-ad} f(0) + g(0) -\frac{b}{bc-ad} g(\pi)  \end{pmatrix}
 &  \text{if} \;\; x=0, \vspace{1mm}  \\
{\displaystyle \frac{1}{2}} \begin{pmatrix}  -\frac{c}{bc-ad}
    f(0)+f(\pi)+ \frac{a}{bc-ad} g(\pi) \\ -d f(\pi)  - c g(0) +g(\pi)    \end{pmatrix}
 &  \text{if} \;\; x=\pi.
     \end{cases}
\end{equation}
Moreover, if both $f(\pi -t)$ and $g(t)$ are  continuous on some
closed subinterval of $(0,\pi),$ then the convergence (\ref{c2}) is
uniform on that interval. The convergence is uniform on the
closed interval $[0, \pi]$ if and only if $f$  and $g$ are continuous
on $[0, \pi]$ and $\begin{pmatrix}  f \\  g  \end{pmatrix} $
satisfies the boundary condition $bc$ given by (\ref{8a}).
\end{Theorem}

\begin{proof}
Let  $A$ and $A^{-1}$ be the operators defined in the proof of
Theorem~\ref{thm01} by (\ref{24}) and (\ref{26}) in case $bc$ is
strictly regular or periodic type, and by (\ref{94}) and (\ref{96})
in case $bc$ is regular but not strictly regular or periodic type.
The operators
$A$  and $A^{-1}$ act on a vector-function $\begin{pmatrix}  f(t)  \\
g (t)\end{pmatrix}$ by multiplying $f(t), \, g(t), \, f(\pi -t), \,
g(\pi -t) $ by some $C^\infty$-functions, so $A$  and $A^{-1}$ are
defined point-wise. Recall from the proof of Theorem~\ref{thm01} that
\begin{equation}
\label{c01} \varphi_k^\nu = A e_k^\nu, \;\; \tilde{\varphi}_k^\nu =
(A^{-1})^*
e_k^\nu, \;\; \text{where}\;\;  e_k^1 =\begin{pmatrix}  e^{ikt}  \\
0\end{pmatrix}, \;\; e_k^2= \begin{pmatrix}  0  \\
e^{ikt} \end{pmatrix}.
\end{equation}

Suppose $f$ and $g$  are functions of bounded variation on $[0,\pi],$
and let
\begin{equation}
\label{c8}
\begin{pmatrix}  F  \\ G \end{pmatrix}(t) := A^{-1}
\begin{pmatrix}  f  \\ g \end{pmatrix}(t);
\end{equation}
then $F$  and $G$ are functions of bounded variation on $[0,\pi]$
also (as products of functions of bounded variations by
$C^\infty$-functions).

In view of (\ref{26}) and
(\ref{96}), the functions $F(t)$  and $G(t)$ are continuous at $t$ if
and only if $f(\pi -t)$ and $g(t) $ are continuous at $t.$
Therefore,\\

(i) if $f(\pi -t) $ and $g(t) $ are continuous on some closed interval $I \subset [0, \pi] $ then $F$ and $G$ are continuous on $I$ as well.\\

Moreover, by Lemma \ref{lembc}, \\

(ii) if $f$ and $g$ are continuous on $[0, \pi]$ and
$\begin{pmatrix}  f  \\ g \end{pmatrix}$  satisfies the boundary conditions (\ref{8a}), then $F$ and $G$ are
continuous on $[0, \pi]$ and
$\begin{pmatrix}  F  \\ G \end{pmatrix}$
satisfies the periodic boundary conditions $F(0) = F(\pi),$
$G(0) = G(\pi).$\\

By (\ref{c01}),
$$ \left \langle
\begin{pmatrix}  f\\g  \end{pmatrix}, \tilde{\varphi}_m^1  \right
\rangle   =\left \langle A^{-1}
\begin{pmatrix}  f\\g  \end{pmatrix},
e_m^1  \right \rangle =\left \langle
\begin{pmatrix}  F\\G  \end{pmatrix},
\begin{pmatrix}  e^{imt} \\ 0  \end{pmatrix}  \right \rangle =
\langle  F,e^{imt} \rangle,
$$
and similarly, $\left \langle
\begin{pmatrix}  f\\g  \end{pmatrix}, \tilde{\varphi}_m^2  \right
\rangle =\langle  G,e^{imt} \rangle.$

By Lemma~\ref{CT},  the Fourier series of $F$ and $G$ with respect to the system $\{e^{imt}, \, m \in 2\mathbb{Z}\}$ converge point-wise,
and  the convergence is uniform  in the cases (i) and (ii) mentioned above.
Let $\tilde{F}(x)$ and $\tilde{G}(x)$
denote, respectively, the point-wise sums of those series at $x \in
[0,\pi].$ Fix a point $x \in [0,\pi]; $ then
$$  \sum_{m=-M}^M
\sum_{\nu=1}^2 \left \langle  \begin{pmatrix}  f\\g
\end{pmatrix}, \tilde{\varphi}_m^\nu  \right \rangle \varphi_m^\nu (x)
 =A \left ( \sum_{-M}^M \left [
\langle  F,e^{imt} \rangle e_m^1  + \langle  G,e^{imt} \rangle e_m^2
\right ] \right )(x) $$
$$ = A \begin{pmatrix}
\sum_{m=-M}^M  \langle  F, e^{imt} \rangle  e^{imt}  \\
\sum_{m=-M}^M  \langle  G, e^{imt} \rangle  e^{imt}
\end{pmatrix} (x) \to    A  \begin{pmatrix}  \tilde{F}  \\
\tilde{G}  \end{pmatrix} (x) \quad \text{as} \quad M \to \infty.
$$ Therefore,  the expansion of $\begin{pmatrix}  f\\g
\end{pmatrix} $ about the basis $\Phi$
 converges point-wise 
to the
vector-function
  \begin{equation}
\label{c10}
  \begin{pmatrix}   \tilde{f} (x) \\ \tilde{g} (x)
\end{pmatrix} := A  \begin{pmatrix}
 \tilde{F}  \\ \tilde{G}    \end{pmatrix} (x).
 \end{equation}
Moreover, the convergence is uniform  
on every closed subinterval  of $[0,\pi] $ on
which both $f(\pi -t)$ and $g(t) $ are  continuous,
and  on the closed interval $[0, \pi]$ 
if $f$ and $g$ are continuous on $[0, \pi]$ and
$\begin{pmatrix}  f  \\ g \end{pmatrix}$  satisfies the boundary conditions (\ref{8a}).

For $x\in (0,\pi),$ Lemma~\ref{CT} implies that $$
\begin{pmatrix}  \tilde{F} (x) \\ \tilde{G} (x)    \end{pmatrix} =
 \begin{pmatrix} \frac{1}{2} [F(x-0)+F(x+0)] \\  \frac{1}{2}
[G(x-0)+G(x+0)] \end{pmatrix}=\frac{1}{2}  \begin{pmatrix}
 F(x-0) \\  G(x-0)\end{pmatrix}+\frac{1}{2}  \begin{pmatrix}
 F(x+0) \\  G(x+0)\end{pmatrix}, $$
 so
 \begin{equation}
\label{c11}
\begin{pmatrix}  \tilde{f} (x) \\ \tilde{f} (x)    \end{pmatrix} =
\frac{1}{2} \,A \begin{pmatrix}
 F(x-0) \\  G(x-0)\end{pmatrix}+\frac{1}{2}\, A  \begin{pmatrix}
 F(x+0) \\  G(x+0)\end{pmatrix}.
 \end{equation}

In case $bc $ is strictly regular or periodic type the operator
$A^{-1}$ is given by (\ref{26}), so  $$\begin{pmatrix}
 F(x-0) \\  G(x-0)\end{pmatrix} = A^{-1} \begin{pmatrix} f \\ g
 \end{pmatrix} (x-0) = \begin{pmatrix} e^{-i \tau_1 x}
 [\alpha_1^\prime f(\pi - x+0) + \alpha_2^\prime g(x-0) ] \\
e^{-i \tau_2 x}
 [\beta_1^\prime f(\pi - x+0) + \beta_2^\prime g(x-0) ]
\end{pmatrix} $$ and  we obtain by  (\ref{24}) and (\ref{20a})
$$ A\begin{pmatrix}
 F(x-0) \\  G(x-0) \end{pmatrix}  =\begin{pmatrix}
 \alpha_1 [\alpha_1^\prime f( x+0) + \alpha_2^\prime g(\pi-x-0) ]
 \\ \alpha_2 [\alpha_1^\prime f(\pi - x+0) + \alpha_2^\prime g(x-0)
 ]\end{pmatrix}
$$ $$ +\begin{pmatrix}
 \beta_1
 [\beta_1^\prime f(x+0) + \beta_2^\prime g(\pi-x-0) ]
 \\ \beta_2 [\beta_1^\prime f(\pi - x+0) + \beta_2^\prime g(x-0) ]
  \end{pmatrix} = \begin{pmatrix} f(x+0) \\  g(x-0)
  \end{pmatrix}.
$$
In an analogous way it follows that $  A\begin{pmatrix}
 F(x+0) \\  G(x+0) \end{pmatrix} (x)=\begin{pmatrix} f(x-0) \\  g(x+0)
  \end{pmatrix}. $
 Therefore, in view of (\ref{c11}), we obtain that (\ref{c03})
 holds for strictly regular or periodic type $bc.$

In case of regular $bc$ which is not strictly regular or periodic
type, (\ref{96}) implies that
$$ \begin{pmatrix}
 F(x-0) \\  G(x-0)\end{pmatrix} =\frac{1}{\Delta} \begin{pmatrix} e^{-i \tau_*
 x}[(\beta_2 + \alpha_2 x)f(\pi-x +0)- (\beta_1 - \pi\alpha_1 +  \alpha_1 x)
 g(x-0)\\
e^{-i \tau_* x}
 [-\alpha_2 f(\pi - x+0) + \alpha_1 g(x-0) ],
\end{pmatrix}, $$
where $\Delta = \alpha_1 \beta_2 - \alpha_2 \beta_1 + \pi \alpha_1
\alpha_2. $ Therefore, by (\ref{94}) it follows that
$$
 A\begin{pmatrix}
 F(x-0) \\  G(x-0)\end{pmatrix} = \frac{1}{\Delta} \begin{pmatrix}
\alpha_1 (\beta_2 + \alpha_2 \pi-\alpha_2 x) f(x+0) - \alpha_1
(\beta_1 - \alpha_1 x) g(\pi -x-0)\\  \alpha_2 (\beta_2 +\alpha_2 x)
f(\pi -x+0) -\alpha_2 (\beta_1 - \pi\alpha_1 +  \alpha_1 x) g(x-0)
 \end{pmatrix}
$$
$$
+ \frac{1}{\Delta} \begin{pmatrix}
(\beta_1 - \alpha_1 x)(- \alpha_2 f(x+0) + \alpha_1  g(\pi -x-0)\\
(\beta_2 +\alpha_2 x)(-\alpha_2 f(\pi -x+0) +\alpha_1 g(x-0)
 \end{pmatrix}
 =\begin{pmatrix}
 f(x+0) \\  g(x-0)\end{pmatrix}.
$$
Similar calculation shows  that $  A\begin{pmatrix}
 F(x+0) \\  G(x+0) \end{pmatrix} (x)=\begin{pmatrix} f(x-0) \\  g(x+0)
  \end{pmatrix}. $
 Therefore,  (\ref{c03})
 holds in case $bc$ is regular but not strictly regular or periodic type.

Next we evaluate $\tilde{f}(0), \tilde{f}(\pi), \tilde{g}(0),
\tilde{g}(\pi).$  For convenience, the calculations  are presented
in a matrix form.

If $bc$ is regular or periodic type, then by (\ref{c10}) and
(\ref{24})
\begin{equation}
 \label{cc1}
\begin{pmatrix} \tilde{f}(0) \\ \tilde{f}(\pi) \\ \tilde{g}(0) \\
\tilde{g}(\pi)     \end{pmatrix}= \left [
\begin{array}{cccc}
0 & \alpha_1 z_1 & 0 & \beta_1 z_2 \\ \alpha_1 & 0& \beta_1 & 0
\\ \alpha_2 & 0& \beta_2 & 0 \\ 0 & \alpha_2 z_1 & 0 & \beta_2 z_2
\end{array}
\right ] \begin{pmatrix} \tilde{F}(0) \\ \tilde{F}(\pi) \\
\tilde{G}(0) \\ \tilde{G}(\pi)     \end{pmatrix},
\end{equation}
where $z_1 = e^{i\tau_1 \pi}$ and  $z_2 = e^{i\tau_2 \pi}$  are the
roots of (\ref{13}) (in case $bc$ is strictly regular $z_1 \neq z_2;$
if $bc$ is periodic type then $z_1 = z_2 = z_*, $ and $\tau_1 =
\tau_2 = \tau_*$).

 In view of Lemma \ref{CT}, we have
 \begin{equation}
 \label{cc2}
\begin{pmatrix} \tilde{F}(0) \\ \tilde{F}(\pi) \\
\tilde{G}(0) \\ \tilde{G}(\pi)     \end{pmatrix}= \left [
\begin{array}{cccc}
1/2 & 1/2 & 0 & 0 \\ 1/2 & 1/2 & 0 & 0
\\ 0 & 0& 1/2 & 1/2 \\ 0 & 0& 1/2 & 1/2
\end{array}
\right ] \begin{pmatrix} F(0) \\ F(\pi) \\ G(0) \\ G(\pi)
\end{pmatrix}.
\end{equation}
On the other hand, by (\ref{c8}) and (\ref{26}) it follows that
\begin{equation}
 \label{cc3}
\begin{pmatrix} F(0) \\ F(\pi) \\ G(0) \\ G(\pi)
\end{pmatrix}=\left [
\begin{array}{cccc}
0 & \alpha_1^\prime & \alpha_2^\prime & 0 \\  \alpha_1^\prime/z_1
& 0&  0& \alpha_2^\prime/z_1
\\ 0 & \beta_1^\prime & \beta_2^\prime & 0 \\  \beta_1^\prime/z_2
& 0&  0& \beta_2^\prime/z_2
\end{array}
\right ] \begin{pmatrix} f(0) \\ f(\pi) \\ g(0) \\ g(\pi)
\end{pmatrix}.
\end{equation}
Now (\ref{cc1})--(\ref{cc3})  imply
\begin{equation}
\label{c101}
\begin{pmatrix} \tilde{f}(0) \\ \tilde{f}(\pi) \\ \tilde{g}(0) \\
\tilde{g}(\pi)     \end{pmatrix}= \frac{1}{2}\mathcal{M}
\begin{pmatrix} f(0)
\\ f(\pi)
\\ g(0) \\ g(\pi)
\end{pmatrix},
\end{equation}
where
 \begin{equation}
\label{c102}
 \mathcal{M}=  \left [
\begin{array}{cccc}
\alpha_1 \alpha_1^\prime +\beta_1 \beta_1^\prime & \alpha_1
\alpha_1^\prime z_1 +\beta_1 \beta_1^\prime z_2 &  \alpha_1
\alpha_2^\prime z_1 +\beta_1 \beta_2^\prime z_2 & \alpha_1
\alpha_2^\prime  +\beta_1 \beta_2^\prime \\
 \frac{\alpha_1 \alpha_1^\prime}{z_1} +\frac{\beta_1 \beta_1^\prime}{z_2} &
 \alpha_1 \alpha_1^\prime +\beta_1 \beta_1^\prime & \alpha_1
\alpha_2^\prime  +\beta_1 \beta_2^\prime & \frac{\alpha_1
\alpha_2^\prime}{z_1}  +\frac{\beta_1 \beta_2^\prime}{z_2}
\\
 \frac{\alpha_2 \alpha_1^\prime}{z_1} +\frac{\beta_2 \beta_1^\prime}{z_2} &
 \alpha_2 \alpha_1^\prime +\beta_2 \beta_1^\prime & \alpha_2
\alpha_2^\prime  +\beta_2 \beta_2^\prime & \frac{\alpha_2
\alpha_2^\prime}{z_1}  +\frac{\beta_2 \beta_2^\prime}{z_2}
\\
\alpha_2 \alpha_1^\prime +\beta_2 \beta_1^\prime & \alpha_2
\alpha_1^\prime z_1 +\beta_2 \beta_1^\prime z_2 &  \alpha_2
\alpha_2^\prime z_1 +\beta_2 \beta_2^\prime z_2 & \alpha_2
\alpha_2^\prime  +\beta_2 \beta_2^\prime
\end{array}
\right ].
\end{equation}
Next we evaluate the entries of the matrix
$\mathcal{M}=(\mathcal{M}_{ij}).$ In view of (\ref{20a}) we have
$$
\alpha_1 \alpha_1^\prime +\beta_1 \beta_1^\prime =1,
\quad \alpha_2 \alpha_2^\prime +\beta_2 \beta_2^\prime =1, \quad
\alpha_1 \alpha_2^\prime +\beta_1 \beta_2^\prime=0,
\quad  \alpha_2 \alpha_1^\prime +\beta_2 \beta_1^\prime=0,
$$
so $\mathcal{M}_{ii}=1, \; i=1,2,3,4,$ and
$\mathcal{M}_{14}=\mathcal{M}_{23}=\mathcal{M}_{32}=\mathcal{M}_{41}=0.$
In order to find the remaining elements of $\mathcal{M}$ recall that
$\begin{pmatrix}  \alpha_1 \\ \alpha_2  \end{pmatrix}$ and
$\begin{pmatrix}  \beta_1 \\ \beta_2  \end{pmatrix}$ are eigenvectors
of the matrix $  \left [\begin{array}{cc}  b&a \\d  &c
\end{array}  \right ]$ which correspond to its eigenvalues $- z_1$
and $-z_2 $ (see the text between (\ref{13}) and (\ref{20a})).
Therefore, we have $$ b\alpha_1 +a \alpha_2 =-z_1 \alpha_1, \quad
 d\alpha_1 +c \alpha_2 =-z_1 \alpha_2, \quad
$$
$$ b\beta_1 +a \beta_2 =-z_2 \beta_1, \quad
 d\beta_1 +c \beta_2 =-z_2 \beta_2.
$$
In addition, (\ref{13}) implies that
$$
z_1 + z_2 = - (b+c), \quad z_1 z_2 = bc-ad.
$$
Using the above formulas we obtain
$$
\mathcal{M}_{12} =\alpha_1 \alpha_1^\prime z_1 +\beta_1
\beta_1^\prime z_2 = - \alpha_1^\prime  (b\alpha_1 + a \alpha_2) -
\beta_1^\prime (b \beta_1 + a \beta_2) $$
$$
= -b (\alpha_1\alpha_1^\prime + \beta_1 \beta_1^\prime) - a(\alpha_2
\alpha_1^\prime + \beta_2 \beta_1^\prime ) = -b;
$$
$$
\mathcal{M}_{13} = \alpha_1\alpha_2^\prime z_1    + \beta_1
\beta_2^\prime z_2 = - \alpha_2^\prime  (b\alpha_1 + a \alpha_2) -
\beta_2^\prime (b \beta_1 + a \beta_2) $$
$$
= -b (\alpha_1\alpha_2^\prime + \beta_1 \beta_2^\prime) - a(\alpha_2
\alpha_2^\prime + \beta_2 \beta_2^\prime ) = -a;
$$
$$
\mathcal{M}_{21}= \frac{\alpha_1 \alpha_1^\prime}{z_1} +\frac{\beta_1
\beta_1^\prime}{z_2} = \frac{1}{z_1 z_2} (\alpha_1 \alpha_1^\prime
z_2 +\beta_1 \beta_1^\prime z_1) $$
$$
=\frac{1}{bc-ad} \,\left [\alpha_1 \alpha_1^\prime (-b-c -z_1)
+\beta_1 \beta_1^\prime (-b-c-z_2) \right ]
$$
$$
=\frac{1}{bc-ad} \, \left [-(b+c)(\alpha_1 \alpha_1^\prime+ \beta_1
\beta_1^\prime ) - \mathcal{M}_{12} \right ] =\frac{-c}{bc-ad};
$$
$$
\mathcal{M}_{24} = \frac{\alpha_1 \alpha_2^\prime}{z_1}
+\frac{\beta_1 \beta_2^\prime}{z_2} = \frac{1}{z_1 z_2} (\alpha_1
\alpha_2^\prime z_2 +\beta_1 \beta_2^\prime z_1)
$$
$$
=\frac{1}{bc-ad} \,\left [\alpha_1 \alpha_2^\prime (-b-c -z_1)
+\beta_1 \beta_2^\prime (-b-c-z_2) \right ]
$$
$$
=\frac{1}{bc-ad} \, \left [-(b+c)(\alpha_1 \alpha_2^\prime+ \beta_1
\beta_2^\prime ) - \mathcal{M}_{13} \right ] = \frac{a}{bc-ad}.
$$

In an analogous way one can find
$\mathcal{M}_{31},  \mathcal{M}_{34},  \mathcal{M}_{42}$  and $\mathcal{M}_{43};$
we omit the details   and give the final result:
\begin{equation}
\label{c17}
 \mathcal{M}= \frac{1}{2}\left [
\begin{array}{cccc}
1 &  -b  &  -a  &  0 \\  \frac{-c}{bc-ad} & 1 &  0  & \frac{a}{bc-ad}\\
\frac{d}{bc-ad}  &  0  &  1  &  \frac{-b}{bc-ad}  \\ 0  &  -d  &  -c  &  1
  \end{array}  \right ].
\end{equation}
Hence, (\ref{c3}) holds if $bc$ is strictly regular or periodic type.

 In the case $bc $ is regular but not strictly regular or periodic type we
 use the same argument but work with the operators $A$  and $A^{-1} $
 defined by (\ref{94}) and (\ref{96}).
By (\ref{c10}) and (\ref{94})
\begin{equation}
 \label{cc11}
\begin{pmatrix} \tilde{f}(0) \\ \tilde{f}(\pi) \\ \tilde{g}(0) \\
\tilde{g}(\pi)     \end{pmatrix}= \left [
\begin{array}{cccc}
0 & \alpha_1 z_* & 0 & \beta_1 z_* \\ \alpha_1 & 0& \beta_1- \alpha_1
\pi & 0
\\ \alpha_2 & 0 & \beta_2 & 0 \\ 0 & \alpha_2 z_* & 0 & (\beta_2 + \alpha_2 \pi)
z_*
\end{array}
\right ] \begin{pmatrix} \tilde{F}(0) \\ \tilde{F}(\pi) \\
\tilde{G}(0) \\ \tilde{G}(\pi)     \end{pmatrix},
\end{equation}
where $\Delta = \alpha_1 \beta_2 - \alpha_2 \beta_1 + \pi \alpha_1
\alpha_2$ and $z^* = e^{i\pi \tau_*}.$

On the other hand, by (\ref{c8}) and (\ref{96}) it follows that
\begin{equation}
 \label{cc31}
\begin{pmatrix} F(0) \\ F(\pi) \\ G(0) \\ G(\pi)
\end{pmatrix}=\frac{1}{\Delta} \left [
\begin{array}{cccc}
0 & \beta_2 & -\beta_1 +\pi \alpha_1 & 0 \\  \frac{\beta_2 + \pi
\alpha_2}{z_*} & 0& 0& \frac{-\beta_1}{z_*}
\\ 0 & -\alpha_2 & \alpha_1 & 0 \\  \frac{-\alpha_2}{z_*}
& 0&  0& \frac{\alpha_1}{z^*}
\end{array}
\right ] \begin{pmatrix} f(0) \\ f(\pi) \\ g(0) \\ g(\pi)
\end{pmatrix}.
\end{equation}
Now (\ref{cc11}), (\ref{cc2}) and (\ref{cc31})  imply that
(\ref{c101}) holds with
 \begin{equation}
\label{c104}
 \mathcal{M}= \frac{1}{\Delta} \left [
\begin{array}{cccc}
\Delta & (\alpha_1 \beta_2 - \beta_1 \alpha_2)z_* & \pi \alpha_1^2 z_*
 & 0 \vspace{1mm}\\
 \frac{\Delta + \pi \alpha_1 \alpha_2 }{z_*}  &
\Delta  & 0 & -\frac{\pi \alpha_1^2}{z_*} \vspace{1mm}
\\
\frac{\pi \alpha_2^2}{z_*} &0
 &\Delta  & \frac{\alpha_1 \beta_2 - \beta_1 \alpha_2}{z_*} \vspace{1mm}
\\
0 & -\pi \alpha_2^2 z_* &  (\Delta + \pi \alpha_1 \alpha_2 )z_* & \Delta
\end{array}
\right ].
\end{equation}
The parameters $\alpha_1, \alpha_2, \beta_1, \beta_2 $  come from
(\ref{80}), where we consider three cases: (i) $ \; a=0; \;$ (ii) $\;
d=0;\; $ (iii) $\; a\neq 0, \, b\neq 0.$

In case (iii), we have
$$
\alpha_1 = a, \;\; \alpha_2 = (c-b)/2, \;\; \beta_1 = 0, \;\; \beta_2
= \pi b.
$$
Recall also that $ z_* = -\frac{b+c}{2} $ and $\; z_*^2 =bc - ad \;$
because $z_* $ is a double root of (\ref{13}). Therefore,
$$\Delta= \alpha_1 \beta_2 - \alpha_2 \beta_1 + \pi \alpha_1
\alpha_2= \pi ab +\pi a \frac{c-b}{2} = \pi a \frac{b+c}{2} = - \pi a
z_*.
$$
Next we evaluate the entries of $\mathcal{M} $ in case (iii):
$$
\mathcal{M}_{12} = \frac{1}{\Delta} (\alpha_1 \beta_2 - \alpha_2
\beta_1) z_* = \frac{1}{-\pi a z_*}  \pi ab z_* = -b;
$$
$$
\mathcal{M}_{21} =  \frac{\Delta + \pi \alpha_1 \alpha_2}{\Delta z_*}
= \frac{\pi a \frac{b+c}{2} + \pi a \frac{c-b}{2}}{-\pi a z_*^2}=
\frac{-c}{bc-ad};
$$
$$
\mathcal{M}_{13} =  \frac{\pi \alpha_1^2 z_*}{\Delta} = \frac{\pi a^2
z_*}{-\pi a z_*} = -a;\quad \mathcal{M}_{24} =  -\frac{\pi \alpha_1^2
}{\Delta z_*} = \frac{-\pi a^2 }{-\pi a z_*^2} = \frac{a}{bc-ad}.
$$
In a similar way we calculate $\mathcal{M}_{31}, \mathcal{M}_{34},
\mathcal{M}_{42}, \mathcal{M}_{43} $ and obtain that in case (iii)
the matrix $\mathcal{M}$ is given by (\ref{c17}).

An elementary calculation (which is omitted) shows that in cases (i)
and (ii) the matrix $\mathcal{M}$ is given by (\ref{c17}) also, so
(\ref{c3}) holds if $bc$ is regular but not strictly regular or
periodic type as well. This completes the proof.
\end{proof}

\section{Generalizations}

\subsection{Weighted eigenvalue problems, general potential matrices}

Suppose   $\rho \in L^1 ([x_1, x_2]) $ and $\rho (x) \geq const
>0. $ Let $L^2 ([x_1,x_2], \rho) $ be the space of all
measurable functions $f: [x_1, x_2] \to \mathbb{C}$  such that $$
\|f\|_\rho^2 = \int_{x_1}^{x_2} |f(x)|^2 \rho (x) dx < \infty. $$
Suppose that $$ T=
 \begin{pmatrix}  T_{11}  &  T_{12}  \\   T_{21}  & T_{22}
\end{pmatrix}, \quad \frac{1}{\rho}T_{ij} \in  L^2 ([x_1,x_2], \rho).$$
Consider the operator
\begin{equation}
\label{g.1} L_{bc} (T, \rho )y :=  \frac{1}{\rho (x)} \left [ i
\begin{pmatrix} 1 & 0
\\ 0 & -1
\end{pmatrix}   \frac{dy}{dx}  + Ty \right ],
\quad y=\begin{pmatrix}y_1 \\ y_2
\end{pmatrix},
\end{equation}
 subject to the boundary conditions $bc$
\begin{eqnarray}
\label{88g}  y_1 (x_1) +b \, y_1 (x_2) + a \, y_2 (x_1) =0,
\\ \nonumber   d \, y_1 (x_2) + c \, y_2 (x_1) +  y_2 (x_2)=0,
\end{eqnarray}
in the domain $Dom \,L_{bc} (T, \rho ) \subset \left ( L^2
([x_1,x_2] , \rho)\right )^2 $ which consists of all absolutely
continuous functions $y$ such that (\ref{88g}) holds and
$y^\prime_{1}/\rho, y^\prime_{2}/\rho \in  L^2 ([x_1,x_2], \rho).$
 It is easy to
see that $L_{bc} (T, \rho )$ is a densely defined closed operator.
A standard computation of the adjoint operators leads to the
following.

\begin{Lemma}
\label{lemadj} In the above notations,
\begin{equation}
\label{adj1} (L_{bc} (T, \rho ))^* = L_{\widetilde{bc}} (T^*, \rho),
\quad \text{where} \quad T^* =\begin{pmatrix}  \overline{T}_{11}  &
\overline{T}_{21}
\\   \overline{T}_{12} & \overline{T}_{22}
\end{pmatrix}
\end{equation}
and the boundary conditions $\widetilde{bc}$ are defined by
\begin{equation}
\label{adj2} \overline{b} y_1 (x_1) + y_1 (x_2) +\overline{d} y_2
(x_2) = 0, \quad  \overline{a}y_1 (x_1) + y_2 (x_1) + \overline{c}
y_2 (x_2)=0.
\end{equation}
\end{Lemma}

The boundary conditions (\ref{adj2}) are not written in the standard
form (\ref{88g}) but a multiplication of the system of equations
(\ref{adj2}) from the left by the inverse matrix
$$
\begin{pmatrix}
\overline{b} & \overline{d}  \\ \overline{a} & \overline{c}
\end{pmatrix}^{-1} =
\begin{pmatrix}
\overline{c}/\overline{\Delta}
 & -\overline{d}/\overline{\Delta}  \\ -\overline{a}/\overline{\Delta}
 & \overline{b}/\overline{\Delta}
\end{pmatrix}, \quad
\Delta = bc - ad,
$$
would bring the boundary conditions $\widetilde{bc}$ to the standard
form. This observation leads to the following

\begin{Corollary}
\label{cor22}
The operator $L_{bc} (T, \rho )$ is self-adjoint if
and only if
$$
b=\overline{c}/\overline{\Delta}, \quad
a=-\overline{d}/\overline{\Delta}, \quad
d=-\overline{a}/\overline{\Delta}, \quad
c=\overline{b}/\overline{\Delta}, \quad  T= T^*.
$$
\end{Corollary}

An appropriate change of the variable transforms the operator
$L_{bc}(T,\rho) $ into an operator acting in $\left (L^2
([0,\pi])\right )^2. $ Indeed, let
\begin{equation}
\label{g.2}
t(x) = K \int_{x_1}^x  \rho (\xi) d\xi, \quad      x_1 \leq x \leq x_2,
\end{equation}
where the constant $K>0 $ is chosen so that
$t(x_2)=K \int_{x_1}^{x_2}
\rho (\xi) d\xi=\pi,$ and let $x(t) :   \; [0,\pi]  \to [x_1,x_2]
$ be the inverse function of $t(x).$ The change of variable
$x=x(t) $ give rise of an isomorphism $$ W: L^2 ([x_1,x_2], \rho)
\to L^2 ([0,\pi]), \quad (Wf)(t)= f(x(t))$$ because $\int_0^\pi
|f(x(t))|^2 dt= K\cdot \int_{x_1}^{x_2} |f(x)|^2 \rho (x) dx.$ Of
course, the operator  $$ W^{(2)} : \left( L^2 ([x_1,x_2] ,
\rho)\right )^2 \to \left ( L^2 ([0,\pi] )\right )^2, \quad
W^{(2)} \begin{pmatrix} f \\ g  \end{pmatrix}
=
\begin{pmatrix} Wf \\ Wg  \end{pmatrix} $$
is also an isomorphism.

Consider the operator
\begin{equation}
\label{g.12}
 L_{bc} (S)u :=  i \begin{pmatrix}  1  &  0
\\ 0 & -1
\end{pmatrix}   \frac{du}{dt}  + S u,  \quad u= \begin{pmatrix}
u_1  \\u_2   \end{pmatrix}
\end{equation}
with $$  S=
\begin{pmatrix}  S_{11}  &  S_{12}  \\   S_{21}  & S_{22}
\end{pmatrix}, \quad  S_{ij}(t) = \frac{1}
{K\rho (x(t))} T_{ij}(x(t)), \quad i,j \in \{1,2\}, $$ subject to
the boundary conditions $bc$
 \begin{eqnarray}
\label{8.g}  u_1 (0) +b u_1 (\pi) + a u_2 (0) =0,
\\ \nonumber d u_1 (\pi) + c u_2 (0) +  u_2
(\pi)=0,
\end{eqnarray}
in the domain $ D(L_{bc}(S)) \subset \left (L^2 ([0,\pi]) \right)^2$
which consists of all absolutely continuous functions $u$ such that
(\ref{8.g}) holds and $u^\prime_{1}, u^\prime_{1} \in L^2
([0,\pi]).$

\begin{Lemma}
\label{lemg1} The operators $L_{bc}(T, \rho)
$ and $K \cdot L_{bc}(S)$  are similar.
\end{Lemma}

\begin{proof}
Change the variables in (\ref{g.1}), (\ref{88g}) by
\begin{equation}
\label{g.3} x= x(t), \quad u(t) = y(x(t)) =\begin{pmatrix}  y_1
(x(t))  \\ y_2(x(t)) \end{pmatrix} , \quad 0 \leq t \leq \pi.
\end{equation}
  Then $
 u^\prime (t) = y^\prime (x(t))\cdot x^\prime (t)
=y^\prime (x(t)) \frac{1}{K\rho (x(t))},$  the boundary conditions
(\ref{88g}) transform into (\ref{8.g}), the domain
$D(L_{bc}(T,\rho))$ transforms into $D(L_{bc}(S)),$  so the operator
$L(T_{bc}, \rho) $ transforms into the operator $K \cdot L_{bc}(S).$
In other words, we obtain that $$ W^{(2)} L(T_{bc},\rho) =K\cdot
L_{bc}(S)\, W^{(2)}, $$ which completes the proof.

\end{proof}

Set
\begin{equation}
\label{g1} s_{1} (t) = \int_{0}^t   S_{11} (\tau) d \tau, \quad
s_{2} (t) = \int_{0}^t   S_{22} (\tau) d \tau, \quad 0 \leq t \leq
\pi.
\end{equation}

\begin{Proposition}
\label{propT}  In the above notations,
the Dirac operator
\begin{equation}
\label{g2} L_{bc} (S)u = i \begin{pmatrix}  1  &  0  \\   0  & -1
\end{pmatrix}   \frac{du}{dt} +
\begin{pmatrix}  S_{11}  &  S_{12}  \\  S_{21}  & S_{22}
\end{pmatrix}u, \quad S_{ij} \in L^2 ([0,\pi]),
\end{equation}
subject to the boundary conditions $bc$
\begin{eqnarray}
\label{8g}  u_1 (0) +b u_1 (\pi) + a u_2 (0) =0,
\\ \nonumber d u_1 (\pi) + c u_2 (0) +  u_2 (\pi)=0,
\end{eqnarray}
is similar to the Dirac operator
\begin{equation}
\label{g3} L_{\widetilde{bc}} (v)\tilde{u} =
L^0\tilde{u}+v\tilde{u}, \quad v=
\begin{pmatrix} 0  & S_{12}  e^{-i(s_{1}(t) +s_{2} (t))} \\
S_{21} e^{i(s_{1}(t) +s_{2} (t))} & 0  \end{pmatrix},
\end{equation}
subject to the boundary conditions $\widetilde{bc}$
\begin{eqnarray}
\label{g8}  \tilde{u}_1 (0) +\tilde{b} \tilde{u}_1 (\pi) +
\tilde{a} \tilde{u}_2 (0) =0,
\\ \nonumber \tilde{d} \tilde{u}_1 (\pi) +
\tilde{c} \tilde{u}_2 (0) +  \tilde{u}_2 (\pi)=0,
\end{eqnarray}
where
\begin{equation}
\label{g9} \tilde{b}   = b e^{is_{1}(\pi)},   \quad    \tilde{a} = a,
\quad
 \tilde{d}  = de^{i (s_{1}(\pi)+s_{2}(\pi))},   \quad
 \tilde{c}  = c e^{is_{2}(\pi)}.
\end{equation}

\end{Proposition}

\begin{proof}
A simple calculation shows that formally
\begin{equation}
\label{g11} A L_{bc}(S) = L_{\widetilde{bc}}(v) A, \quad
\text{where} \quad A =
\begin{pmatrix} e^{-is_{1}(t)}  &  0  \\  0  & e^{is_{2}(t)}
\end{pmatrix}.
\end{equation}
The domain $Dom(L_{bc}(S))$ consists of all absolutely continuous
functions $u= \begin{pmatrix}u_1 \\u_2   \end{pmatrix}$ such that
(\ref{8g}) holds and $u_1^\prime, u_2^\prime \in L^2([0,\pi]),$ and
the domain $Dom(L_{\widetilde{bc}}(v))$ consists of all absolutely
continuous functions $\tilde{u}= \begin{pmatrix}\tilde{u}_1
\\\tilde{u}_2
\end{pmatrix}$ such that (\ref{g8}) holds and $\tilde{u}_1^\prime,
\tilde{u}_2^\prime \in L^2([0,\pi]).$ Therefore,
$u \in Dom(L_{bc}(S)) $ if and only if $\tilde{u} = A u  \in Dom
(L_{\widetilde{bc}}(v)). $ This, together with (\ref{g11}), means
that the operator $L_{bc}(S) $  subject to the boundary conditions
(\ref{8g}) is similar to the operator $L_{\widetilde{bc}}(v) $
subject to the boundary conditions (\ref{g8}).
\end{proof}

In view of Lemma~\ref{lemg1} and Proposition~\ref{propT}, now we
can extend our results from the previous sections to the case of
weighted eigenvalue problems on an arbitrary finite interval
$[x_1, x_2].$

\begin{Definition}
\label{GR} We say that the equations (\ref{88g}) give regular,
strictly regular or periodic type boundary conditions for the
operator (\ref{g.1}) if  (\ref{g8}) are regular, strictly regular
or periodic type boundary conditions for the operator (\ref{g3}).
\end{Definition}

By (\ref{10}) and  (\ref{g9}), the boundary conditions (\ref{g8})
are regular  if $$\tilde{b} \tilde{c} -\tilde{a} \tilde{d} =
(bc-ad) e^{i(s_{1}(\pi)+s_{2}(\pi))} \neq 0, $$ so  (\ref{8g}) are
{\em regular}
 boundary conditions for the operator (\ref{g2}) if and only if
 \begin{equation}
\label{g12}
bc - ad \neq 0.
\end{equation}
From (\ref{11}) and (\ref{g9}) it follows that  (\ref{g8})  are {\em
strictly regular} boundary conditions for the operator (\ref{g2}) if
and only if $ (\tilde{b} -\tilde{c} )^2  + 4\tilde{a} \tilde{d} \neq
0 $ which is equivalent to
 \begin{equation}
\label{g14} \left  (be^{is_1(\pi)} -ce^{is_2(\pi)} \right )^2  +
4ade^{i(s_1 (\pi) +s_2(\pi))} \neq 0.
\end{equation}
Finally, by  (\ref{72}) and (\ref{g9}), the equations (\ref{8g}) give
{\em periodic type}  boundary conditions for the operator (\ref{g2})
if
 \begin{equation}
\label{g15} be^{is_{1}(\pi)} =ce^{is_{2}(\pi)},   \quad a=0, \quad
d=0.
\end{equation}
\vspace{3mm}

The next theorem generalizes  the results in Sections 3-5 (see
Theorems~\ref{thm1}, \ref{EC}, \ref{GCT}).

\begin{Theorem}
\label{GGCT} Suppose $\rho \in L^1 ([x_1, x_2]), \; \rho (x) \geq
const
>0,$ and
 \begin{equation}
\label{g22} T= \begin{pmatrix}  T_{11}  &  T_{12} \\ T_{21}  &
T_{22}      \end{pmatrix}, \quad \frac{1}{\rho} T_{ij} \in L^2 ([x_1, x_2],
\rho).
\end{equation}
Consider the Dirac the operator
\begin{equation}
\label{g.22} L_{bc} (T, \rho )y :=  \frac{1}{\rho (x)} \left [ i
\begin{pmatrix} 1 & 0
\\ 0 & -1
\end{pmatrix}   \frac{dy}{dx}  + Ty \right ],
\quad y=\begin{pmatrix}y_1 \\ y_2
\end{pmatrix},
\end{equation}
subject to regular $bc $  (in the sense of Definition~\ref{GR})
 \begin{equation}
\label{g23}
y_1 (x_1) +b y_1 (x_2) + a y_2 (x_1) =0,  \quad
 d y_1 (x_2) + c y_2 (x_1) +  y_2 (x_2)=0.
\end{equation}

(A). If $bc$ are strictly regular (i.e., (\ref{g12}) and (\ref{g14})
hold), then in $(L^2([x_1,x_2],\rho))^2$ there is a basis of Riesz
projections $\{S_N, \; P^\alpha_n, \; \alpha=1,2,\;  N, n \in
2\mathbb{Z}, \; |n|>N\}$ of the operator $L_{bc}(T,\rho)$ such that
$\dim S_N = 2N+2, \; \dim P^\alpha_n =1, $    and
 \begin{equation}
\label{g24}
{\bf f} = S_N {\bf f} + \sum_{|n|>N}  \sum_{\alpha =1}^2
 P^\alpha_n {\bf f} \quad \forall {\bf f}
=\begin{pmatrix}    f_1  \\ f_2  \end{pmatrix} \in
(L^2([x_1,x_2],\rho))^2,
\end{equation}
where the series converge
unconditionally in $(L^2([x_1,x_2],\rho))^2.$

(B). If $bc$ are regular but not strictly regular (i.e., (\ref{g12})
holds but (\ref{g14}) fails), then in $(L^2([x_1,x_2],\rho))^2$
there is a bases of Riesz projections  $\{S_N, \; P_n,\; N, n \in
2\mathbb{Z}, \; |n|>N\}$ of the operator $L_{bc}(T,\rho)$ such that
$\dim S_N = 2N+2, \; \dim P_n =2, $ and
 \begin{equation}
\label{g25} {\bf f} = S_N {\bf f} + \sum_{|n|>N}    P_n {\bf f}
\quad \forall {\bf f} =\begin{pmatrix}    f_1  \\ f_2
\end{pmatrix} \in (L^2([x_1,x_2],\rho))^2,
\end{equation}
where the series converge
unconditionally in $(L^2([x_1,x_2],\rho))^2.$

(C).   If $ \;{\bf f} =\begin{pmatrix}    f_1  \\ f_2
\end{pmatrix}, $ where $f_1 $ and $f_2 $
 are functions of bounded variation on $[x_1, x_2],$ then
 the series (\ref{g24})  and (\ref{g25}) converge point-wise
 to a function ${\bf \tilde{f}} (x)
=\begin{pmatrix}    \tilde{f}_1 (x)  \\ \tilde{f}_2 (x) \end{pmatrix} $
 in the sense that
 \begin{equation}
\label{g26}
 (S_N {\bf f}) (x) + \lim_{M \to \infty} \sum_{N<|n| \leq M}
   \sum_{\alpha =1}^2 \left (P^\alpha_n {\bf f} \right ) (x)
= {\bf \tilde{f}} (x)
\end{equation}
in the strictly regular case, and
 \begin{equation}
\label{g27}
 (S_N {\bf f}) (x) + \lim_{M \to \infty} \sum_{N<|n| \leq M}
   (P_n {\bf f}) (x) ={\bf \tilde{f}} (x)
\end{equation}
if $bc $ is regular but not strictly regular.   Moreover,
\begin{equation}
\label{g28} {\bf \tilde{f}} (x) =  \frac{1}{2} ({\bf f} (x-0)  +
{\bf f} (x+0))  \quad \text{if}  \; \; x \in (x_1,x_2),
\end{equation}
and
\begin{equation}
\label{g30}{\bf \tilde{f}} (x_1) = \frac{1}{2}
\begin{pmatrix} f_1(x_1+0) -b f_1 (x_2-0) -a f_2(x_1+0) \\
   \frac{d}{bc-ad} f_1(x_1+0) + f_2(x_1+0) -
   \frac{b}{bc-ad} f_2(x_2-0)  \end{pmatrix},
   \end{equation}
   \begin{equation}
   \label{g31}  {\bf \tilde{f}} (x_2)  =
    \frac{1}{2} \begin{pmatrix}  \frac{-c}{bc-ad}
    f_1 (x_1+0)+f_1(x_2-0 )+ \frac{a}{bc-ad} f_2 (x_2-0) \\
    -d f_1 (x_2-0)  - c f_2(x_1+0) +f_2(x_2-0)    \end{pmatrix}.
 \end{equation}

If, in addition, both $f_1 (x_2 -x)$ and $f_2(x)$ are  continuous on
some closed subinterval of $(x_1,x_2 )$ then the convergence in
(\ref{g26}) and (\ref{g27}) is uniform on that interval.
The convergence is uniform on the
closed interval $[x_1, x_2]$ if and only if $f_1$  and $f_2$ are continuous
on $[0, \pi]$ and $\begin{pmatrix}  f_1 \\  f_2  \end{pmatrix} $
satisfies the boundary condition $bc$ given by (\ref{g23}).

\end{Theorem}

{\em Remark.}  One can easily see by (\ref{g23}), (\ref{g30}) and
(\ref{g31}) that if the function ${\bf f}$ is continuous at $x_1$ and
$ x_2$ then
$$
{\bf \tilde{f}} (x_1) = {\bf f} (x_1), \quad {\bf \tilde{f}} (x_2) =
{\bf f} (x_2)
$$
if and only if  ${\bf f}$ satisfies the boundary conditions
(\ref{g23}).

\begin{proof}
In view of Lemma~\ref{lemg1} and Proposition~\ref{propT}, (A) and
(B) follow from Theorem~\ref{thm1}.  Below, we show that (C)
follows from  Theorem~\ref{GCT}.

By Lemma~\ref{lemg1}, a suitable change of variable $t=t(x)$
transforms the operator $L_{bc}(T, \rho) $ subject to the boundary
conditions given by the matrix $\begin{pmatrix} 1 & b  & a & 0 \\ 0
& d & c & 1 \end{pmatrix}$ on $[x_1, x_2]$  into the operator $
L_{bc} (S) $ subject to $bc $ given by the same matrix on  $[0,
\pi]$.
 Therefore, it is enough to prove (C) in the
case where $x_1=0, \; x_2=\pi, \;\rho \equiv 1$ and  $T\equiv S.$

By Proposition~\ref{propT}, the operator $L_{bc}(S)$
 is similar to the operator $L_{\widetilde{bc}}(v) $
 defined in
(\ref{g3}) and subject to the boundary conditions
$\widetilde{bc}$
 given by the
matrix $\begin{pmatrix} 1 & \tilde{b} & \tilde{a} & 0 \\ 0 &
\tilde{d} & \tilde{c}  & 1
\end{pmatrix},$
where $\tilde{a}, \tilde{b}, \tilde{c}, \tilde{d}$ are defined by
(\ref{g9}).  By~(\ref{g11}), $$ L_{bc}(S)= A^{-1}
L_{\widetilde{bc}}(v) A \quad \text{with} \quad A=
\begin{pmatrix} e^{-is_1 (t)} & 0
\\ 0 &e^{is_2 (t)} \end{pmatrix}, $$
where $s_1 (t) $  and $s_2 (t) $  come from (\ref{g1}).

 Since the operators $A$ and $A^{-1} $ act on
vector-functions by multiplying their components by exponential
functions, the point-wise convergence of the spectral decompositions
of the operator $L_{\widetilde{bc}}(v) $ yields a point-wise
convergence of the spectral decompositions of the operator
$L_{bc}(S).$ Therefore, under the assumptions in (C), (\ref{g28})
holds, and the convergence is uniform on a closed subinterval $I
\subset (0,\pi )$ provided that $f_1 (\pi -t)$ and $f_2 (t)$ are
continuous on $I.$  Moreover, if $f_1 $
and $f_2 $ are continuous and
$\begin{pmatrix}  f_1  \\ f_2   \end{pmatrix}$ satisfies the boundary condition (\ref{g23}), then $A\begin{pmatrix}  f_1  \\ f_2   \end{pmatrix}$ satisfies the boundary conditions $\widetilde{bc},$ so the uniform
convergence on $[0,\pi]$ follows from Theorem~\ref{GCT}.

Next we consider the convergence at the  points $x_1=0, x_2= \pi$
 and show that (\ref{g30}) and (\ref{g31}) hold. Let $\{\hat{S}_N (v),
 \; \hat{P}_n (v), \; |n|>N\}$ be a basis of Riesz projections of the
operator $L_{\widetilde{bc}}(v),  $  and let $$ \{S_N
=A^{-1}\hat{S}_N (v) A, \; P_n =A^{-1} \hat{P}_n (v) A, \; |n|>N\}$$
be the corresponding basis of Riesz projections of the operator
$L_{bc} (S).$  Then  we have, for every $t \in [0, \pi],$ $$ {\bf
\tilde{f}}(t) = (S_N {\bf f})(t) + \lim_{M\to \infty}
\sum_{N<|n|\leq M} (P_n {\bf f})(t) $$
$$ = ( A^{-1} \hat{S}_N (v) A {\bf f})(t) + \lim_{M\to \infty}
\sum_{N<|n|\leq M} \left (A^{-1}\hat{P}_n (v) A {\bf f}\right )(t).
$$ Therefore,
\begin{equation}
\label{g51} {\bf \tilde{f}}(t)= A^{-1} \begin{pmatrix}
\widehat{f_1}
\\ \widehat{f_1}
\end{pmatrix} (t), \quad    t \in [0, \pi],
\end{equation}
where
\begin{equation}
\label{g52}
\begin{pmatrix}
\widehat{f_1} \\ \widehat{f_2}
\end{pmatrix} (t)
= S_N (v) A \begin{pmatrix} f_1 \\ f_2
\end{pmatrix} (t)
+ \lim_{M\to \infty} \sum_{N<|n|\leq M} P_n (v) A \begin{pmatrix}
f_1 \\ f_2  \end{pmatrix} (t).
\end{equation}
Since
$$
S \begin{pmatrix}
f_1 \\ f_2
\end{pmatrix} (t)  =
\begin{pmatrix}
e^{-is_1 (t)} f_1 (t) \\ e^{i s_2 (t)} f_2 (t)
\end{pmatrix} \quad \text{and} \quad
A^{-1} \begin{pmatrix}
  g_1 \\ g_2
\end{pmatrix} (t)  =
\begin{pmatrix}
e^{is_1 (t)} g_1 (t) \\ e^{-i s_2 (t)} g_2 (t)
\end{pmatrix},
$$
we obtain
$$
\begin{pmatrix}
\widehat{f_1} (0) \\ \widehat{f_1} (\pi) \\ \widehat{f_2} (0) \\
\widehat{f_2} (\pi)
\end{pmatrix}
= \frac{1}{2} M_{\widetilde{bc}}
\begin{pmatrix}
f_1 (0) \\ e^{-is_1 (\pi)} f_1 (\pi) \\ f_2 (0) \\ e^{is_2 (\pi)}
f_2 (\pi)
\end{pmatrix}, \quad
\begin{pmatrix}
\widetilde{f_1} (0) \\ \widetilde{f_1} (\pi) \\ \widetilde{f_2}
(0) \\ \widetilde{f_2} (\pi)
\end{pmatrix} = \begin{pmatrix}
\widehat{f_1} (0) \\ e^{is_1 (\pi)} \widehat{f_1} (\pi)
\\ \widehat{f_2} (0) \\ e^{-is_2 (\pi)} \widehat{f_2} (\pi)
\end{pmatrix},
$$
where $\mathcal{M}_{\widetilde{bc}} $ is the transition  matrix (\ref{c17})
corresponding to the boundary conditions $\widetilde{bc}.$
Since
$$
\mathcal{M}_{\widetilde{bc}} =
 \left [ \begin{array}{cccc}
1 &  -\tilde{b}  &  -\tilde{a}  &  0 \\  \frac{-\tilde{c}}{\tilde{b}\tilde{c}-\tilde{a}\tilde{d}}
& 1 &  0  & \frac{\tilde{a}}{\tilde{b}\tilde{c}-\tilde{a}\tilde{d}}\\
\frac{\tilde{d}}{\tilde{b}\tilde{c}-\tilde{a}\tilde{d}}
&  0  &  1  &  \frac{-\tilde{b}}{\tilde{b}\tilde{c}-\tilde{a}\tilde{d}}
\\ 0  &  -\tilde{d}  &  -\tilde{c}  &  1
  \end{array}  \right ],
$$ an easy calculation (which is omitted) shows that the formulas
(\ref{g30}) and (\ref{g31}) hold with $x_1 =0, \; x_2 =\pi. $ This
completes the proof.
\end{proof}

\section{Self-adjoint separated boundary conditions}

A boundary condition $bc $ given by a matrix $\begin{pmatrix}
1 & b& a & 0 \\ 0  &  d  &  c  & 1   \end{pmatrix} $ is called { \em separated}
if $b=c = 0;$  such $bc$ has the form
\begin{equation}
\label{s1}
y_1 (x_1) + a y_2 (x_1 ) = 0, \quad       d y_1 (x_2)  +  y_2 (x_2 ) = 0.
\end{equation}
In the case
\begin{equation}
\label{s2} a=e^{2i\alpha_1}, \quad  d=e^{-2i\alpha_2}, \quad
\alpha_1, \alpha_2 \in [0, \pi),
\end{equation}
we have a {\em self-adjoint separated} $bc$ which could be written in the
form
\begin{equation}
\label{s3} e^{-i\alpha_1}y_1 (x_1) + e^{i\alpha_1} y_2 (x_1 ) = 0,
\quad
    e^{-i\alpha_2} y_1 (x_2)  + e^{i\alpha_2} y_2 (x_2 ) = 0.
\end{equation}
In view of Corollary \ref{cor22}, (\ref{s3})  gives the
general form of  self-adjoint separated boundary conditions.

\begin{Theorem}
\label{thms1} Consider  on $[x_1, x_2]$ the Dirac operator
\begin{equation}
\label{s4} L_{bc}(D)y=\begin{pmatrix}  i & 0 \\ 0 & -i  \end{pmatrix}
\frac{dy}{dx}+ Dy, \quad D=
\begin{pmatrix}  A_1+iA_2   &  P_1 + i P_2
\\ P_1- i P_2  & A_1 -iA_2    \end{pmatrix} ,
\end{equation}
where $A_1, A_2, P_1, P_2 $ are real $L^2  $-functions, and  $bc$
is given by (\ref{s3}).
\vspace{1mm}

(a)  The spectrum of $L_{bc}(D)$ is discrete;
 each eigenvalue is real and has equal
geometric and algebraic multiplicities.
 Moreover, there are numbers  $N=N(D,bc)\in \mathbb{N}  $ and    $\tau= \tau (D,bc) \in \mathbb{R} $
such that,
with $\ell= x_2 -x_1, $
the interval
 $\left ((\tau - N -\frac{1}{4})\frac{\pi}{\ell}, (\tau + N +\frac{1}{4})\frac{\pi}{\ell} \right)$
 contains exactly $2N+1$ eigenvalues
 (counted with multiplicity),
and  for $n\in \mathbb{Z} $ with $  |n|>N $ there is exactly one (simple!)
eigenvalue  $\lambda_n \in \left( (\tau + n
-\frac{1}{4})\frac{\pi}{\ell},
  (\tau + n +\frac{1}{4})\frac{\pi}{\ell} \right ).$
\vspace{1mm}

(b) There is a Riesz basis in $L^2 ([x_1,x_2], \mathbb{C}^2)$ which
elements are eigenfunctions of the operator $(L_{bc}(D)) $  of the
form
\begin{equation}
\label{s6} \Phi = \left \{\begin{pmatrix}  \varphi_k \\
\overline{\varphi}_k  \end{pmatrix} =
\begin{pmatrix}  u_k + i v_k \\ u_k - i v_k    \end{pmatrix},
 \;\; u_k, v_k \in L^2
([x_1,x_2], \mathbb{R}), \; k \in \mathbb{Z}    \right  \},
\end{equation}
and its adjoint biorthogonal system has the form
\begin{equation}
\label{s6a}
\Psi = \left \{\begin{pmatrix}  \psi_k \\
\overline{\psi}_k  \end{pmatrix} =
\begin{pmatrix}  a_k + i b_k \\ a_k - i b_k    \end{pmatrix},
\;\; a_k,  b_k  \in L^2 ([x_1,x_2], \mathbb{R}), \; k \in \mathbb{Z}  \right  \}.
\end{equation}

(c) If $F =\begin{pmatrix}  F_1 \\ F_2 \end{pmatrix}$ is a function
of bounded variation on $[x_1, x_2],$ then its expansion about the
basis $\Phi $ converges point-wise to a function $\tilde{F} (x), $
\begin{equation}
\label{s31} \tilde{F} (x) = \begin{pmatrix}  \tilde{F}_1 \\
\tilde{F}_2
\end{pmatrix} (x) =\lim_{M\to \infty } \sum_{|k|\leq M} c_k
(F)\begin{pmatrix}  u_k + i v_k \\ u_k - i v_k   \end{pmatrix} (x),
\end{equation}
where
\begin{equation}
\label{s32} c_k (F) = \left \langle  \begin{pmatrix}  F_1 \\ F_2
\end{pmatrix}, \begin{pmatrix}  a_k + i b_k \\ a_k - i b_k    \end{pmatrix}
  \right \rangle,\quad k \in \mathbb{Z},
\end{equation}
and
\begin{equation}
\label{s33}
 \tilde{F} (x)  = \frac{1}{2}
(F(x-0) + F(x+0) ), \quad x_1 < x < x_2,
\end{equation}
\begin{equation}
\label{s34} \tilde{F} (x_1) = \begin{pmatrix}  F_1 (x_1 +0) -
\exp(2i\alpha_1) F_2 (x_1 +0) \\ - \exp(-2i\alpha_1)F_1 (x_1 +0) +
F_2 (x_1 +0)
\end{pmatrix},
\end{equation}
\begin{equation}
\label{s35} \tilde{F} (x_2) = \begin{pmatrix}  F_1 (x_2 -0) -
\exp(2i\alpha_2) F_2 (x_2 -0) \\ - \exp(-2i\alpha_2) F_1 (x_2 -0) +
F_2 (x_2 -0)
\end{pmatrix}.
\end{equation}

Moreover, if $F_1 (x_2 -x)$ and $F_2(x)$ are  continuous on
some closed subinterval of $(x_1,x_2 )$ then the convergence in
(\ref{s31})  is uniform on that interval.
The convergence is uniform on the
closed interval $[x_1, x_2]$ if and only if $F_1$  and $F_2$ are continuous
on $[0, \pi]$ and $\begin{pmatrix}  F_1 \\  F_2  \end{pmatrix} $
satisfies the boundary condition $bc$ given by (\ref{s3}).

\end{Theorem}

\begin{proof}
(a)  Set
$$
s_1 (x) = \int_{x_1}^x (A_1 (\xi) +i A_2 (\xi)) d\xi,
\quad  s_2 (x) = \int_{x_1}^x (A_1 (\xi) - i A_2 (\xi)) d\xi;
$$
then $\overline{s_1 (x)}= s_2 (x),  $
so the sum $s_1 (x) + s_2 (x) $ is real-valued.
As in Proposition~\ref{propT}
(see (\ref{g3})--(\ref{g9})), one can easily see that
the operator $L_{bc}(D) $ is similar to
the Dirac operator $L_{\widetilde{bc}}(v), $  with $$
v=\begin{pmatrix} 0 & (P_1+iP_2)e^{-i(s_1 (x) + s_2 (x))}\\
(P_1-iP_2)e^{i(s_1 (x) + s_2 (x))} &  0
\end{pmatrix},
$$
$$ \widetilde{bc}: \quad \tilde{a}= a = e^{2i\alpha_1}, \;\;
\tilde{b}=b=0, \quad \tilde{c} = 0, \quad \tilde{d} =
e^{-2i\alpha_2}e^{i(s_1 (x_2) + s_2 (x_2))},$$
and
\begin{equation}
\label{s51}
M L_{bc} (D) = L_{\tilde{bc}} (v) M, \quad
 M =
\begin{pmatrix} e^{-is_{1}(x)}  &  0  \\  0  & e^{is_{2}(x)}
\end{pmatrix} .
\end{equation}

The matrix $v$ is hermitian because the sum $ s_1 (x) + s_2 (x)  $ is real-valued. Since
$\widetilde{bc}$  is a self-adjoint  boundary condition of the form (\ref{s3}),
it follows that  the operator
$L_{\widetilde{bc}}(v) $ is self-adjoint.
Therefore,  its spectrum is real, and moreover,  discrete by
Part (A) of Theorem~\ref{GGCT}.

In the case $x_1=0, \, x_2 = \pi$ a localization of the spectrum of
 $L_{\widetilde{bc}} (v)$ can be obtained by the general
scheme from Section 2 and Lemma~\ref{srl}. Indeed, now the
characteristic equation (\ref{13}) becomes $$z^2 = \tilde{a}
\tilde{d}= e^{2i(\alpha_1 -\alpha_2)} e^{i(s_1 (\pi) + s_2 (\pi))},
$$ so its solutions $z_1, \,z_2 $ can be written as
$$ z_1 = e^{i\pi \tau}, \quad z_2= e^{i\pi (\tau +1)}, \quad
\text{where} \;\; \tau = \frac{1}{2\pi}[2\alpha_1 -2\alpha_2 + s_1
(\pi) + s_2 (\pi)]. $$ Therefore, it follows that $Sp (
L^0_{\widetilde{bc}}) = \{\lambda = \tau + n, \; n \in \mathbb{Z}
\},$ so Lemma~\ref{srl} implies (a) in this case.

The general case of an arbitrary interval $[x_1,x_2]$ could be
reduced to the case of $[0,\pi] $ by the change of variable $x= x_1 +
\frac{\ell}{\pi} \, t, \; \ell=x_2 -x_1. $ \vspace{3mm}

 (b)   In view of Part (B) of Theorem \ref{GGCT} there is a
 Riesz basis in $L^2 ([x_1,x_2], \mathbb{C}^2)$
which consists of eigenfunctions of  $L_{\widetilde{bc}}(v) .$
 Moreover,  since  it is a self-adjoint operator
 there is an orthonormal basis which consists of eigenfunctions of
 $L_{\widetilde{bc}}(v). $

 One can easily see that
 the real vector subspace of $L^2([x_1,x_2],\mathbb{C}^2)$
$$ H = \left \{ \begin{pmatrix} \varphi \\ \overline{\varphi}
\end{pmatrix} \; : \;\; \varphi \in L^2([x_1,x_2],\mathbb{C})
\right \} $$
is invariant subspace for both $L_{bc} (D)$ and
$L_{\widetilde{bc}}(v). $
Moreover,  since
$$
\begin{pmatrix} g_1\\ g_2   \end{pmatrix}
= \begin{pmatrix} \varphi \\ \overline{\varphi}   \end{pmatrix}
+i\begin{pmatrix} \psi \\ \overline{\psi}   \end{pmatrix} \quad
\text{with} \;\; \varphi= \frac{g_1+\overline{g_2}}{2} \;\;
\text{and} \;\; \psi= \frac{g_1-\overline{g_2}}{2i}, $$
we have $ \; L^2([x_1,x_2],\mathbb{C}^2) = H \oplus iH .
$

Suppose $ \lambda \in  Sp \, (L_{\widetilde{bc}}(v)),$ and let $ E_\lambda = \{y: \; L_{\widetilde{bc}}(v) y= \lambda y\} $
be the space of eigenvectors corresponding to
$\lambda. $ By (a) we know that $ \lambda $ is real, and $\dim
E_\lambda < \infty. $ Since $\lambda $ is real, one can easily see by
taking the conjugates that
\begin{equation}
\label{s14}
\begin{pmatrix} g_1\\ g_2   \end{pmatrix} \in
E_{\lambda}  \quad \Rightarrow  \quad \begin{pmatrix}
\overline{g}_2\\ \overline{g}_1
\end{pmatrix}  \in E_{\lambda}.
\end{equation}
Suppose that
$
\begin{pmatrix} g_1\\ g_2   \end{pmatrix}
= \begin{pmatrix} \varphi \\ \overline{\varphi}   \end{pmatrix}
+i\begin{pmatrix} \psi \\ \overline{\psi}   \end{pmatrix} \in E_{\lambda}.
$
Then (\ref{s14}) implies
$
\begin{pmatrix} \overline{g}_2\\ \overline{g}_1   \end{pmatrix}
= \begin{pmatrix} \varphi \\ \overline{\varphi}   \end{pmatrix}
-i\begin{pmatrix} \psi \\ \overline{\psi}   \end{pmatrix} \in E_{\lambda},
$
which yields $ \begin{pmatrix} \varphi \\ \overline{\varphi}
\end{pmatrix} =\frac{1}{2} \begin{pmatrix} g_1\\ g_2   \end{pmatrix} +
\frac{1}{2} \begin{pmatrix} \overline{g}_2\\ \overline{g}_1
\end{pmatrix}\in E_{\lambda}$ and $\begin{pmatrix} \psi \\
\overline{\psi}   \end{pmatrix}= \frac{1}{2i} \begin{pmatrix} g_1\\
g_2 \end{pmatrix} - \frac{1}{2i} \begin{pmatrix} \overline{g}_2\\
\overline{g}_1 \end{pmatrix}\in E_{\lambda}.$ Hence,
$$
E_\lambda = (E_\lambda \cap H) \oplus i(E_\lambda \cap H),
$$
which implies that every basis in $(E_\lambda \cap H)$ (regarded as a
real vector space) is a basis in $E_\lambda $ (regarded as a complex
vector space) as well.
Therefore, one may choose in each of the spaces $E_\lambda $
a basis consisting of mutually orthogonal normalized
vectors from $E_\lambda \cap H.$
Since all but finitely many of these spaces are one-dimensional,
it follows that there is an orthonormal basis
$f_k, \; k \in \mathbb{Z}, $
of the form (\ref{s6})
which consists of eigenfunctions of $L_{\widetilde{bc}}(v).$

By (\ref{s51}), the system
$$
\Phi = \{ M f_k, \; k \in \mathbb{Z} \}
$$
is a Riesz basis
in $L^2 ([x_1,x_2], \mathbb{C}^2)$
which consists of eigenfunctions
of  $L_{bc} (D),$ and the corresponding biorthogonal system is
$$
\Psi = \{ (M^{-1})^*   f_k, \; k \in \mathbb{Z} \}, \quad
(A^{-1})^* = \begin{pmatrix}  e^{-i\overline{s_1}} & 0\\ 0&
e^{i\overline{s_2}}.
\end{pmatrix}
$$
Since $\overline{s_1 (x)}= s_2 (x),$
one can easily verify that the system $\Phi$
has the form (\ref{s6})  and
$\Psi$ has the form (\ref{s6a}).

Finally, in view of (\ref{s2}),  (c) follows from part
(C) of Theorem~\ref{GGCT}.

\end{proof}

Next we provide a  version of Theorem \ref{thms1}
for real-valued functions.

\begin{Theorem}
\label{SBC}
Let $\rho \in L^1 ([x_1, x_2]), $  $ \rho (x) \geq
const >0 $ for $x \in [x_1, x_2],$ and
 \begin{equation}
\label{s21} T= \begin{pmatrix}  T_{11}  &  T_{12} \\ T_{21}  &
T_{22}      \end{pmatrix}, \quad \frac{1}{\rho} T_{ij} \in L^2 ([x_1,
x_2],\rho), \quad T_{ij} -\text{real-valued}.
\end{equation}
  Consider the operator
 \begin{equation}
\label{s22} R_{\widehat{bc}} (T) \begin{pmatrix}  u \\ v
\end{pmatrix} :=  \frac{1}{\rho} \left [\begin{pmatrix} 0& -1\\ 1 & 0 \end{pmatrix}
\frac{d}{dx}\begin{pmatrix}  u \\ v      \end{pmatrix} +
T\begin{pmatrix}  u \\ v      \end{pmatrix}
\right ],
\end{equation}
subject to the boundary conditions
 \begin{equation}
\label{s23} \widehat{bc}: \quad u(x_j) \cos \alpha_j +  v(x_j) \sin
\alpha_j  = 0,    \quad \alpha_1 \neq \alpha_2,
 \quad  j=1,2.
\end{equation}

(A).  The spectrum of the operator $R_{\widehat{bc}} $
is discrete;  each eigenvalue is real and has equal
geometric and algebraic multiplicities.
 Moreover, there are numbers $N=N(T,bc),   $   $\tau= \tau (T,bc)$
and  $\ell= \ell (\rho, x_2 -x_1) $
 such that the interval
 $((\tau - N -\frac{1}{4})\frac{\pi}{\ell},
 (\tau + N +\frac{1}{4})\frac{\pi}{\ell})$
 contains
 $2N+1$ eigenvalues $\lambda_k, \;  -N \leq k \leq N $
 (counted with multiplicity),
and  for $n\in \mathbb{Z}, \;|n|>N, $ there is only one (simple!)
eigenvalue $\lambda_n $
  in the interval  $((\tau + n -\frac{1}{4})
  \frac{\pi}{\ell}, (\tau + n +\frac{1}{4})\frac{\pi}{\ell}).$

(B).
Let $ \mathcal{B}=\left \{
\begin{pmatrix}  u_n \\ v_n    \end{pmatrix},
 \; n \in \mathbb{Z}    \right  \}$
 be a system of normalized real-valued eigenfunctions corresponding
 to the sequence of eigenvalues  $(\lambda_n)_{n \in \mathbb{Z}}.$
 Then the system $\mathcal{B}$ is a Riesz basis
 in $\left (L^2([x_1,x_2],\rho) \right )^2,$  i.e.,
 \begin{equation}
\label{s25}
\begin{pmatrix}  f\\g \end{pmatrix}=
\sum_{n \in \mathbb{Z}}  C_n (f,g)
\begin{pmatrix}  u_n \\  v_n \end{pmatrix}
\qquad \forall  \begin{pmatrix}  f\\g \end{pmatrix} \in
\left (L^2([x_1,x_2],\rho) \right )^2,
\end{equation}
where the series converge unconditionally.

(C). If $f$ and $g$ are real-valued functions
 of bounded variation on $[x_1, x_2],$
  then
 the series in (\ref{s25}) converges point-wise in the sense that
 \begin{equation}
\label{s26} \lim_{M\to \infty}  \sum_{|n| \leq M}
C_n (f,g)
\begin{pmatrix}  u_n(x)\\v_n(x) \end{pmatrix}
:= \begin{pmatrix}    \tilde{f} (x)  \\ \tilde{g} (x)
\end{pmatrix}
\end{equation}
 where
\begin{equation}
\label{s27} \begin{pmatrix}    \tilde{f} (x)  \\ \tilde{g} (x)
\end{pmatrix} = \frac{1}{2} \left [
\begin{pmatrix}    f (x-0)  \\ g (x-0)
\end{pmatrix}+\begin{pmatrix}    f (x+0)  \\ g (x+0)
\end{pmatrix} \right ] \quad \text{for} \;\; x\in
(x_1,x_2),
\end{equation}
and
\begin{equation}
\label{s28}  \begin{pmatrix}
\tilde{f} (x_1) \\  \tilde{g} (x_1)
\end{pmatrix}
 = \frac{1}{2}
\begin{pmatrix}    f (x_1+0) (1-\cos 2\alpha_1) - g(x_1+0) \sin 2\alpha_1  \\
- f(x_1+0) \sin 2\alpha_1+  g (x_1+0) (1+\cos 2\alpha_1)
\end{pmatrix},
\end{equation}
\begin{equation}
\label{s29}  \begin{pmatrix}
\tilde{f} (x_2) \\  \tilde{g} (x_2)
\end{pmatrix}
 = \frac{1}{2}
\begin{pmatrix}    f (x_2-0) (1-\cos 2\alpha_2) - g(x_2-0) \sin 2\alpha_2  \\
- f(x_2-0) \sin 2\alpha_2+  g (x_2-0) (1+\cos 2\alpha_2)
\end{pmatrix}.
\end{equation}

In addition, if the functions $f$ and $g$ are continuous on a
closed subinterval  $ [x_1 +\delta, x_2 - \delta] \subset (x_1,x_2),$
then the convergence in (\ref{s26}) is uniform on
$ [x_1 +\delta, x_2 - \delta].$   The convergence is uniform on
the closed interval $[x_1, x_2]$ if and only if $f$ and $g$ are continuous on $[x_1, x_2]$  and
$\begin{pmatrix} f  \\ g  \end{pmatrix}$ satisfies the boundary condition
 $\widehat{bc}$ given in (\ref{s23}).
\end{Theorem}

\begin{proof}
(A)  In view  Lemma \ref{lemg1},  we may assume that
$\rho \equiv 1. $
Set
$$
A_1 =\frac{1}{2} (T_{11} + T_{22}), \;  P_1 =\frac{1}{2} (T_{11} - T_{22}), \;
A_2 =\frac{1}{2} (T_{21} - T_{12}), \; P_2 =\frac{1}{2} (T_{21} + T_{12})
$$
and consider the operator $L_{bc} (D) $ in (\ref{s4}) with  $D=
\begin{pmatrix}  A_1+iA_2   &  P_1 + i P_2
\\ P_1- i P_2  & A_1 -iA_2    \end{pmatrix}$
and $bc$ given by (\ref{s3}).

 A simple calculation shows that
 $ \begin{pmatrix}  u \\ v    \end{pmatrix}$ satifies (\ref{s23})
 if and only if $\begin{pmatrix}  u + i v \\ u - i v   \end{pmatrix}$
 satisfies (\ref{s3}),
 and
 \begin{equation}
  \label{s101}
    L_{bc} (D) \begin{pmatrix}  u + i v \\ u - i v   \end{pmatrix}
  = \lambda  \begin{pmatrix}  u + i v \\ u - i v   \end{pmatrix}
\; \Leftrightarrow \;
R_{\widehat{bc}}(T) \begin{pmatrix}  u \\ v    \end{pmatrix}= \lambda
 \begin{pmatrix}  u \\ v    \end{pmatrix}.
\end{equation}
Therefore, $\lambda $ is an eigenvalue of $R_{\widehat{bc}}(T)$
if and only if it is an eigenvalue of $L_{bc} (D)$
with the same geometric multiplicity, and
 the localization of eigenvalues of $R_{\widehat{bc}}(T)$
given in (A) follows from part (a) of Theorem~\ref{thms1}.
But it remains to explain that the spectrum of $R_{\widehat{bc}}(T)$
is discrete -- see below.
\vspace{2mm}

(B)  By Theorem \ref{s1}, the system $\Phi $ in (\ref{s6}) is a Riesz basis in
$L^2([x_1,x_2],\mathbb{C}^2)$ consisting of eigenfunctions of the
operator $L_{bc} (D),$ and its  biorthogonal system is given by (\ref{s6a}).  Fix $\begin{pmatrix}  f  \\ g\end{pmatrix} \in
L^2 ([x_1, x_2], \mathbb{R}^2)$  and consider the expansion of
$\begin{pmatrix}  f +i g \\f- i g  \end{pmatrix}$ about the Riesz basis
$\Phi. $  Then
\begin{equation}
\label{s102}
\begin{pmatrix}  f +i g \\f- i g  \end{pmatrix}= \sum_k C_k (f,g)
\begin{pmatrix}  u_k +i v_k \\u_k- i v_k  \end{pmatrix},
\end{equation}
where the series  converges in  unconditionally in
$L^2([x_1,x_2],\mathbb{C}^2)$ and
\begin{equation}
\label{s103}
C(f,g) = \left \langle
\begin{pmatrix}  f +i g \\f- i g  \end{pmatrix},
\begin{pmatrix}  a_k +i b_k \\a_k- i b_k  \end{pmatrix}
\right \rangle =
\left \langle
\begin{pmatrix}  f  \\ g  \end{pmatrix},
2\begin{pmatrix}  a_k \\ b_k  \end{pmatrix}
\right \rangle.
\end{equation}
By taking the first components in (\ref{s102}) and separating the real and imaginary parts we obtain
\begin{equation}
\label{s104}
\begin{pmatrix}  f \\ g  \end{pmatrix}= \sum_k C_k (f,g)
\begin{pmatrix}  u_k  \\ v_k  \end{pmatrix} \quad
\forall \begin{pmatrix}  f \\ g  \end{pmatrix} \in L^2([x_1, x_2],  \mathbb{R}^2),
\end{equation}
where the series  converge  unconditionally in
$L^2([x_1,x_2],\mathbb{R}^2).$

In view of (\ref{s101}), (\ref{s103}) and (\ref{s104}),
the system
 \begin{equation}
 \label{s106}
\mathcal{B} =\left \{
\begin{pmatrix}  u_k \\ v_k    \end{pmatrix}:\;\;
\begin{pmatrix}  u_k + i v_k \\ u_k - i v_k    \end{pmatrix}
\in \Phi,
 \quad k \in \mathbb{Z}    \right  \}
\end{equation}
is a Riesz basis in $L^2([x_1,x_2],\mathbb{R}^2)$
(and therefore, in $L^2([x_1,x_2],\mathbb{C}^2)$)
which consists of eigenfunctions of
the operator $R_{\widehat{bc}}(T),$
and its
biorthogonal system is
\begin{equation}
\label{s107} \mathcal{B}^* = \left \{2\begin{pmatrix}  a_k \\ b_k
\end{pmatrix}:
\begin{pmatrix}  a_k + i b_k \\ a_k - i b_k
\end{pmatrix}\in \Psi,
\quad k \in \mathbb{Z}  \right  \}.
\end{equation}
Now one can use the Riesz basis (\ref{s106})
in order to construct the resolvent
$(\lambda -R_{\widehat{bc}}(T))^{-1}$
for any $\lambda \neq \lambda_k, \; k \in \mathbb{Z}, $
which shows that the spectrum of $R_{\widehat{bc}}(T)$
is discrete and consists of the eigenvalues $\lambda_k, \; k \in \mathbb{Z}.$ \vspace{2mm}

(C)  Let $f$ and $g$ be functions of bounded variation
on $[x_1, x_2].$
Set
$$
F= \begin{pmatrix}  F_1  \\ F_2 \end{pmatrix} =
\begin{pmatrix}  f +i g \\f- i g\end{pmatrix}; \quad \text{then}
 \;\; f=Re \, F_1, \quad g = Im F_1
$$
In view of Part (c) of Theorem~\ref{thms1}, we have
\begin{equation}
 \label{s110}
\lim_{M\to \infty}  \sum_{|n| \leq M}
C_n (f,g)
\begin{pmatrix}  u_n(x) +i v_n (x) \\ u_n (x) - i v_n(x) \end{pmatrix}
:= \tilde{F} (x),
\end{equation}
where
$\tilde{F} (x) =
\begin{pmatrix}    \tilde{F}_1 (x)  \\ \tilde{F}_2 (x)
\end{pmatrix}$
is given by Formulas (\ref{s33})--(\ref{s35})
in terms of $F.$  Obviously, we have $\tilde{F}_2 (x)=
\overline{\tilde{F}_1 (x)}.$  Taking the first components in (\ref{s110})
and separating the real and imaginary parts we obtain
\begin{equation}
 \label{s111}
\lim_{M\to \infty}  \sum_{|n| \leq M}
C_n (f,g)
\begin{pmatrix}  u_n(x)  \\   v_n(x) \end{pmatrix}
:= \begin{pmatrix}    \tilde{f}(x)  \\ \tilde{g}(x)
\end{pmatrix}
\end{equation}
with $\tilde{f}(x)= Re \, \tilde{F}_1 (x)$  and
$\tilde{g}(x)= Im \, \tilde{F}_1 (x).$
Now (C) follows immediately from Part (c) of Theorem~\ref{thms1}.
\end{proof}

We are thankful to  R. Szmytkowski for bringing our
attention to the point-wise convergence problem of spectral
decompositions of 1D Dirac operators. In the case of self-adjoint separated
boundary conditions, our point-wise convergence results
(see (\ref{s28}) and (\ref{s29}))
confirm the
formula suggested by
 R. Szmytkowski (\cite[Formula 3.14]{SZ01}).

\section{Appendix: Discrete Hilbert transform and multipliers}

The aim of this Appendix is to prove Lemma \ref{lemsob}. In fact, we
explain that if $f \in H(\Omega) $ and $g \in C^1 ([0,\pi]), $ then
$f \cdot g \in H(\Omega)$ for a wider class of weights then we need
in Lemma \ref{lemsob} -- see below Proposition \ref{propA}.

Recall that if $ \xi=(\xi_k) \in \ell^2 (\mathbb{Z})$ then its
discrete Hilbert transform is defined by $$(\mathcal{H} \xi)_n
=\sum_{k\neq n} \frac{\xi_k}{n-k},\quad n \in \mathbb{Z}. $$ It is
well known that  the operator $\mathcal{H}: \ell^2 (\mathbb{Z}) \to
\ell^2 (\mathbb{Z}) $ is bounded. Moreover, let $\Omega = (\Omega
(k))_{k\in\mathbb{Z}} $ be a weight sequence such that
\begin{equation}\label{a0}
\Omega (0) \geq 1, \quad \Omega (-k) = \Omega (k), \quad \Omega
(k)\leq \Omega (k+1) \; \text{for} \; k \geq 0.
\end{equation}
Then it is known by \cite[Theorem 10]{HMW} that the discrete Hilbert transform
$\mathcal{H} $ acts continuously in the weighted space
$$\ell^2 (\Omega) =\{\xi:  \;  \|\xi\|_\Omega^2 =
 \sum  |\xi_k|^2  \Omega^2 (k) <\infty
    \} $$ if and only if the weight $\Omega $ satisfies
the condition
\begin{equation}
\label{a1} \sup_{k,n} \left ( \frac{1}{n+1} \sum_{m=k}^{k+n}
\Omega^2 (m)\times \frac{1}{n+1} \sum_{m=k}^{k+n} \frac{1}{\Omega^2
(m)} \right ) < \infty.
\end{equation}
See \cite{AM2} for the proof of a particular version of this
criterion which is good enough for Lemma \ref{lemA1}.

\begin{Lemma} \label{lemA1} Suppose the weight $\Omega $ satisfies
(\ref{a0}) and
\begin{equation}
\label{a2} \exists C>0 : \quad  \Omega (2k) \leq C \,\Omega (k) \quad
\forall \, k \in 2\mathbb{Z};
\end{equation}
\begin{equation}
\label{a3} \exists C>0 : \quad   \Omega (k) \leq C \sqrt{1+|k|} \quad
\forall \, k \in 2\mathbb{Z}.
\end{equation}
If $f \in H(\Omega)$  and $g \in H^1_{per},$ then $f\cdot g \in
H(\Omega).$
\end{Lemma}

\begin{proof}
Recall that $H^1_{per}= H(\sqrt{1+k^2}).$ Let $(\hat{f}(k))$ and
$(\hat{g}(k))$ be the Fourier coefficients of $f$ and $g$ with
respect to the system $e^{ikx}, \; k \in 2\mathbb{Z}.$ It is enough
to show that $\hat{f} \ast \hat{g} \in \ell^2 (\Omega). $ To this
end we consider, for $b \in \left ( \ell^2 (\Omega) \right )^\ast =
\ell^2 (\Omega^{-1}),$ the ternary form $$ T = \sum_m \sum_k
\hat{f}(k)\hat{g}(m-k) b(m)
$$
and show that it is bounded.

Set
$$
\xi (k) =\hat{f}(k)\Omega (k), \quad \eta (k) = \hat{g}(k)
\sqrt{1+k^2}, \quad \beta (k) = b(k)/\Omega (k), \quad k \in
2\mathbb{Z}.
$$
Then we have $\xi, \eta, \beta \in \ell^2 (2 \mathbb{Z})$ and
$$
\|\xi\| = \|\hat{f}\|_{\ell^2 (\Omega)}, \quad
\|\eta\|=\|\hat{g}\|_{\ell^2 (\sqrt{1+k^2})}, \quad  \|\beta\|=
\|b\|_{\ell^2 (\Omega^{-1})}.
$$
Now the ternary form $T$ can be written as
$$
T = \sum_{k,m} \frac{\xi (k)}{\Omega (k)} \cdot \frac{\eta
(m-k)}{\sqrt{1+(m-k)^2}} \cdot \beta (m) \Omega (m),
$$
and the Cauchy inequality implies
$$
|T|^2 \leq \left (\sum_{k,m} |\xi (k)|^2 |\eta (m-k)|^2   \right )
\left ( \sum_{k,m}  \Omega^2 (m) \frac{|\beta (m)|^2}{\Omega^2 (k)
[1+(m-k)^2]} \right )
$$
$$
\leq S \|\xi\|^2 \|\eta\|^2 \|\beta \|^2,
$$
where
$$
S= \sup_m   \sum_k \frac{\Omega^2 (m)}{\Omega^2 (k) [1+(m-k)^2]}.
$$
Next we explain that $S<\infty.$ Indeed, in view of (\ref{a2}), if
$|k| \geq |m|/2 $ then $\Omega (m) \leq \Omega (2k) \leq C \Omega
(k). $ Therefore,
$$
\sum_{|k|\geq |m|/2} \frac{\Omega^2 (m)}{\Omega^2 (k) [1+(m-k)^2]}
\leq C^2 \sum_{|k|\geq |m|/2} \frac{1}{[1+(m-k)^2]} \leq C^2 (1+\pi).
$$
On the other hand, if $|k| <|m|/2 $ then $|m-k| >|m|/2.$  Thus,
$$
\sum_{|k|< |m|/2} \frac{\Omega^2 (m)}{\Omega^2 (k) [1+(m-k)^2]} \leq
\sum_{|j|>|m|/2} \frac{\Omega^2 (m)}{1+j^2} \leq \frac{4\Omega^2
(m)}{1+|m|}.
$$
Now  (\ref{a3}) implies that
$$
S \leq \sup_m \left (  C^2 (1+\pi) + \frac{4 \Omega^2 (m)}{1+|m|}
\right ) <\infty,
$$
which completes the proof.

\end{proof}

\begin{Proposition}
\label{propA} If a weight sequence $\Omega$ satisfies (\ref{a1}) -
(\ref{a3}), $ f \in H(\Omega)$ and  $g \in C^1 ([0,\pi]),$  then $f
\cdot g \in H(\Omega).$
\end{Proposition}

\begin{proof}
The $C^1$-function $g$ could be written as a sum of a linear function
and a periodic $C^1$-function as
$$
g(x) = \ell (x) + g_1 (x) \quad \text{with} \quad   \ell (x) = m \,
x, \; \; m =(g(\pi) - g(0))/\pi.
$$
Since $g_1 \in H_{per}, $  Lemma \ref{lemA1} implies that $f\cdot
g_1 \in H(\Omega). $  So, it remains to prove that $x \, f(x) \in
H(\Omega).$

Since $$ x= \sum_{k \in 2\mathbb{Z}} c(k) e^{ikx}, \quad \text{with}
\quad c(0) = \pi/2,  \quad c(k) = \frac{i}{k} \; \;\text{for} \; \;
k \neq 0,
$$
we obtain that $$c \ast \hat{f} (k) = \frac{\pi}{2}\hat{f} (k) +
i\sum_{j\neq k} \frac{\hat{f} (j)}{k-j}, \quad \text{that is} \quad
c \ast \hat{f} = \frac{\pi}{2} \hat{f} + \mathcal{H} (\hat{f}).
$$
Since (\ref{a1}) holds, $\mathcal{H} (\hat{f}) \in \ell^2 (\Omega)$
(due to the results of \cite{HMW}), so $c \ast \hat{f} \in \ell^2
(\Omega).$ Thus, $x \, f(x) \in H(\Omega),$ which completes the
proof.

\end{proof}

Proposition \ref{propA} would imply Lemma \ref{lemsob} if we check
that the conditions (\ref{a1}) - (\ref{a3}) for the weights $\Omega
$ given by
\begin{equation}
\label{a11} \Omega (k) = (1+|k|)^\alpha,    \quad 0 \leq \alpha <
1/2,
\end{equation}
or
\begin{equation}
\label{a12} \Omega (k) = (\log (e+|k|))^\delta,   \quad 0 \leq
\delta < \infty.
\end{equation}

\begin{Lemma}
\label{lemA2} Weights $\Omega $ given by (\ref{a11}) and (\ref{a12})
satisfy the conditions (\ref{a1}) - (\ref{a3}).
\end{Lemma}

\begin{proof}
Elementary inequalities show that (\ref{a2}) and (\ref{a3}) hold. To
check (\ref{a1}) we have to show that there is $M>0 $ such that
\begin{equation}
\label{a13} \frac{1}{n+1} s(k,n) \times \frac{1}{n+1} S(k,n) \leq M
\quad \forall k \in \mathbb{Z}, \; n \in \mathbb{Z}_{+},
\end{equation}
where
$$
s(k,n) = \sum_{m=k}^{k+n} \frac{1}{\Omega^2 (m)}, \quad S(k,n) =
\sum_{m=k}^{k+n} \Omega^2 (m).
$$
There are three cases:
$$
(a) \; \; k < -2n;  \qquad   (b) \;\;  -2n \leq k  \leq n;  \qquad
(c) \;\; k>n.
$$
By (\ref{a0}), $\Omega (-m) = \Omega (m);$  therefore, with $k_1 =
-k- n $ we have
$$
s(k,n)= s(k_1,n), \quad   S(k,n)= S(k_1,n),
$$
so the case (a) reduces to (c).  If $k>n, $ then by (\ref{a0}) and
(\ref{a2})
$$
\Omega (k) \leq \Omega (m) \leq \Omega (2k) \leq C \, \Omega (k),
\quad \quad k \leq m \leq k+n.
$$
Therefore, it follows
$$
\frac{1}{n+1} s(k,n) \leq \frac{1}{\Omega (k)} \quad \text{and}
\quad \frac{1}{n+1} S(k,n) \leq \Omega (2k) \leq C \Omega (k),
$$
so the product in (\ref{a13}) does not exceed the constant $C$ from
(\ref{a2}).

Next, we consider the case (b) where $-2n \leq k \leq n.$ Then, by
(\ref{a0}), it follows
\begin{equation}
\label{a26} \frac{s(k,n)}{n+1} =\frac{1}{n+1} \sum_k^{k+n}
\frac{1}{\Omega^2 (m)} \leq  \frac{2}{1+2n} 2\sum_0^{2n}
\frac{1}{\Omega^2 (m)} = 4\frac{s(0,2n)}{2n+1}
\end{equation}
and
\begin{equation}
\label{a27} S(k,n) \leq 2 S(0,2n) \leq 2 (2n+1) \Omega (2n).
\end{equation}

If $\Omega $ is of the form (\ref{a11}), then
$$
s(0,2n) = \sum_{j=0}^{2n} \frac{1}{(1+j)^{2\alpha}} \leq 1 +
\int_0^{2n} \frac{1}{(1+x)^{2\alpha}}dx \leq \frac{2}{1-2\alpha}
(1+2n)^{1-2\alpha},
$$
so (\ref{a26}) and (\ref{a26}) show that the product in (\ref{a13})
does not exceed
$$\left ( \frac{1}{2n+1} \cdot \frac{2}{1-2\alpha}
(2n+1)^{1-2\alpha} \right ) \cdot 4 (2n+1)^{2\alpha} \leq
\frac{8}{1-2\alpha}.
$$

If $\Omega $ is of the form (\ref{a12}), then
$$
s(0,2n) = \sum_{0\leq j \leq \sqrt{n}} \frac{1}{(\log (e+j))^\delta}
+ \sum_{\sqrt{n} < j \leq 2n} \frac{1}{(\log (e+j))^\delta}
$$
$$
\leq \frac{1+\sqrt{n}}{1+2n} +\frac{1+ 2n - \sqrt{n}}{1+2n}
\frac{1}{(\log (e+\sqrt{n}))^\delta} \leq \frac{M}{(\log
(e+\sqrt{n}))^\delta}
$$
with
$$
M = 2 \max \frac{(\log (e+\sqrt{n}))^\delta}{1+\sqrt{n}} +1.
$$
Since (\ref{a27}) holds for every monotone weight $\Omega $ it
follows that the product (\ref{a13}) does not exceed $M\cdot
\tilde{M}^\delta $  with $\tilde{M} = \max_{n\geq 0} \frac{\log
(e+2n)}{\log (e+\sqrt{n})}.$ This completes the proof of Lemma
\ref{lemA2}.
\end{proof}

\end{document}